\documentclass[11pt]{article}

\usepackage{amsmath,amssymb,mathtools,amsthm,mathrsfs}
\usepackage{enumitem}
\usepackage[hidelinks]{hyperref}
\usepackage{verbatim}
\usepackage{placeins}
\usepackage{graphicx}
\usepackage{microtype}
\usepackage{booktabs,longtable}

\usepackage[margin=1.5cm]{geometry}

\setlist[itemize]{leftmargin=*,itemsep=2pt,topsep=4pt}
\setlist[enumerate] {leftmargin=*,itemsep=2pt,topsep=4pt}

\usepackage{caption}
\usepackage{booktabs} 

\renewcommand{\S}{\mathbb{S}}

\newcommand{\R}{\mathbb{R}}
\newcommand{\N}{\mathbb{N}}
\newcommand{\Z}{\mathbb{Z}}

\newtheorem{theorem}{Theorem}[section]
\newtheorem{lemma}[theorem]{Lemma}
\newtheorem{remark}[theorem]{Remark}
\newtheorem{proposition}[theorem]{Proposition}
\newtheorem{corollary}[theorem]{Corollary}
\newtheorem{definition}[theorem]{Definition}

\newtheorem*{theorem*}{Theorem}

\title{On ideal lemniscates, butterflies, and circular waves}
\author{Shinya Okabe\and Glen Wheeler}
\date{\today}

\begin{document}
\maketitle

\begin{abstract}
We prove the existence of infinitely many lemniscate-like, butterfly-like, and circular-wave critical points for the length-penalised ideal energy.
For the lemniscate and butterfly family, we use a staged direct minimisation process to establish existence.
A boundary-layer analysis reveals critical points near two Fresnel phases: $3\pi/4$ modulo $2\pi$, corresponding to lemniscates, and $7\pi/4$ modulo $2\pi$, which are the butterflies.
The family of circular waves bifurcates, in a sense, from multiply-covered circles.
We use an adapted shooting method  to establish their existence.
\end{abstract}

\tableofcontents

\section{Introduction}

\subsection{Historical context}
We study planar curves that are critical (and, in constrained classes, minimising) for the
length-penalised {ideal} functional
\[
J_\lambda[\gamma]
:=\frac12\int_\gamma k_s^2\,ds+\lambda\,L[\gamma],
\qquad \lambda>0,
\]
where $\gamma$ is parametrised by arclength, $k$ is curvature, and $k_s$ the derivative of curvature with respect to arclength.

A useful historical precursor of the ideal energy comes from the Bernoulli-Euler story behind Euler's spiral.
In an inverse elasticity problem posed by James Bernoulli, one asks for the natural shape of a thin lamina that becomes straight when a weight is attached at one end. Once the strip is straightened, the bending moment at arclength distance $s$ from the force is $M=Fs$; assuming the strip does not stretch, the curvature of the unstressed lamina is therefore proportional to arclength, $k(s)=cs$, so the curve is Euler's spiral \cite{Levien2008EulerSpiralHistory,Levien2009FromSpiralToSpline}. In modern language, Euler spirals are precisely curves with constant non-zero $k_s$, so this classical mechanical thought experiment points naturally toward energies that penalise curvature variation. 

The ideal energy appeared next in  geometric design.
We refer the interested reader to Section \ref{sec:related}, where we give a brief survey.

\subsection{Main results}
When $\lambda=0$ one recovers $J_0$, the \emph{free ideal energy}.
In the closed planar setting, stationary points of $J_0$ have been completely classified: they consist of multiply-covered circles
\cite{AMWW20}.
From the variational and flow viewpoints, the scaling of $J_0$ strongly favours length growth, making compactness and
global control delicate; one natural remedy is to impose an external length constraint or to couple the energy to a
length term, motivating the study of $J_\lambda$.
 
Adding the term $\lambda L$ to $J_0$ (to produce $J_\lambda$) is a singular perturbation.
The primary task of this paper is to describe the space of critical points of $J_\lambda$.
Our construction produces lemniscate- and butterfly-type closed critical points whose collection contains infinitely
many geometrically distinct curves.
The critical points for $J_0$, the multiply-covered circles, are not close to these new families of critical points.
Instead, the circles become periodic but not closed curves: circular waves.
We describe these below.

\begin{itemize}
\item \emph{Lemniscate and butterfly type.}
For every fixed $\lambda>0$ we construct infinitely many geometrically distinct smooth closed immersed critical points
\(
\{\Gamma_{n,j}\}
\)
with turning number $0$.
These are four glued copies of a fundamental quarter-arc, itself constructed by direct minimisation.
These quarter-arcs may not glue smoothly together to form a closed curve, however, there is one degree of freedom retained: the midpoint angle of the quarter-arc.
Exploiting this freedom, we discover two Fresnel phases
\(
3\pi/4,
7\pi/4
\pmod{2\pi}.
\)
At the selected value the midpoint angle is free at first order, and this enables the glued curve to be smooth across the join and thus be a new critical point.
The collection contains infinitely many geometrically distinct critical points; we prove in fact that
\(
J_\lambda[\Gamma_{n,j}]\to\infty
\)
as $n$ tends to infinity; see Theorem~\ref{thm:main-infinite}.

\item \emph{Circular waves.}
Although circles cease to be stationary among {closed} curves once $\lambda>0$, remnants of the $\lambda=0$
degeneracy persist as {non-closed} stationary immersions
\(
\Gamma:\R\to\R^2
\)
whose curvature jet and tangent are periodic but whose position drifts by a nonzero translation per period.
Starting from a multiply-covered semicircle at $\lambda=0$, a perturbative shooting/implicit-function-theorem argument
produces, for each winding index $N\in\N_0$, a one-parameter family of such waves.  Their intrinsic mean turning rate
distinguishes the fixed-length branches.  After dilation, a suitable choice of one wave from each branch gives, for
every prescribed $\lambda>0$, infinitely many pairwise non-congruent circular waves stationary for $J_\lambda$.
See Theorem~\ref{thm:circular-waves-infinite}.
\end{itemize}

We present visualisation of the families exposed here in Figures \ref{fig:selected-first-pair}, \ref{fig:circular-waves}, \ref{fig:selected-lemniscate-type-candidates}, and \ref{fig:selected-butterfly-type-candidates}.

Although there is no critical point of $J_0$ with turning number zero, a homothetic solution to the free ideal flow, whose profile is called the `ideal lemniscate', has been conjectured to exist in \cite{AW26}.
It is tempting to conjecture then that this `ideal lemniscate' becomes, after singular perturbation in the energy by length to form $J_\lambda$, the least-energy selected lemniscate-type critical point.  The lemniscate in Figure \ref{fig:selected-first-pair} is suggestive of this possibility.
We note that the least-energy lemniscate was also identified recently by Okabe-Yamaguchi \cite{OY26}.

If this does indeed turn out to be the case, then that would open the door to associate all of the critical points found here for $J_\lambda$ to (distinct) homothetic solutions to the free ideal flow.
That would greatly expand the known homothetic solutions, which currently consists only of the epicyclic expanders constructed in \cite{AW26}.

\begin{figure}[!t]
  \centering
  
  \hspace{1cm}\includegraphics[width=0.37\linewidth,height=0.36\textheight,keepaspectratio]{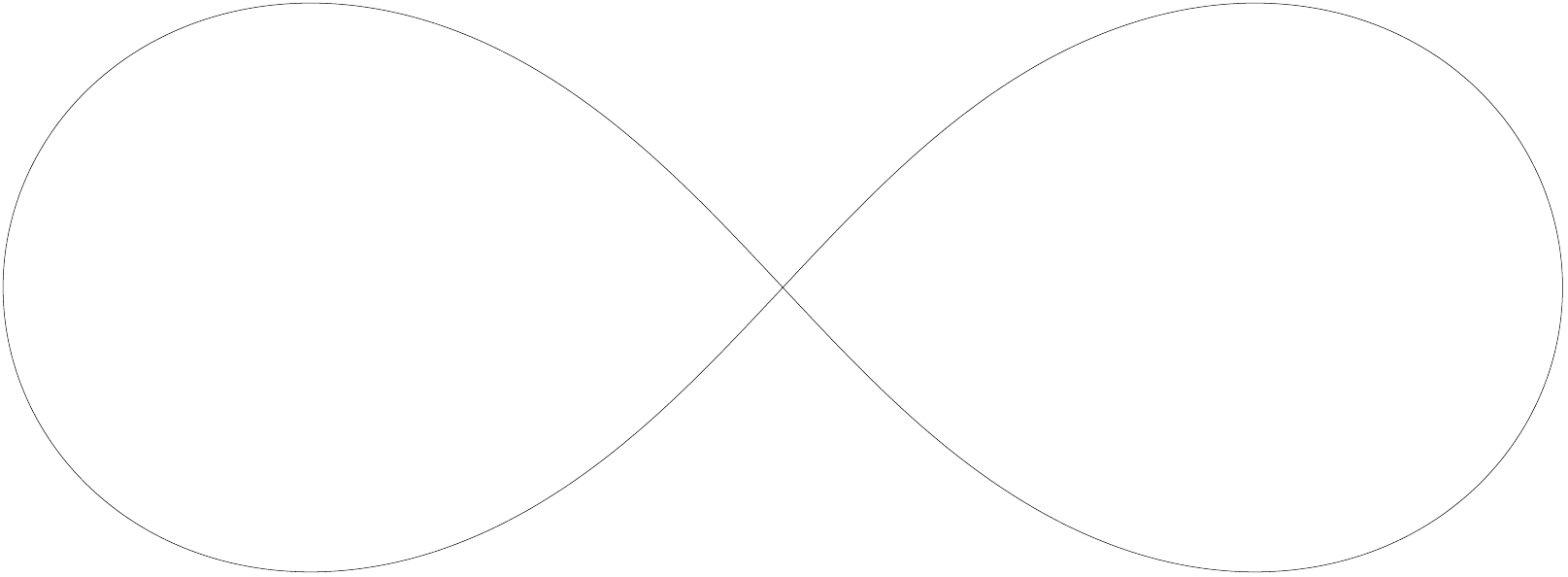}
  \hspace{3cm}\includegraphics[width=0.17\linewidth,height=0.36\textheight,keepaspectratio]{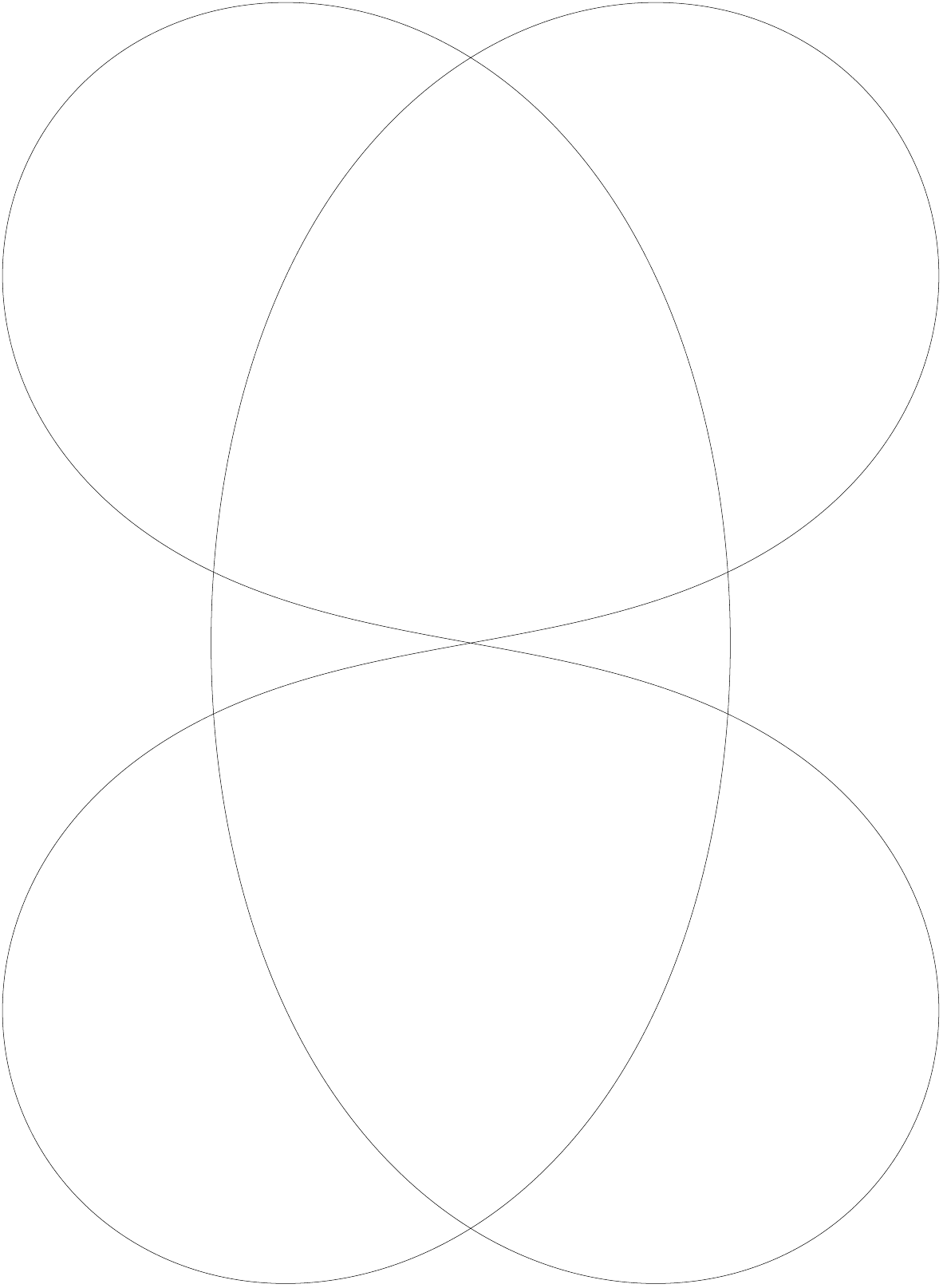}
  \hspace{2cm}
  
  \caption{The two least-energy numerically selected critical point candidates for \(J_1\).  Left: curve \(1\), the least-energy lemniscate-type candidate.  Right: curve \(2\), the least-energy butterfly-type candidate.  These are the first two selected zeros of the midpoint balance defect \(Q\) in the numerical scan.
  We present many further examples from each class in Figures \ref{fig:selected-lemniscate-type-candidates} and \ref{fig:selected-butterfly-type-candidates}, and diagnostics for the critical curves in Table \ref{tab:selected-critical-point-data}.}
  \label{fig:selected-first-pair}
\end{figure}

\subsection{Related literature}
\label{sec:related}
The free ideal energy $J_0$ measures the $L^2$-variation of curvature along the curve. The same unpenalised curvature-variation functional appeared earlier in computer-aided geometric design under the name \emph{minimum variation curve} (MVC) energy, where one minimises the arclength integral of the square of the arclength derivative of curvature to obtain fair curves \cite{MoretonSequin1992FairSurface,Moreton1992MVC}.
An independent CAD precursor appears in Ohlin \cite{O85,O87}: Moreton \cite{Moreton1992MVC} attributes to Ohlin a ``curvature variation minimising'' spline.
In that literature, Euler spirals are identified as one family of minimisers when endpoint curvatures are prescribed \cite{Levien2008EulerSpiralHistory,Levien2009FromSpiralToSpline}. Energies of this type
(and closely related discrete surrogates) arise naturally in curve fairing and geometric design,
where one seeks ``minimum-variation'' shapes with slowly varying curvature; see, for example,
the optimisation framework of Moreton-S\'equin and subsequent algorithmic work
\cite{MoretonSequin1992FairSurface,HararyTal2010EulerSpirals}.
From the analytic side, the $L^2(ds)$-steepest descent gradient flow for $J_0$ is a sixth-order geometric evolution equation called the \emph{ideal flow}.
For closed planar curves, global existence
and convergence theory under a uniform length bound or in a scale-invariant neighbourhood of a multiply-covered circle are available \cite{AMWW20}; there it is additionally shown that the only smooth
critical points of $J_0$ among closed curves are multiply-covered circles.
As mentioned, a key difficulty is the intrinsic scaling of $J_0$, which encourages growth of
$L[\gamma]$ along the free ideal flow; one way to remedy this is to impose or enforce length control.
In this direction, the {length-constrained ideal curve flow} was introduced, for which
long-time existence and exponential convergence (for initial data close to a multiply-covered circle in a scale-invariant sense) is known  \cite{MWW}.
Very recently this analysis was generalised to higher-order ideal energies \cite{McW26}.
Related well-posedness and stability results are also available for boundary-value problems for the ideal flow and,
more generally, for higher-order curvature-derivative energies under Neumann-type conditions
\cite{McCoyWheelerWu2020SixthOrderBC,Wu2021ShortTimeExistence,McCoyWheelerWu2022HighOrderNeumann}.


\subsection{Angle formulation and the symmetry-reduced Euler-Lagrange equation}

Let $\Gamma:[0,L]\to\R^2$ be a unit-speed immersion with tangent
$T=(\cos\theta,\sin\theta)$ and curvature $k=\theta_s$.
Then
\[
\frac12\int_\gamma k_s^2\,ds=\frac12\int_0^L \theta_{ss}^2\,ds,
\qquad
L[\gamma]=L.
\]
Stationarity of $J_\lambda$ subject to particular constraints (see Lemma \ref{lem:EL-and-natural}) yields the scalar fourth-order Euler-Lagrange equation
\begin{equation}\label{eq:thetaequiv}
\theta_{ssss}+\alpha\cos\theta=0,
\end{equation}
where $\alpha\in\R$ is the Lagrange multiplier associated with the closure constraint.  Equivalently, one may write $\theta_{ssss}=c\cos\theta$ with $c:=-\alpha$.
In the shooting formulation this constant $c$ is expressed in terms of the conserved quantities (and, in the penalised
problem, $\lambda$).

Conversely, if $\theta$ is a smooth solution of \eqref{eq:thetaequiv}
satisfying the boundary/seam/symmetry conditions of the relevant
construction, and if
\begin{equation}
\Gamma(s):=\int_0^s(\cos\theta(\sigma),\sin\theta(\sigma))\,d\sigma
\label{eq:defgamma}
\end{equation}
is an immersion, then $\Gamma$ is stationary for the corresponding fixed-length angle problem within that
class.  Full geometric stationarity for $J_\lambda$ additionally requires the length-stationarity/Hamiltonian
condition established in Proposition~\ref{prop:tw-hamiltonian-upgrade}.

\subsection{Theorem Statements}

We collect here the principal results of the paper.

We say that a closed unit-speed immersion $\Gamma:\R/L[\Gamma]\Z\to\R^2$ is a
\emph{symmetric null-turning curve} 
if it is obtained by 
\[
\text{(doubled arc)} \;\longrightarrow\;
\text{($x$-axis reflection)}.
\]
If $\Gamma$ is additionally a  smooth critical point of $J_\lambda$ among closed immersions then we call it a \emph{symmetric null-turning critical point} of $J_\lambda$.
Such curves have two discrete symmetries and zero average curvature.

We work in the $(\ell,\theta)$ variables and define the doubled admissible set
\begin{equation}\label{eq:Adouble}
\begin{aligned}
\mathcal{A}^{(2)}(\Phi):=\big\{(\ell,\theta):\;&\ \ell>0,\quad \theta\in H^2(0,2\ell),\\
&\theta(0)=\tfrac{\pi}{2},\quad
  \theta(\ell)=\tfrac{\pi}{2}+\Phi,\quad
  \theta(2\ell)=\tfrac{\pi}{2},\\
&\theta\ \text{even about }\ell,\quad
  \int_0^\ell\sin\theta\,ds=0\big\}.
\end{aligned}
\end{equation}
For $(\ell,\theta)\in\mathcal{A}^{(2)}(\Phi)$ we set
\begin{equation}\label{eq:Jdouble}
J_\lambda^{(2)}(\ell,\theta)
:=\frac12\int_0^{2\ell}\theta_{ss}^2\,ds+2\lambda\,\ell.
\end{equation}
If $\gamma_{\ell,\theta}:[0,2\ell]\to\R^2$ is the doubled arc reconstructed from $\theta$, then
$J_\lambda^{(2)}(\ell,\theta)=J_\lambda[\gamma_{\ell,\theta}]$.  After reflection across the $x$-axis, the resulting
closed curve has twice this energy; see Lemma~\ref{lem:doubled-energy-additivity}.

For a compact interval $I\subset\R$ we also use the corresponding turning-window class
\begin{equation}\label{eq:Adouble-window}
\mathcal{A}^{(2)}(I)
:=
\bigcup_{\Phi\in I}\mathcal{A}^{(2)}(\Phi).
\end{equation}

\begin{theorem}[Infinitely many symmetric null-turning critical points]\label{thm:main-infinite}
Fix $\lambda>0$ and set
\[
\rho_1:=\frac{3\pi}{4} \quad\text{(lemniscates)},\qquad \rho_2:=\frac{7\pi}{4} \quad\text{(butterflies)}.
\]
There exist $\delta>0$ and $N_0\in\N$ with the following property.
For each turning window
\[
I_{n,j}:=[2\pi n+\rho_j-\delta,\,2\pi n+\rho_j+\delta],\quad 
n\ge N_0,\qquad j\in\{1,2\},
\]
there is a minimiser $(\ell_{n,j},\theta_{n,j})$ of $J^{(2)}_\lambda$ over the window class
$\mathcal{A}^{(2)}(I_{n,j})$ that satisfies the midpoint transversality condition
\[
\lim_{s\uparrow\ell_{n,j}}(\theta_{n,j})_{sss}(s)=0.
\]
Here and below, a superscript $-$ on a midpoint value denotes the trace from the first half-interval.  Define the
selected lifted turning by
\[
\Phi_{n,j}^{\dagger}:=\theta_{n,j}(\ell_{n,j})-\frac\pi2\in\operatorname{int}I_{n,j}.
\]
Consequently the doubled arc reconstructed from $\theta_{n,j}$ is smooth through the midpoint, and its
$x$-axis reflection produces a $C^\infty$ closed unit-speed immersed critical point
\[
\Gamma_{n,j}:\R/(4\ell_{n,j})\Z\to\R^2
\]
of $J_\lambda$.

The curve $\Gamma_{n,j}$ has the following properties.
\begin{enumerate}[label=\textup{(\roman*)},leftmargin=2.2em]
\item \textbf{Symmetries.}
Let $B_{n,j}:=\Gamma_{n,j}(\ell_{n,j})$.  Then $B_{n,j}=\Gamma_{n,j}(3\ell_{n,j})$ and $\Gamma_{n,j}$ enjoys point reflection about $B_{n,j}$ and reflection across the $x$-axis, in the sense stated in \eqref{eq:point-symmetry-gamma} and \eqref{eq:doubled-xaxis-close}.

\item \textbf{Criticality for $J_\lambda$.}
The curve $\Gamma_{n,j}$ is a smooth closed critical point of $J_\lambda$ among all smooth closed immersions:
\[
\delta J_\lambda[\Gamma_{n,j}](V)=0
\quad\text{for every smooth variation field }V\text{ along }\Gamma_{n,j}.
\]

\item \textbf{Selected quarter-arc turning and turning number.}
If $\Theta_{n,j}$ is a continuous lifted tangent angle for $\Gamma_{n,j}$ normalised by $\Theta_{n,j}(0)=\pi/2$, then
\[
\Theta_{n,j}(\ell_{n,j})=\frac\pi2+\Phi_{n,j}^{\dagger},\quad
\Theta_{n,j}(2\ell_{n,j})=\frac\pi2,
\]
\[
\Theta_{n,j}(3\ell_{n,j})=\frac\pi2-\Phi_{n,j}^{\dagger},\quad
\Theta_{n,j}(4\ell_{n,j})=\frac\pi2.
\]
Thus the four quarter-arc turnings are
\[
\Phi_{n,j}^{\dagger},\ -\Phi_{n,j}^{\dagger},\ -\Phi_{n,j}^{\dagger},\ \Phi_{n,j}^{\dagger},
\]
and the total turning of $\Gamma_{n,j}$ is $0$.

\item \textbf{Quantitative energy lower bound and distinctness.}
The energy satisfies
\begin{equation}\label{eq:main-lemniscate-energy-lower}
J_\lambda[\Gamma_{n,j}]
\ge
\frac{16}{3}\Big(\frac{9}{2}\Big)^{\!1/4}\lambda^{3/4}|\Phi_{n,j}^{\dagger}|^{1/2}.
\end{equation}
In particular $J_\lambda[\Gamma_{n,j}]\to\infty$ as $n\to\infty$ for each $j\in\{1,2\}$.  Consequently the
collection $\{\Gamma_{n,j}\}$ contains an infinite pairwise non-congruent subsequence.
\end{enumerate}
\end{theorem}

\begin{remark}
\label{rem:lpif}
The selected lemniscate-type critical points from Theorem \ref{thm:main-infinite} bear a striking visual similarity to the super-lemniscates found in \cite{MW26}, despite satisfying different equations.
There is an interesting non-trivial overlap between the set of (open) curves with curvature satisfying \eqref{eq:thetaequiv} and the super-lemniscates.
We give the details in Appendix \ref{app:lpif}.
\end{remark}

For the circular waves, we set up the following curvature-jet system:
\[
k:=\theta_s,\qquad p:=k_s=\theta_{ss},\qquad q:=k_{ss}=\theta_{sss}.
\]
For stationary curves of the length-penalised ideal functional the Euler-Lagrange system admits an autonomous first-order formulation in $(k,p,q,\theta)$.
In the presence of a vertical initial tangent $\theta(0)=\pi/2$ and the reflection-seam condition $p(0)=k_s(0)=0$,
the conserved quantities reduce to a single constant forcing
\begin{equation}\label{eq:c-def-cwN}
c:=ab+\lambda,\qquad a:=k(0),\quad b:=q(0),
\end{equation}
and the stationary system becomes
\begin{equation}\label{eq:shooting-cwN}
\left\{
\begin{aligned}
k_s&=p,\\
p_s&=q,\\
q_s&=c\cos\theta,\\
\theta_s&=k,
\end{aligned}
\right.
\qquad
k(0)=a,\quad p(0)=0,\quad q(0)=b,\quad \theta(0)=\frac{\pi}{2}.
\end{equation}
Equivalently, $\theta$ satisfies the fourth-order scalar ODE $\theta_{ssss}=c\cos\theta$.

\begin{theorem}[Infinitely many geometrically distinct circular waves]\label{thm:circular-waves-infinite}
Fix $\ell>0$.
For each $N\in\N_0$ there exist $\varepsilon_{0,N}>0$ and $C^\infty$ functions
\[
\varepsilon\longmapsto (a_N(\varepsilon),b_N(\varepsilon),\lambda_N(\varepsilon)),
\qquad |\varepsilon|<\varepsilon_{0,N},
\]
with $(a_N(0),b_N(0),\lambda_N(0))=((2N{+}1)\pi/\ell,0,0)$, such that for each $|\varepsilon|<\varepsilon_{0,N}$ the
solution of the curvature-jet system \eqref{eq:shooting-cwN} satisfies the endpoint conditions
\[
\theta(\ell)=\frac{\pi}{2}+(2N+1)\pi,\qquad p(\ell)=0,\qquad q(\ell)=\varepsilon.
\]
Moreover:
\begin{enumerate}[label=\textup{(\roman*)},leftmargin=2.2em]
\item $\lambda_N(\varepsilon)>0$ for all sufficiently small $\varepsilon\neq 0$.
\item For $0<|\varepsilon|<\varepsilon_{0,N}$ the corresponding arc extends by repeated reflection to a $C^\infty$
unit-speed stationary immersion $\Gamma_{N,\varepsilon}:\R\to\R^2$ whose tangent is $2\ell$-periodic and whose position
is periodic up to translation:
\[
\Gamma_{N,\varepsilon}(s+2\ell)=\Gamma_{N,\varepsilon}(s)+P_{N,\varepsilon},
\qquad P_{N,\varepsilon}\neq 0.
\]
In particular, $\Gamma_{N,\varepsilon}$ is not closed.
\item The waves are pairwise non-congruent across different $N$:
if $N\neq M$, then $\Gamma_{N,\varepsilon}$ and $\Gamma_{M,\varepsilon'}$ are not congruent for any sufficiently small
$\varepsilon,\varepsilon'\neq 0$.  More precisely, their intrinsic mean turning rates are
$\Omega(\Gamma_{N,\varepsilon})=(2N+1)\pi/\ell$.
\item For every prescribed $\lambda_\ast>0$, one can choose one sufficiently small nonzero parameter on each
$N$-branch and dilate the corresponding waves to obtain infinitely many pairwise non-congruent circular waves
stationary for $J_{\lambda_\ast}$.
\end{enumerate}
Consequently, there exist infinitely many geometrically distinct circular waves.
\end{theorem}

\paragraph{Organisation of the paper.}
Section~\ref{sec:aux} collects the coercive turning estimates, scaling identities, and reflection/gluing lemmata used
throughout.  In Section~\ref{sec:doubled-direct-method} we set up and solve the symmetry-built-in direct-method
minimisation on a doubled interval.  Section~\ref{sec:infinite-family} applies this construction for a sequence of
lifted turnings to prove Theorem~\ref{thm:main-infinite}.  Finally, Section~\ref{sec:circular-waves} develops the
perturbative shooting and reflection scheme from multiply-covered semicircles to prove
Theorem~\ref{thm:circular-waves-infinite}. 
Appendix~\ref{app:lpif} details the relationship between the lemniscate critical points observed here and the super-lemniscates found in \cite{MW26}.
Appendix~\ref{sec:symcrit} records a finite-group version of the principle of symmetric criticality for comparison.
Appendix~\ref{sec:numerical-visualisation} briefly describes the numerical
procedure used to generate the figures.

\subsection*{Acknowledgements}

This work began in earnest after the MATRIX research program ``Gradient Flows in Geometry and PDE'' held in January 2025.
It was continued over subsequent visits, for which the authors thank their home institutions and grants FT250100880 and DP250101080 for partial financial support.

\begin{figure}[hp]
  \centering

\includegraphics[width=0.5\linewidth,trim=0cm 2cm 0cm 0cm,clip=true,angle=270]{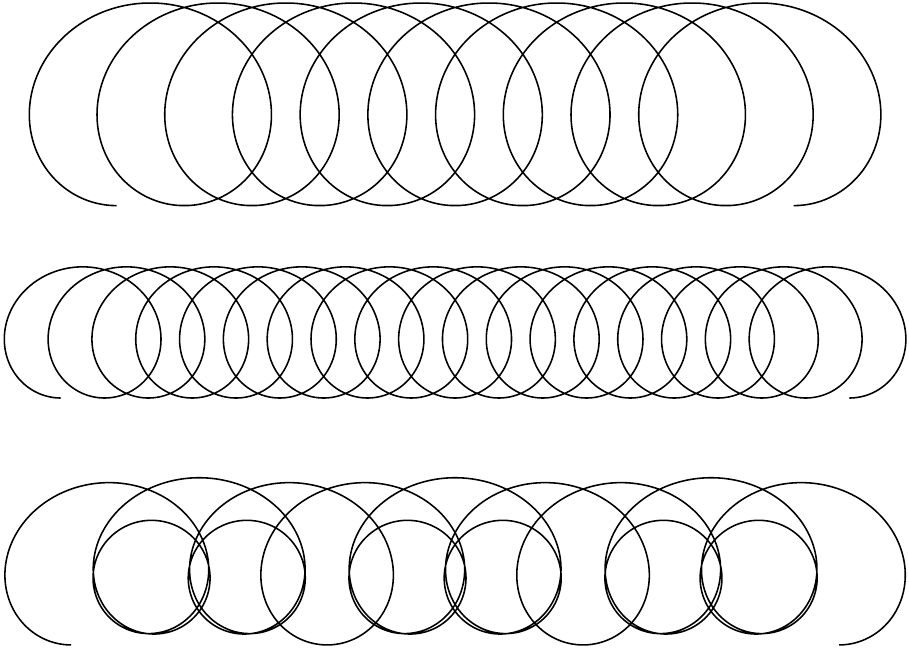}

  \vspace{1.0em}

  \begin{minipage}{\linewidth}
    \centering
    \small
    \setlength{\tabcolsep}{6pt}
    \renewcommand{\arraystretch}{1.10}

    \begin{tabular}{r r r r r}
      \toprule
      $N$ & $\Delta\Theta$ (shooting cell) & $L_{\mathrm{cell}}$ &
      $\tfrac12\!\int_{\mathrm{cell}} k_s^2\,ds$ & $J_{1,\mathrm{cell}}$\\
      \midrule
      0 & $2\pi$  & 3.054066  & 0.608250 & 3.662317 \\
      1 & $6\pi$  & 9.162204  & 1.824752 & 10.986955 \\
      2 & $10\pi$ & 15.241431 & 3.273536 & 18.514967 \\
      \bottomrule
    \end{tabular}
  \captionof{table}{Diagnostics over the reflected shooting cell for the circular waves above, after rescaling so that
    the stationary parameter is $\lambda=1$.  Here $\Delta\Theta=2(2N+1)\pi$ is the lifted tangent-angle increment over
    the chosen cell, $L_{\mathrm{cell}}$ is its arclength, and
    $J_{1,\mathrm{cell}}=\frac12\int_{\mathrm{cell}}k_s^2\,ds+L_{\mathrm{cell}}$.  The cell is a tangent period but is not
    asserted to be primitive.}
    \label{tab:circular-waves}
    
    \end{minipage}

  \vspace{0.4em}

  \caption{Circular wave solutions for indices $N=0,1,2$.
  For each $N$ we shoot a semicircle-like stationary arc on $[0,\ell]$ with lifted endpoint angle
  $\theta(\ell)=\frac{\pi}{2}+(2N+1)\pi$ and seam conditions $k_s(0)=k_s(\ell)=0$, then extend to a shooting cell
  $[0,2\ell]$ by reflection and tile by the resulting translation vector.  Each wave is then rescaled by the
  dilation $\Gamma\mapsto\rho\,\Gamma$ with $\rho=\lambda(\varepsilon)^{1/4}$ so that the stationary parameter becomes
  $\lambda=1$.  The curves are periodic up to translation,
  hence not closed, but are smooth stationary immersions for $J_1$.}
  \label{fig:circular-waves}
\end{figure}

\newcommand{\SelectedCriticalPlot}[1]{%
  \begin{minipage}[c][0.118\textheight][c]{0.333333\linewidth}%
    \centering
    \includegraphics[width=0.985\linewidth,height=0.114\textheight,keepaspectratio]{figs/selected/#1}%
  \end{minipage}%
}

\begin{figure}[p]
  \centering
  \setlength{\parindent}{0pt}%
  \setlength{\parskip}{0pt}%
  \noindent%
  \SelectedCriticalPlot{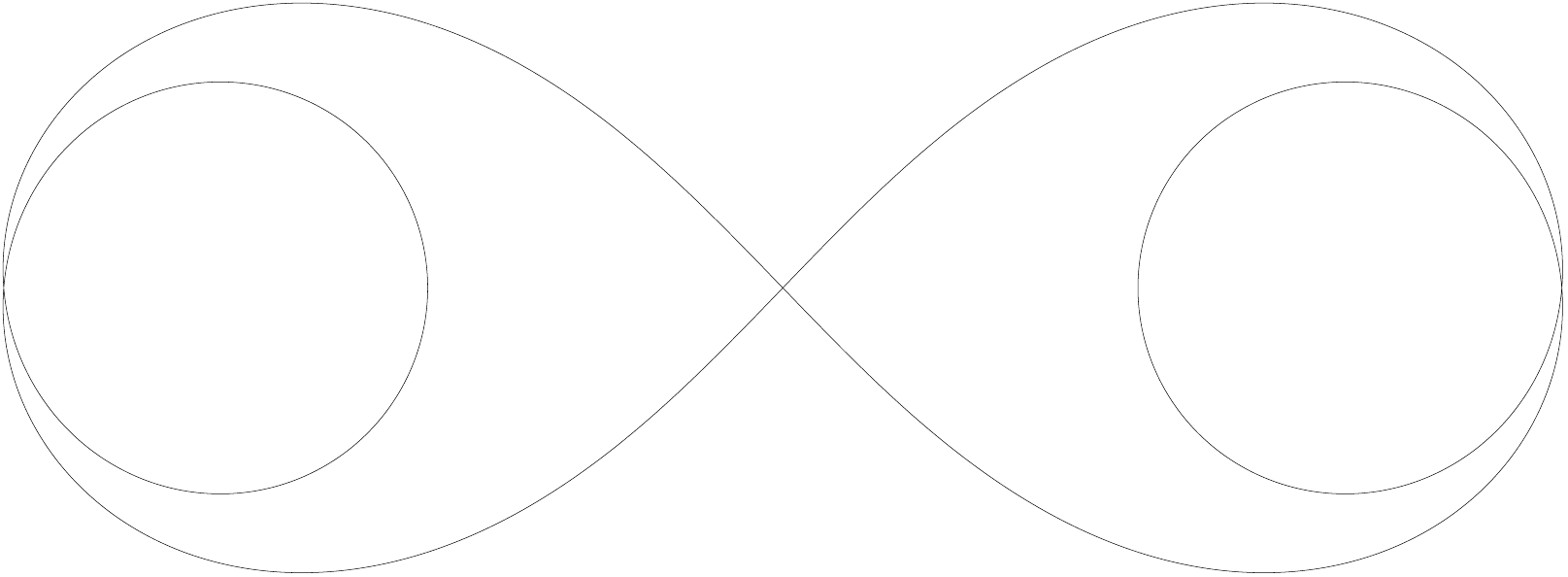}%
  \SelectedCriticalPlot{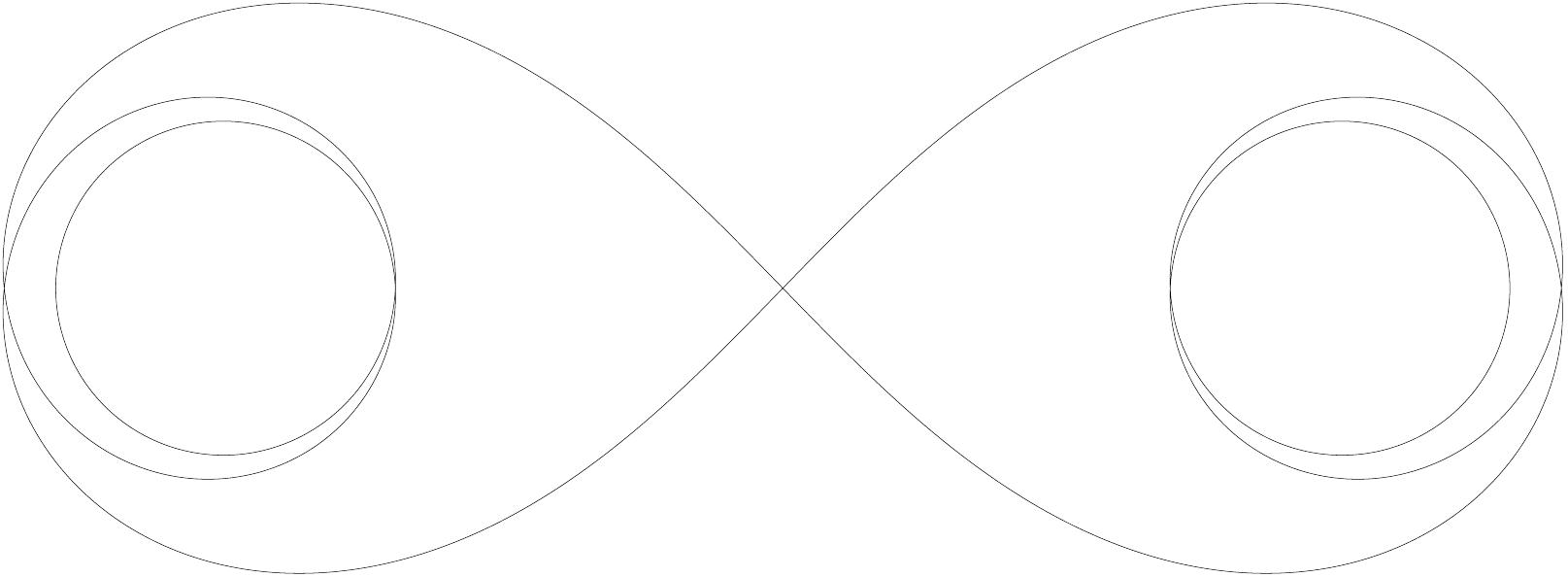}%
  \SelectedCriticalPlot{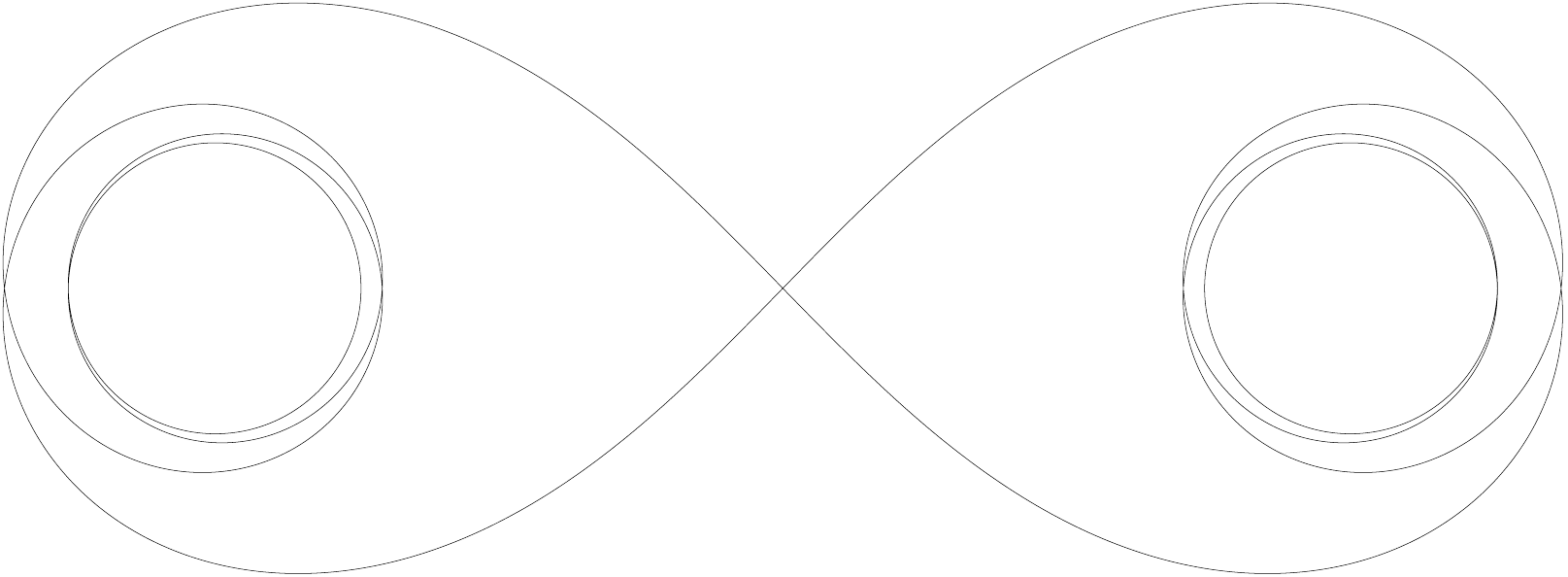}%
  \par\nointerlineskip%
  \SelectedCriticalPlot{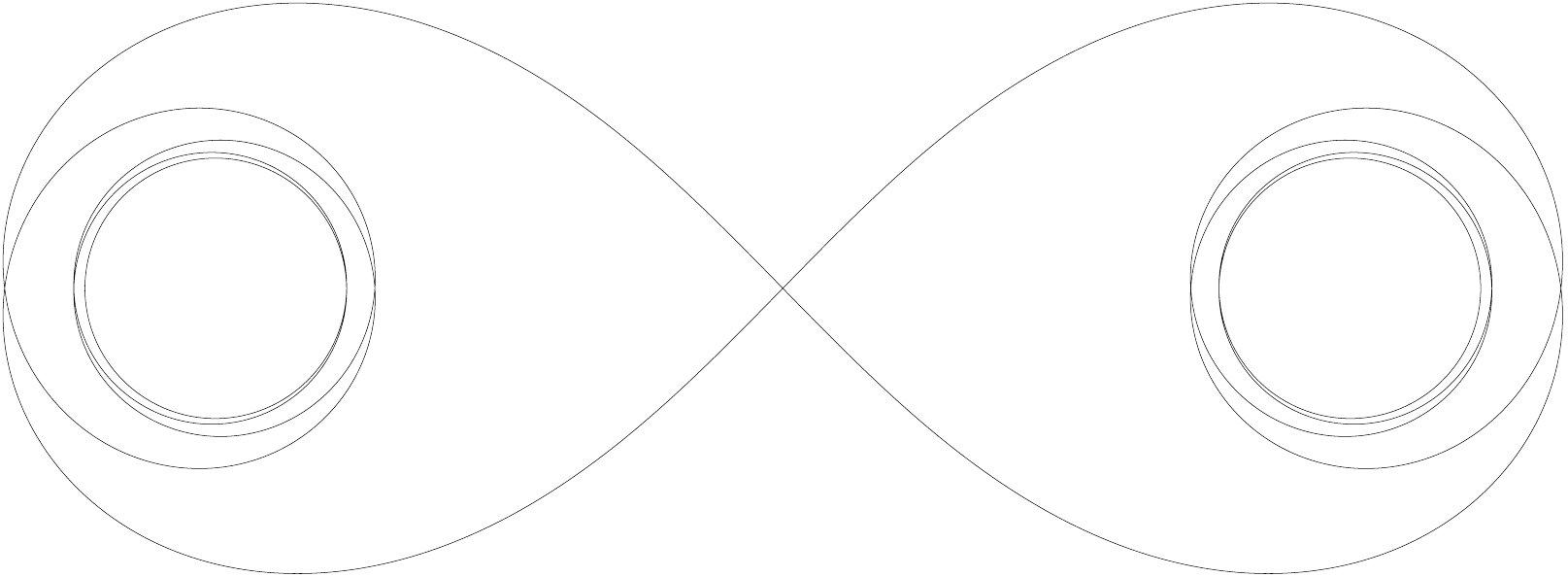}%
  \SelectedCriticalPlot{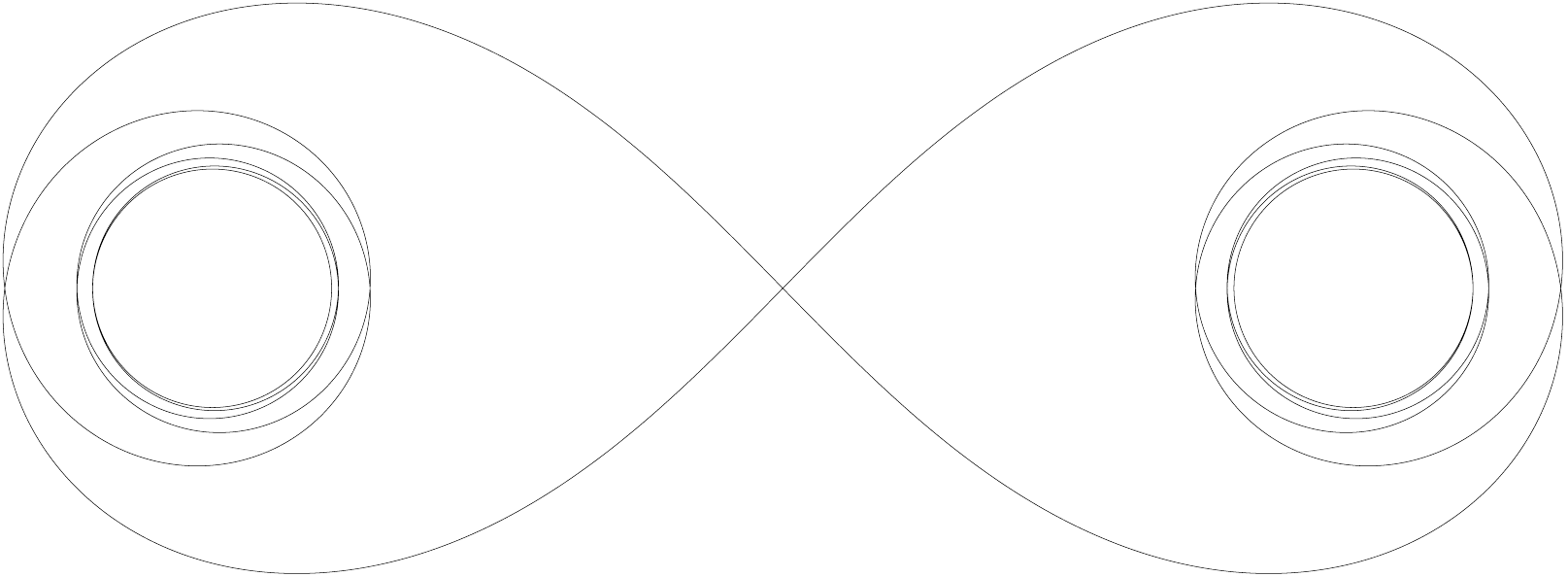}%
  \SelectedCriticalPlot{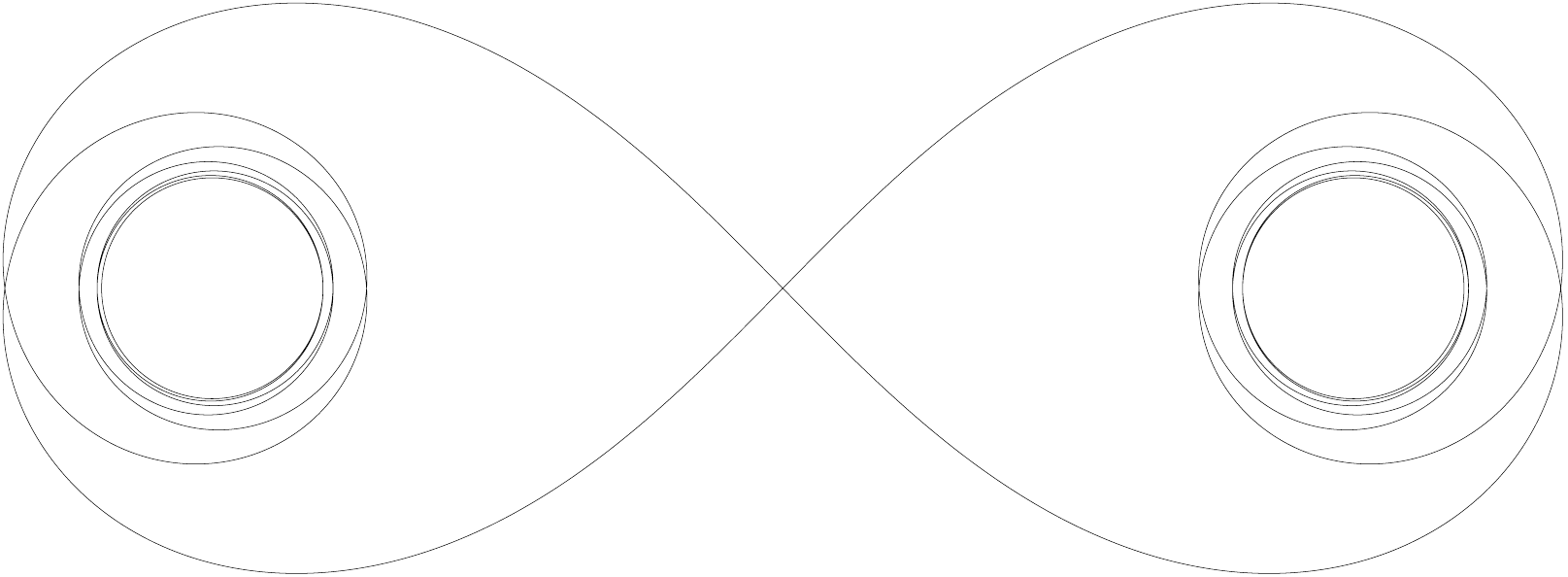}%
  \par\nointerlineskip%
  \SelectedCriticalPlot{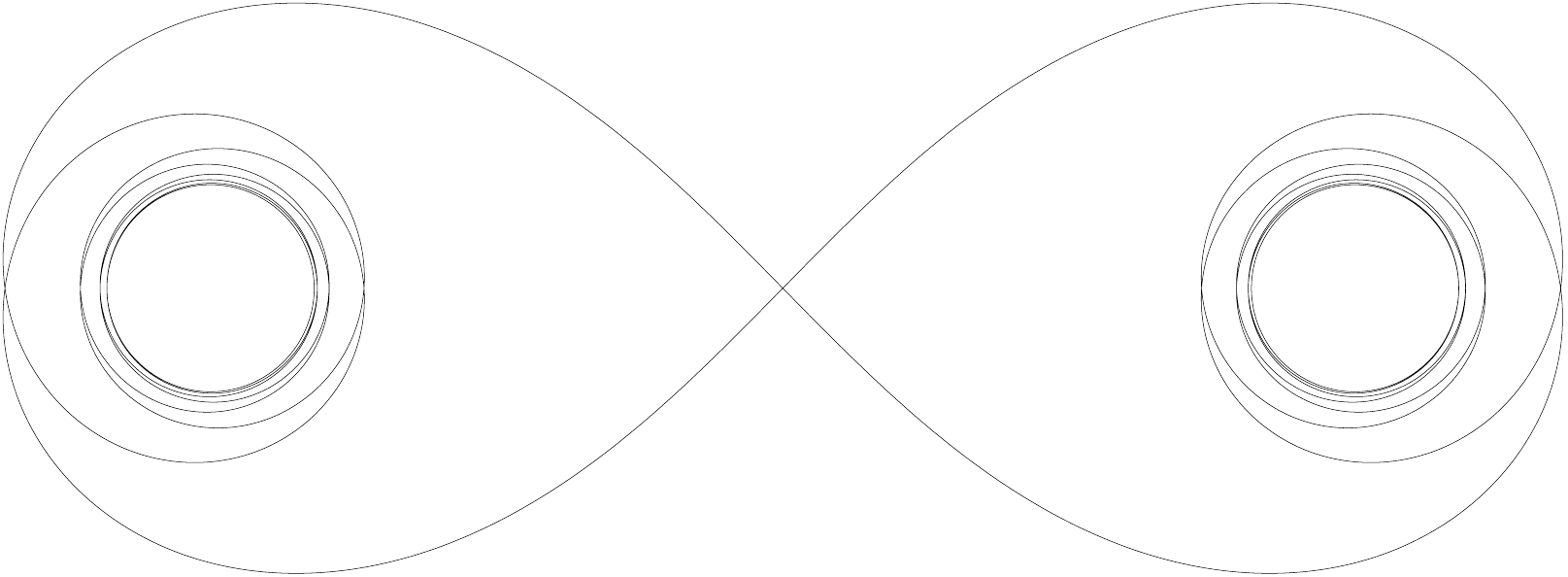}%
  \SelectedCriticalPlot{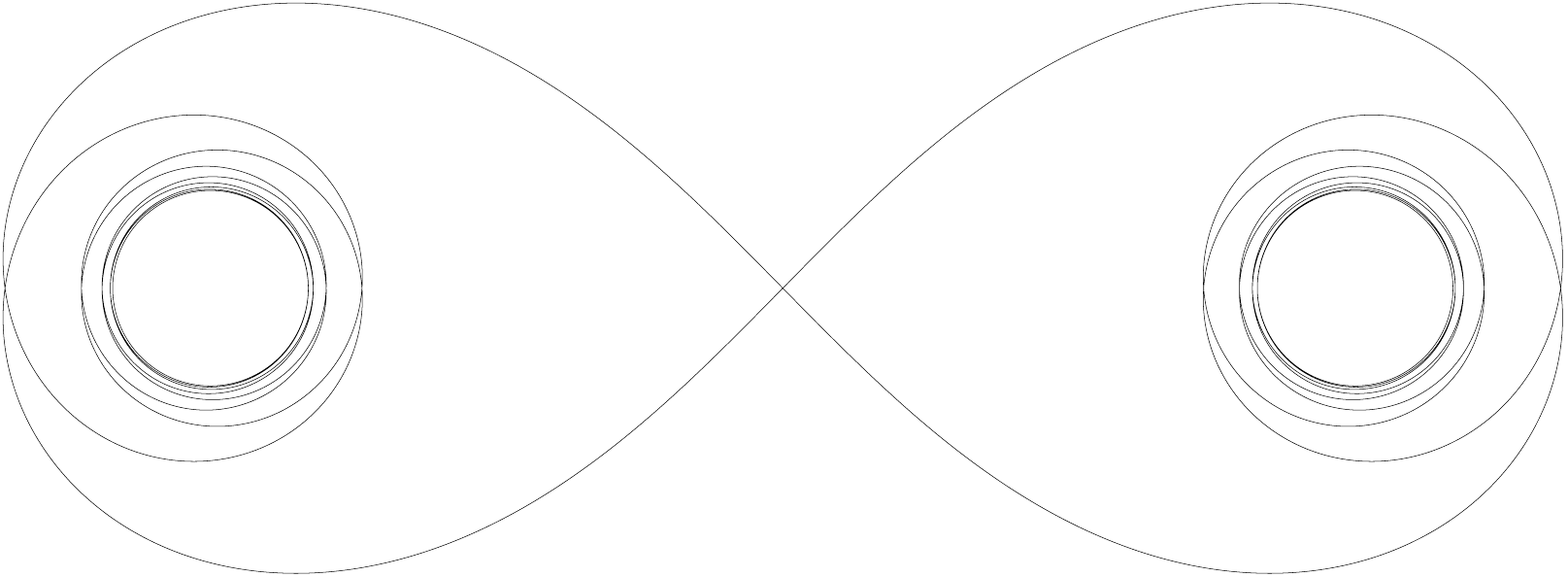}%
  \SelectedCriticalPlot{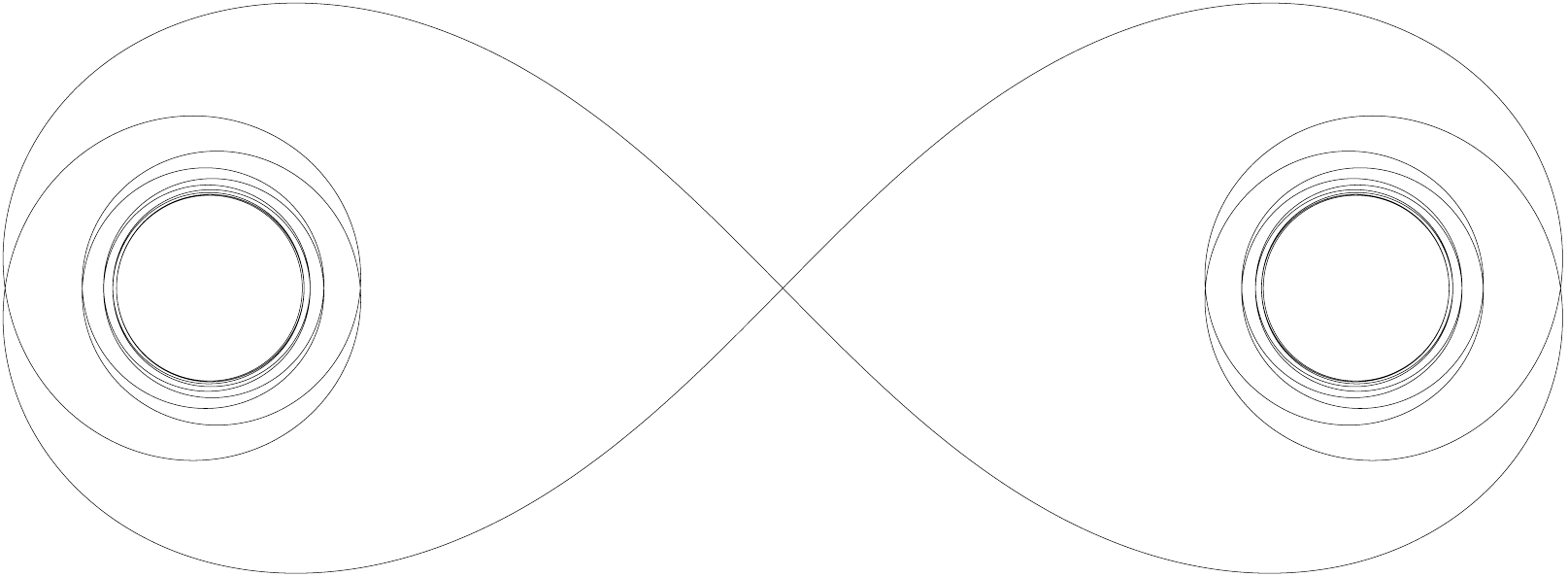}%
  \par\nointerlineskip%
  \SelectedCriticalPlot{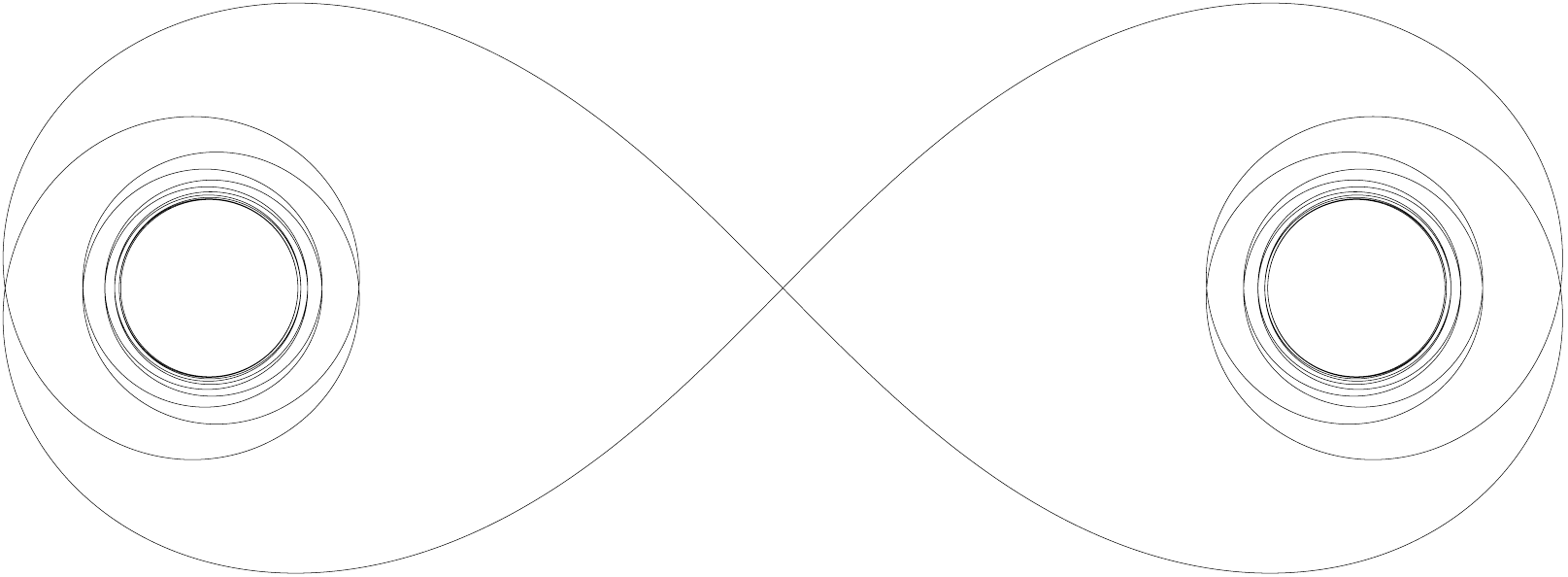}%
  \SelectedCriticalPlot{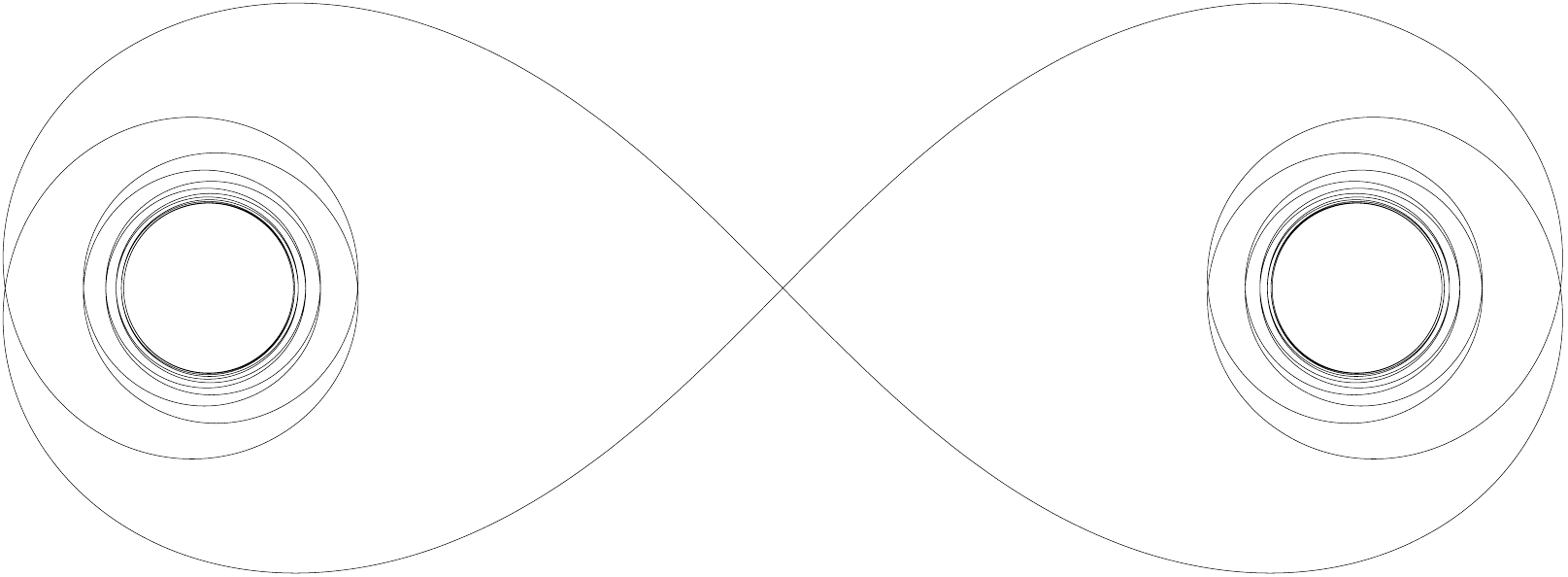}%
  \SelectedCriticalPlot{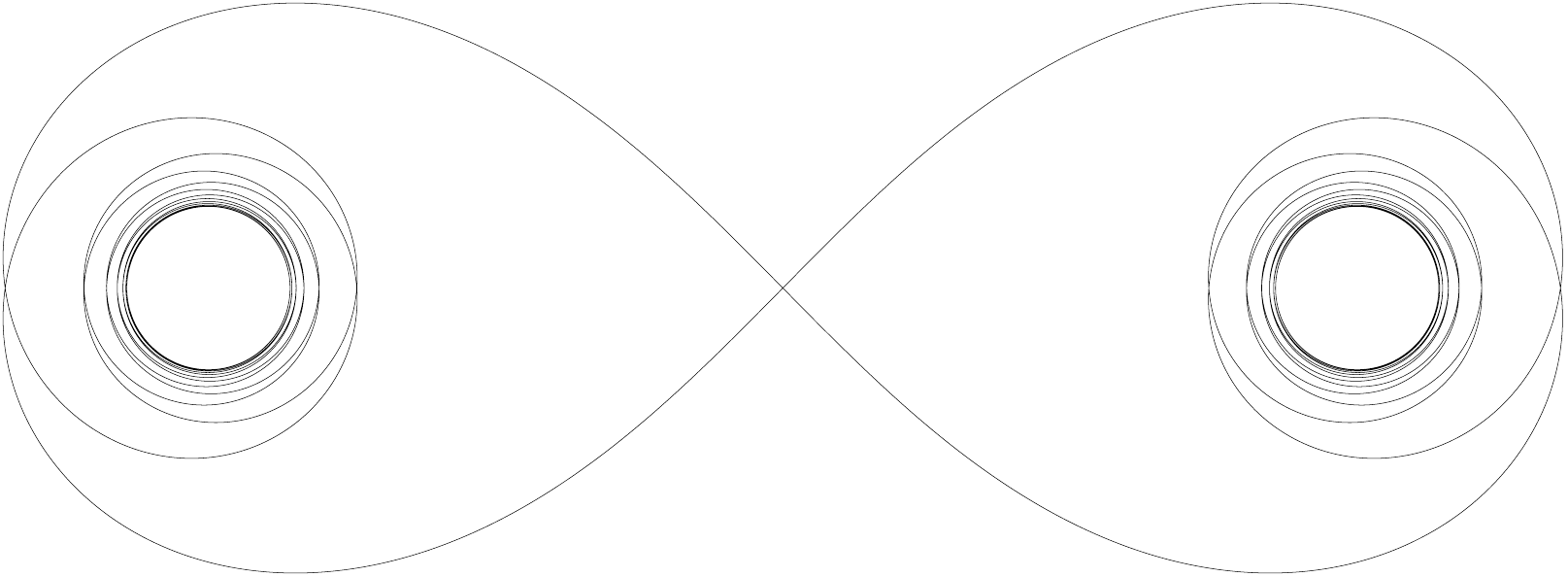}%
  \par\nointerlineskip%
  \SelectedCriticalPlot{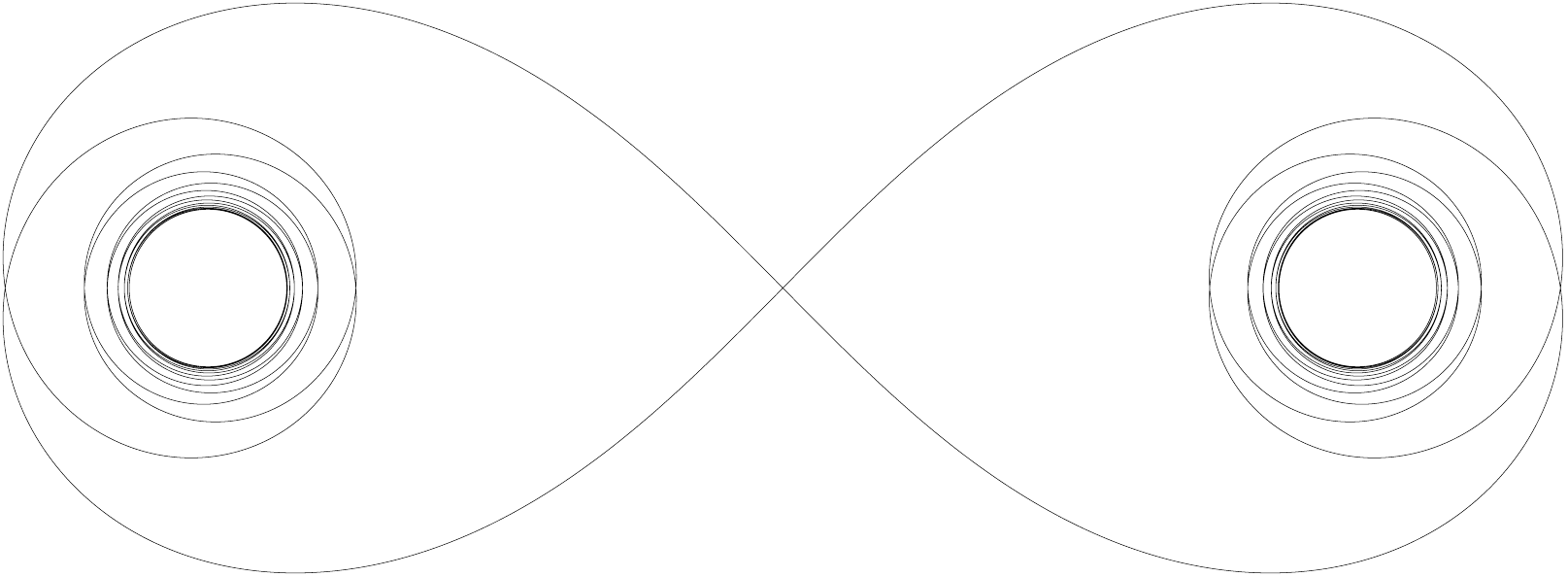}%
  \SelectedCriticalPlot{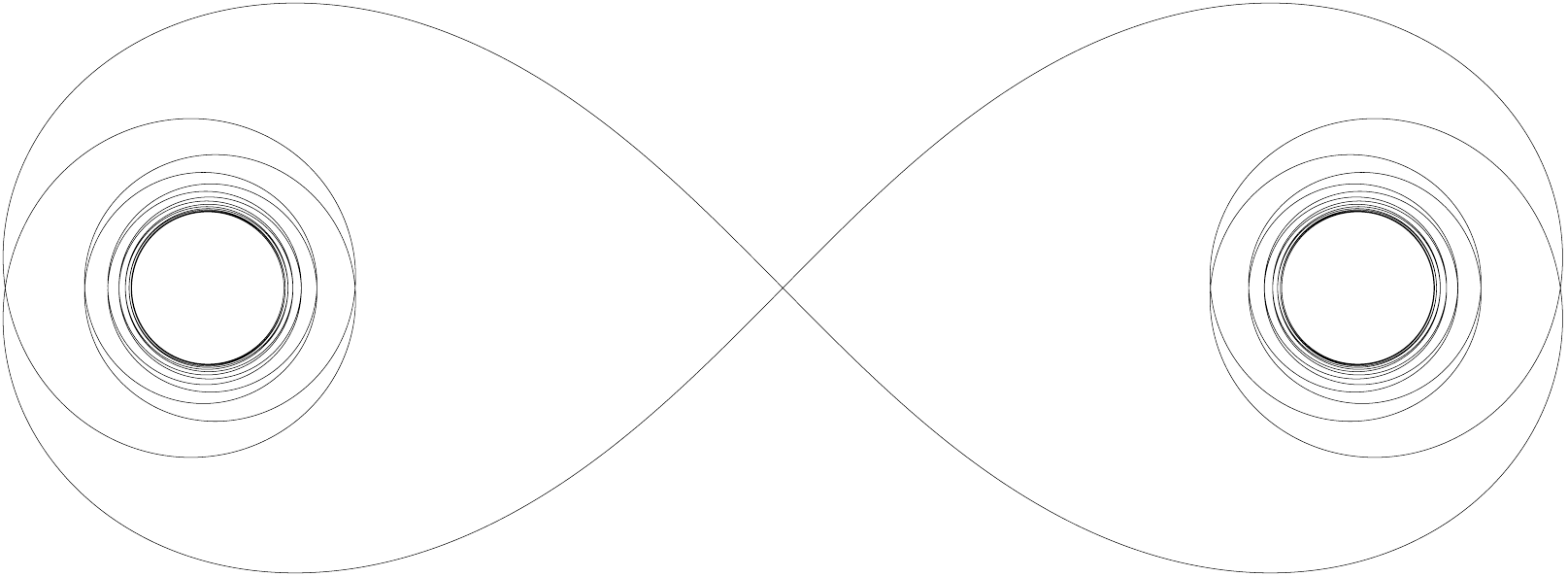}%
  \SelectedCriticalPlot{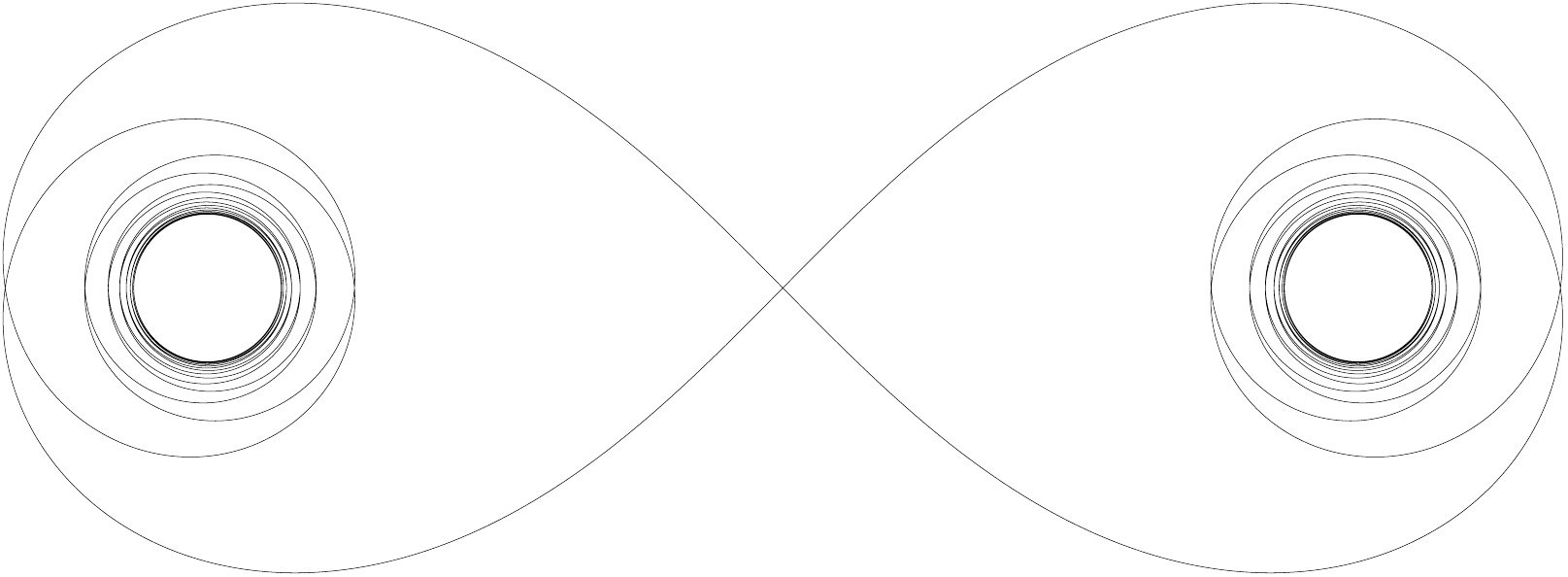}%
  \par\nointerlineskip%
  \SelectedCriticalPlot{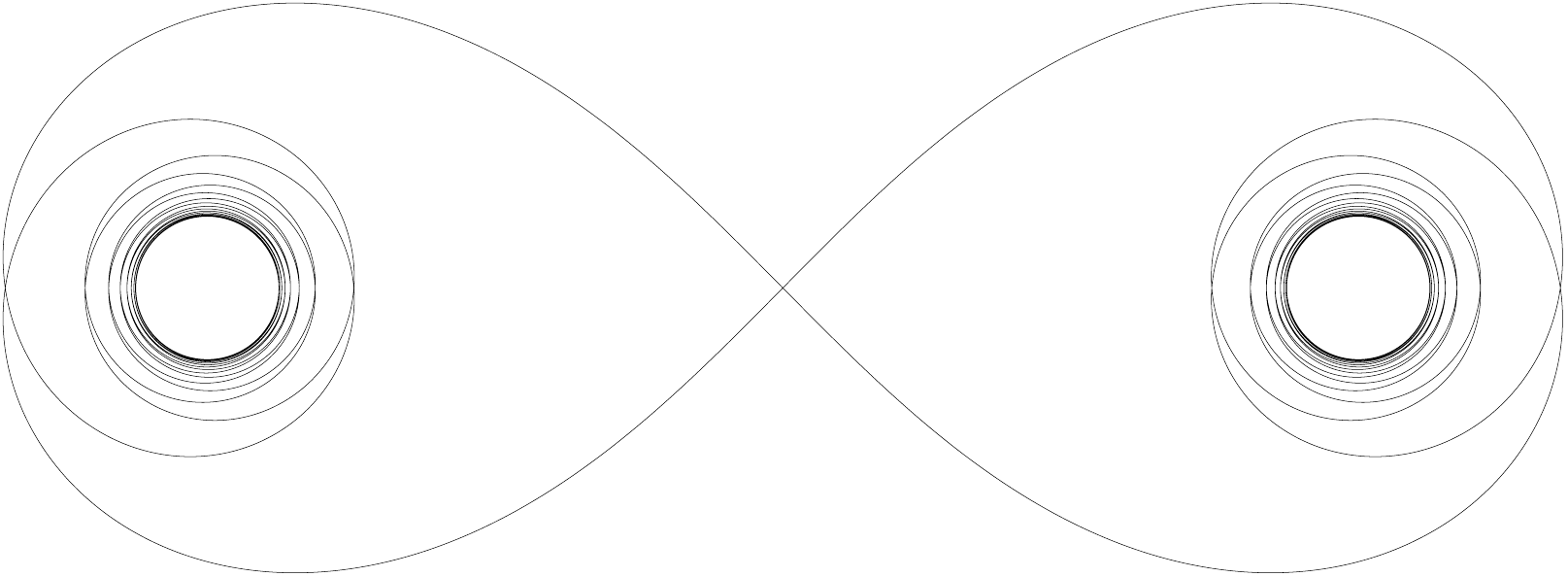}%
  \SelectedCriticalPlot{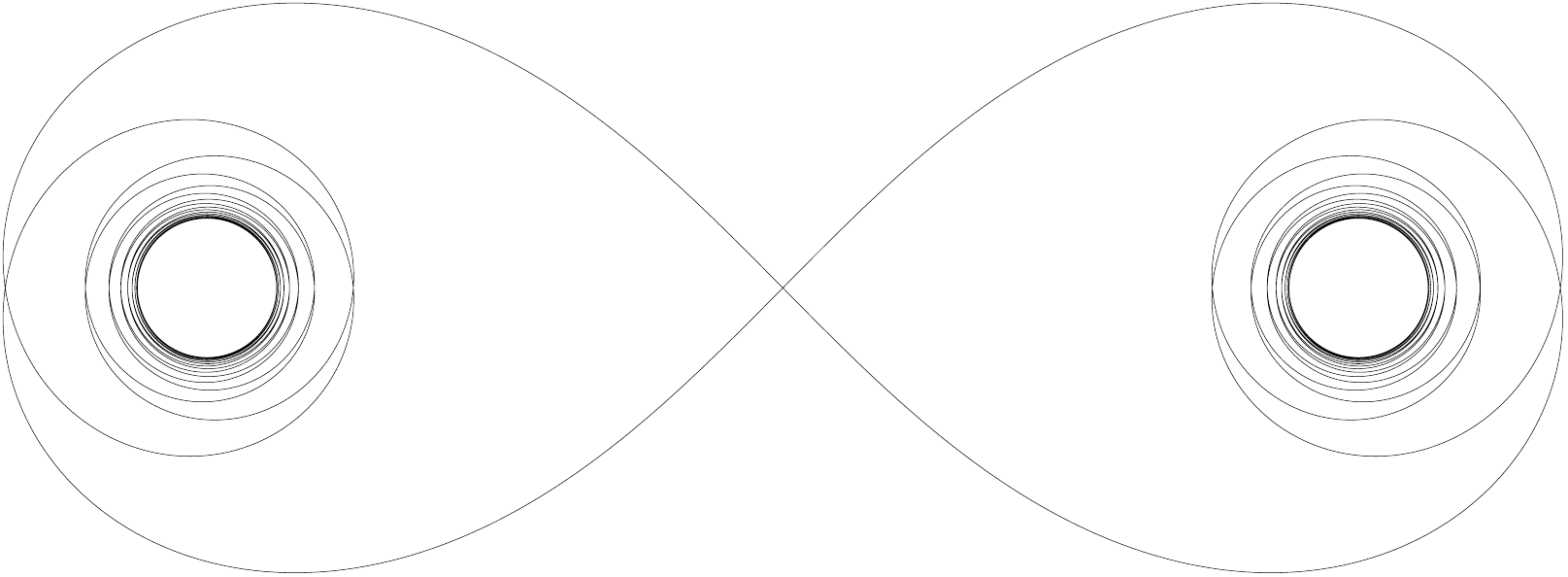}%
  \SelectedCriticalPlot{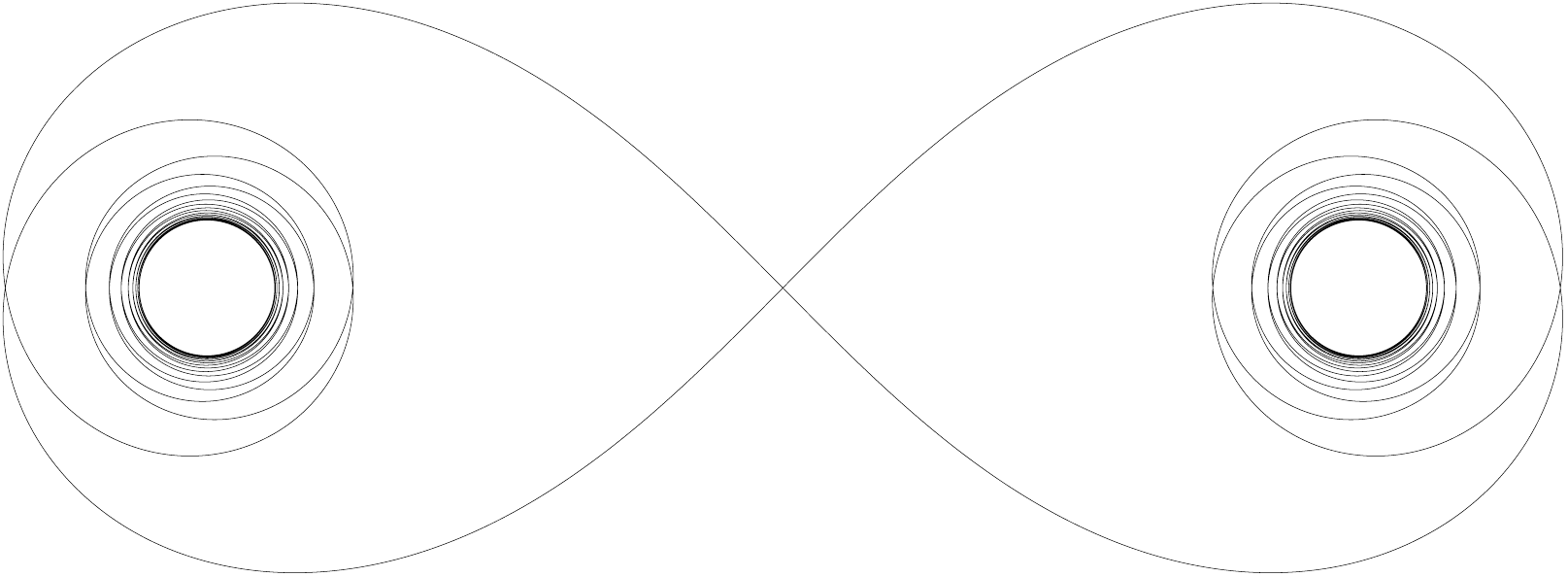}%
  \par\nointerlineskip%
  \SelectedCriticalPlot{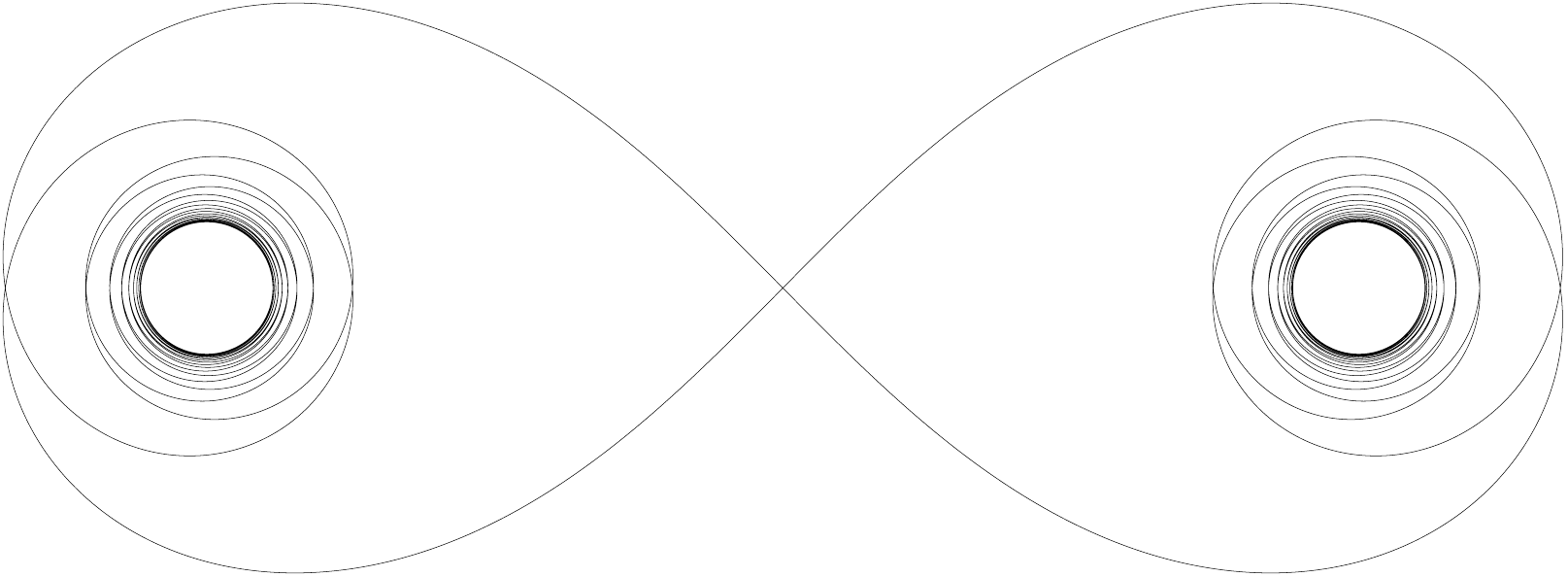}%
  \SelectedCriticalPlot{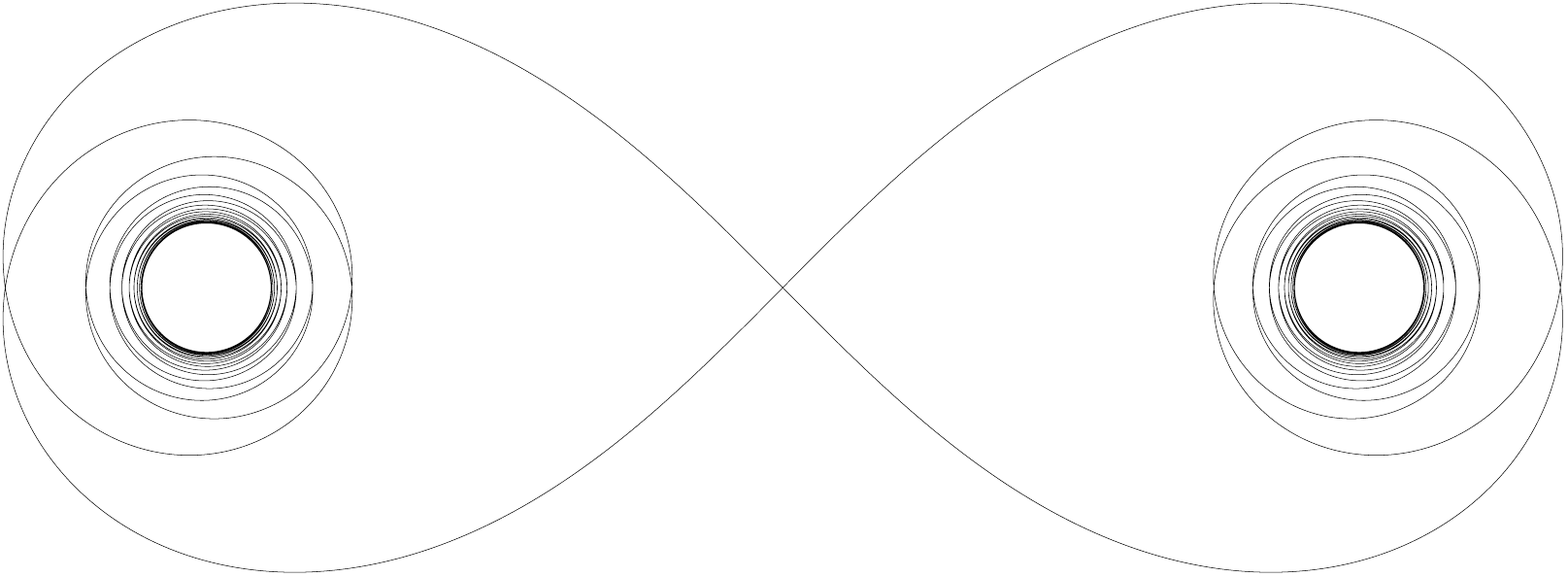}%
  \SelectedCriticalPlot{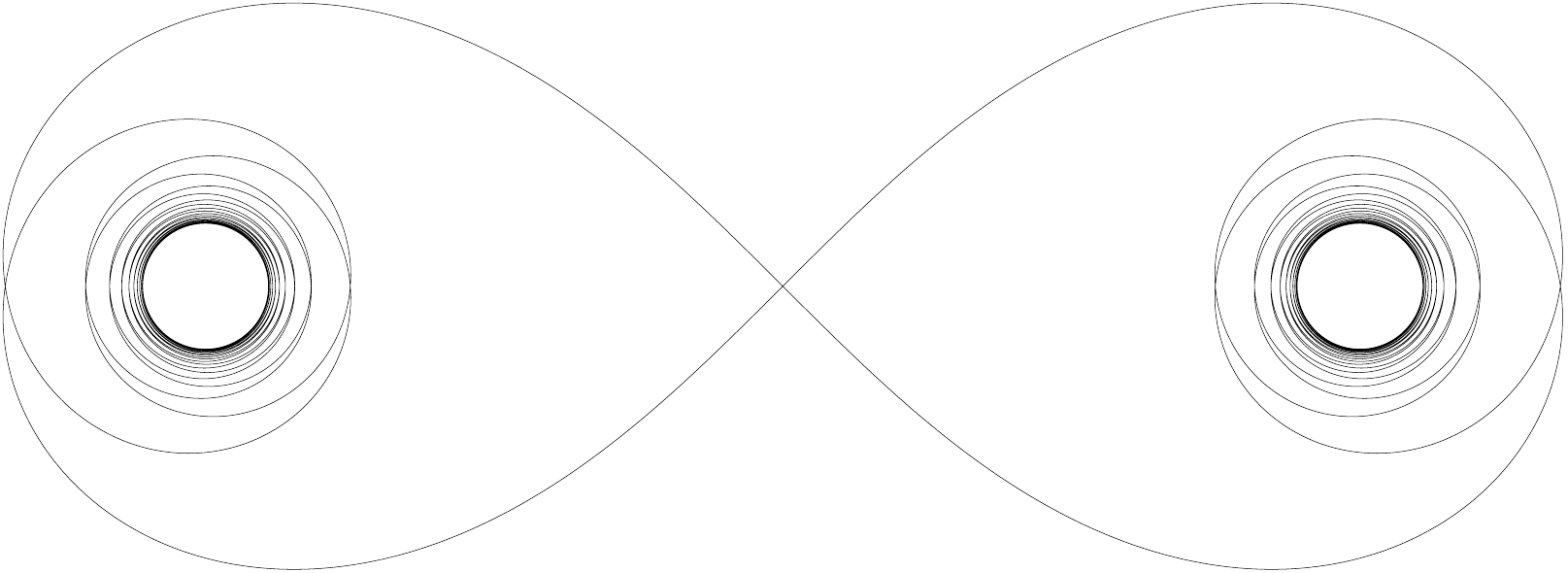}%
  \par\nointerlineskip%

  \caption{Selected lemniscate-type critical point candidates for \(J_1\).  The plots are ordered by their index in the numerical scan, and thus have increasing energy, from top-left to bottom-right.}
  \label{fig:selected-lemniscate-type-candidates}
\end{figure}

\begin{figure}[p]
  \centering
  \setlength{\parindent}{0pt}%
  \setlength{\parskip}{0pt}%
  \noindent%
  \SelectedCriticalPlot{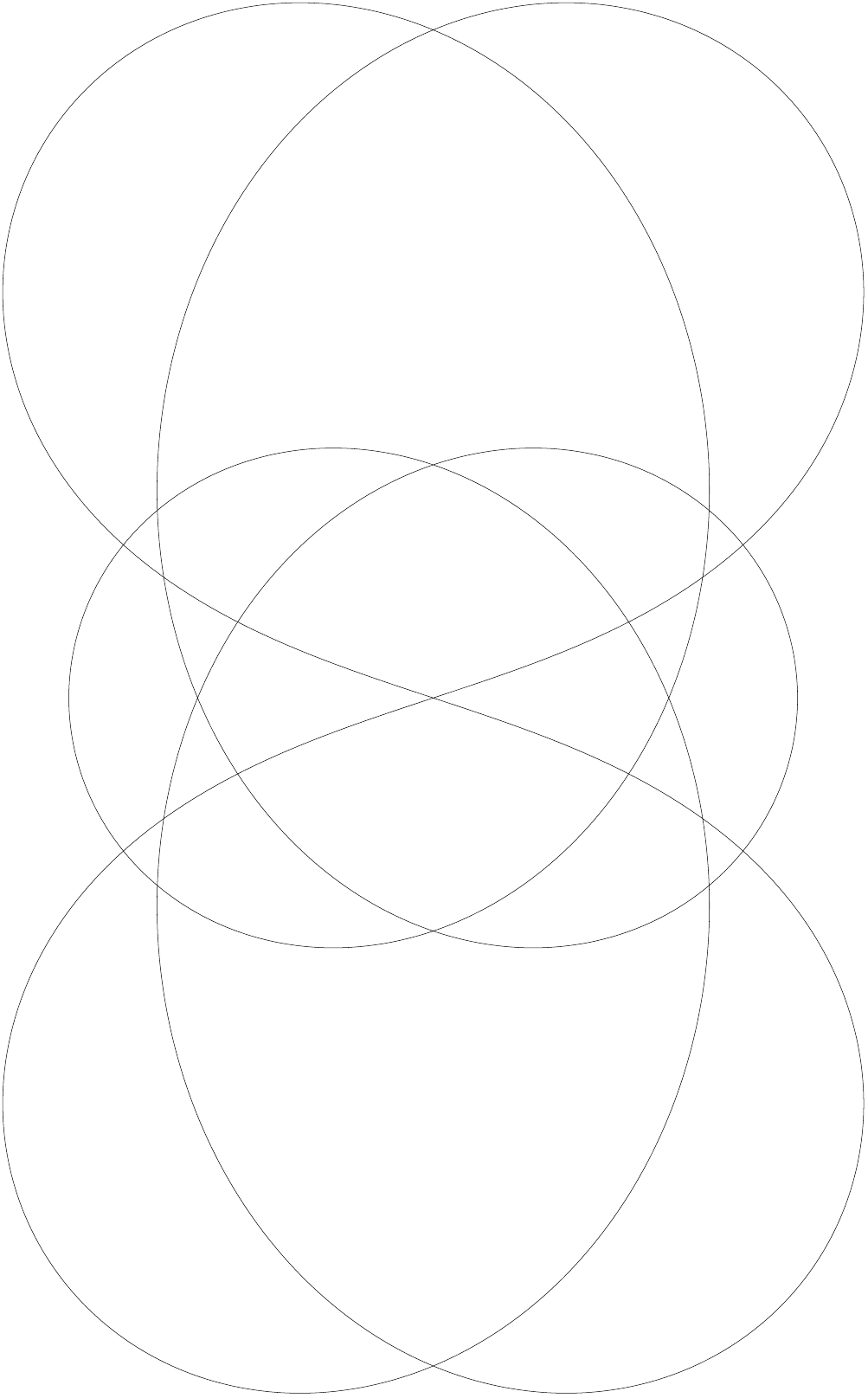}%
  \SelectedCriticalPlot{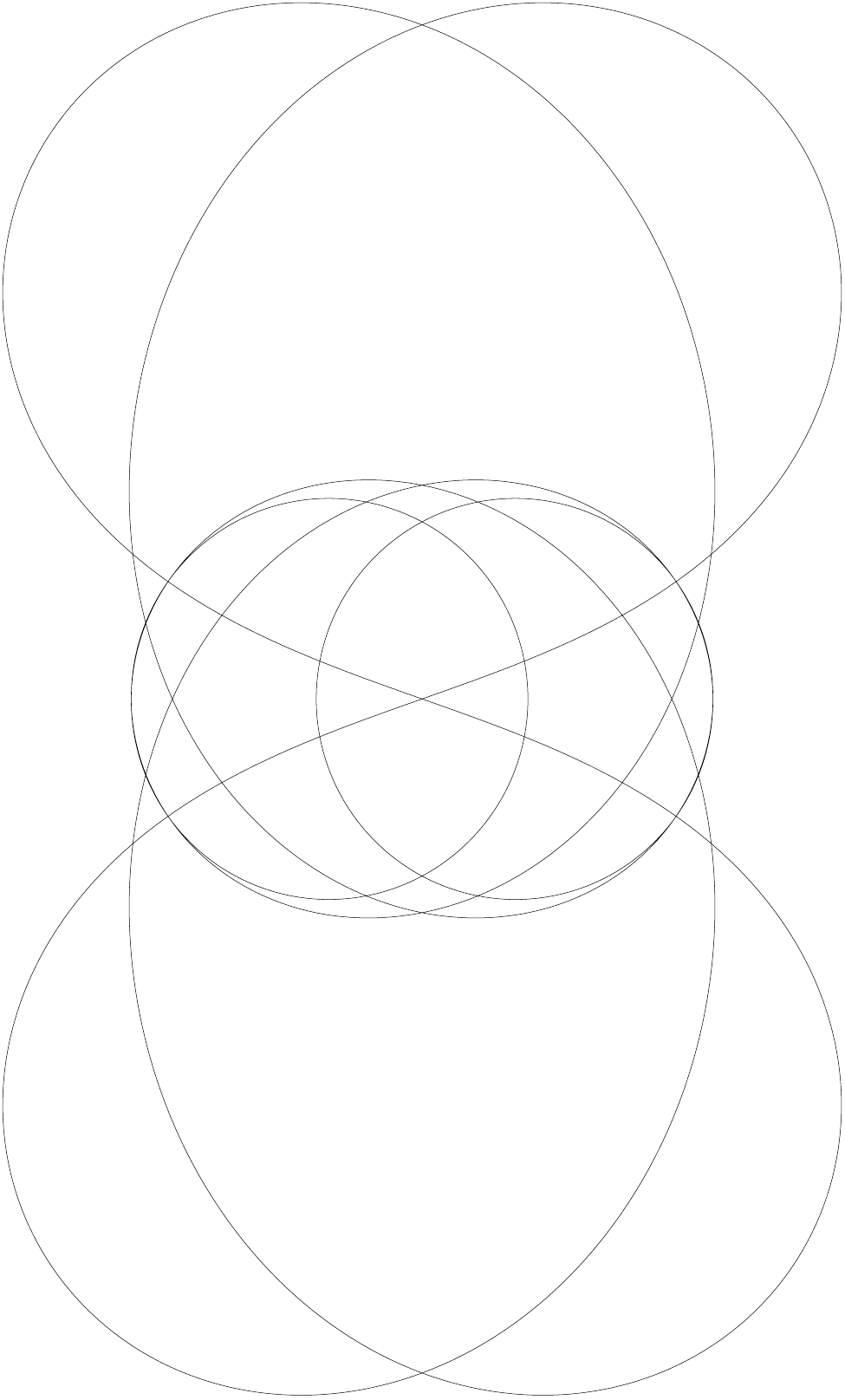}%
  \SelectedCriticalPlot{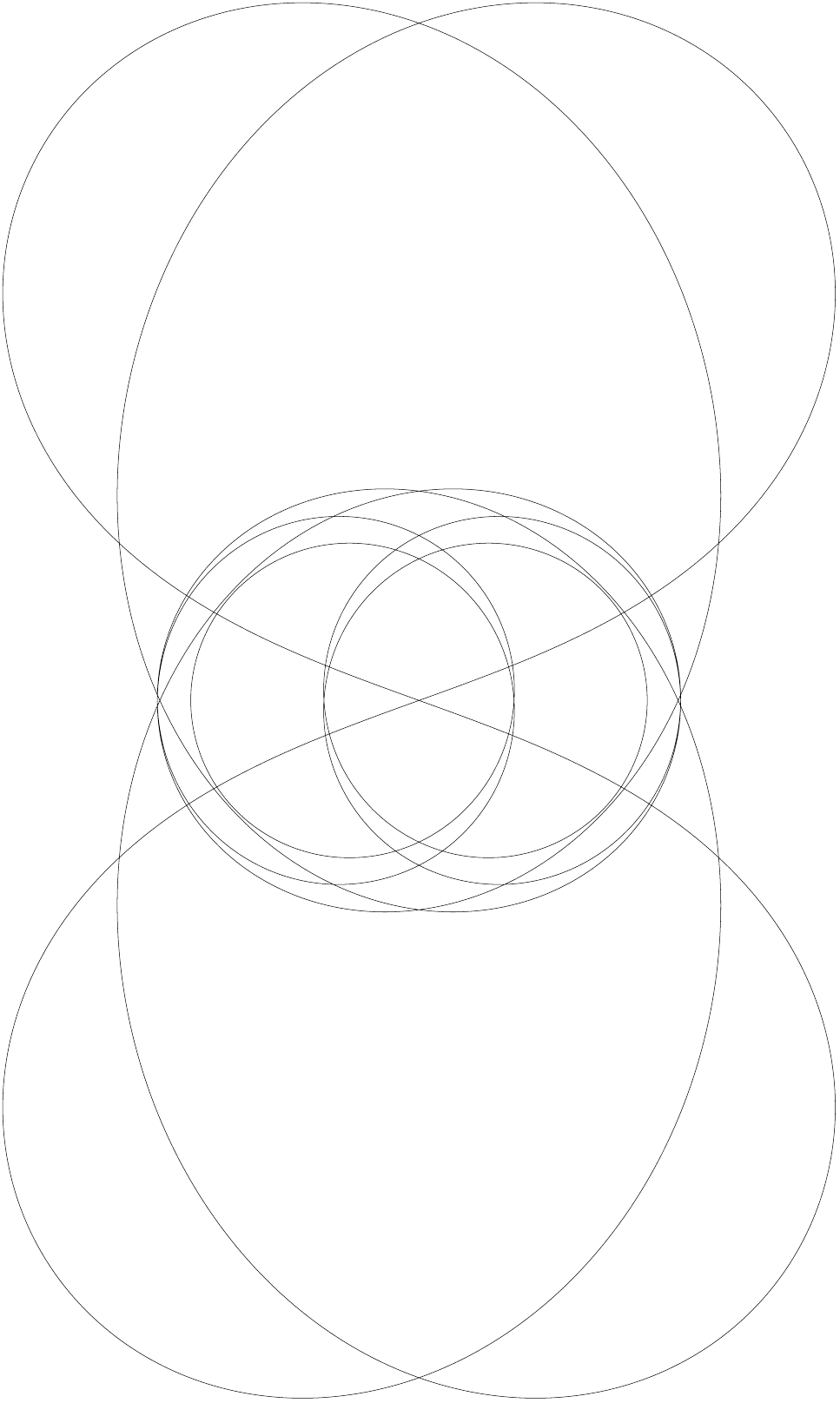}%
  \par\nointerlineskip%
  \SelectedCriticalPlot{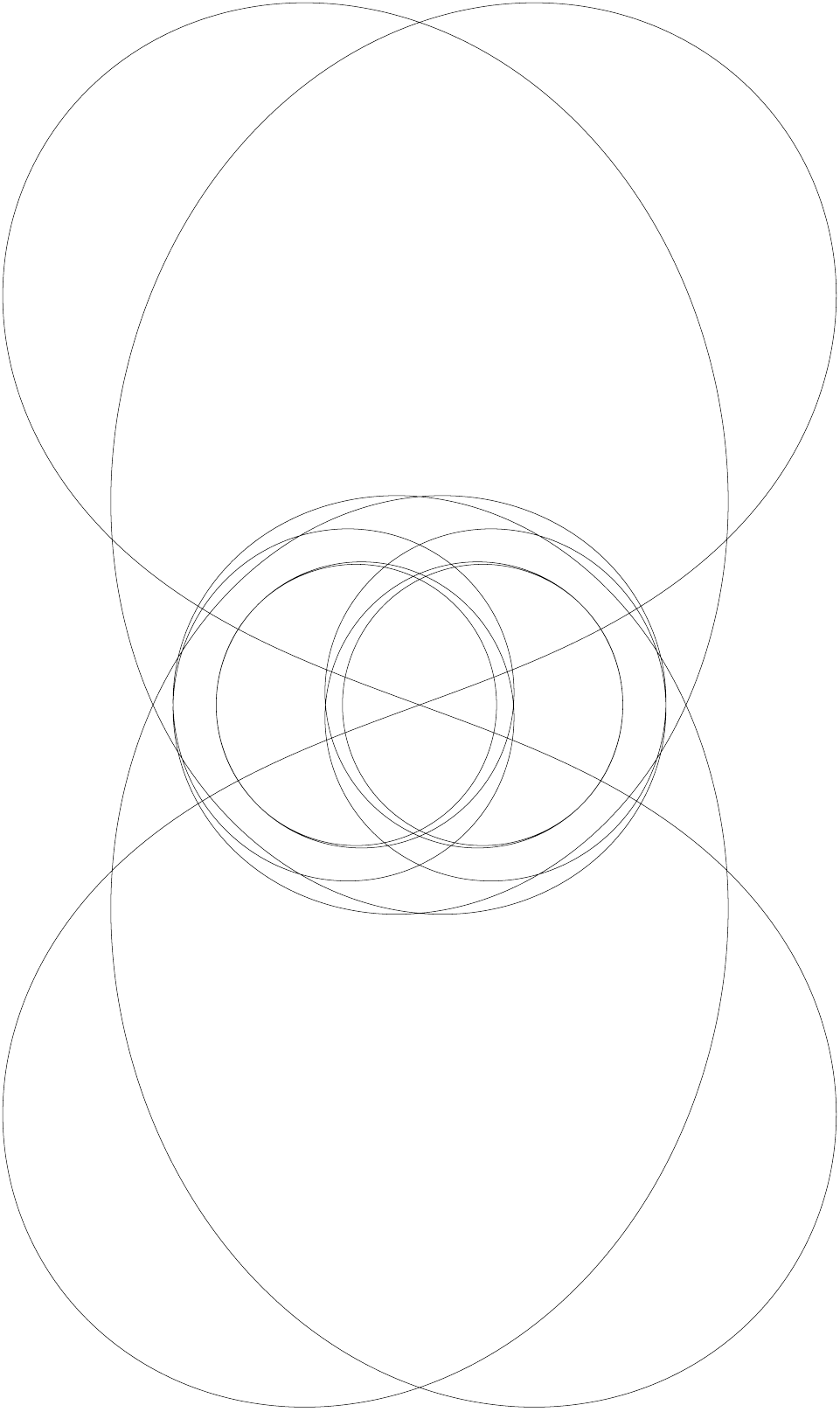}%
  \SelectedCriticalPlot{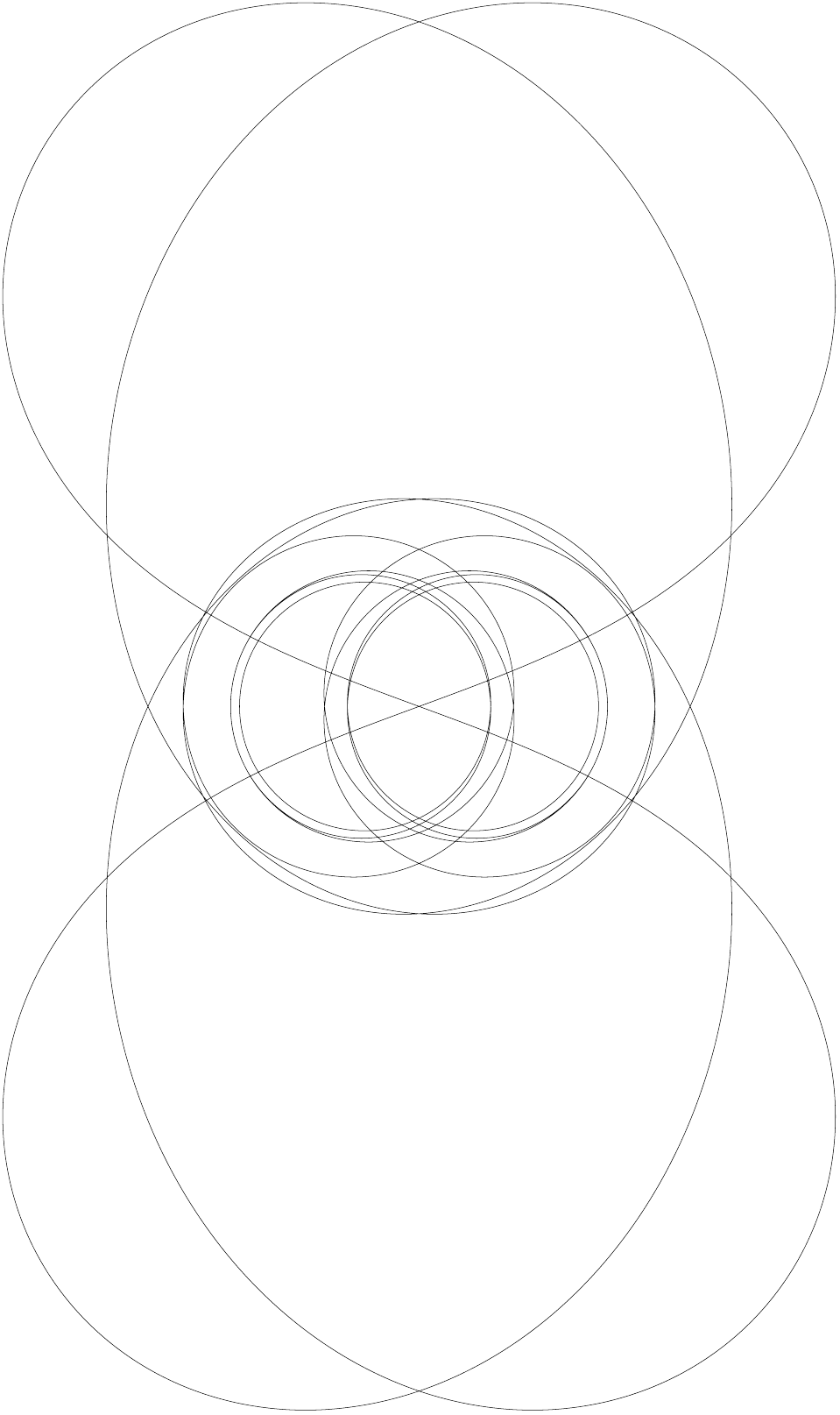}%
  \SelectedCriticalPlot{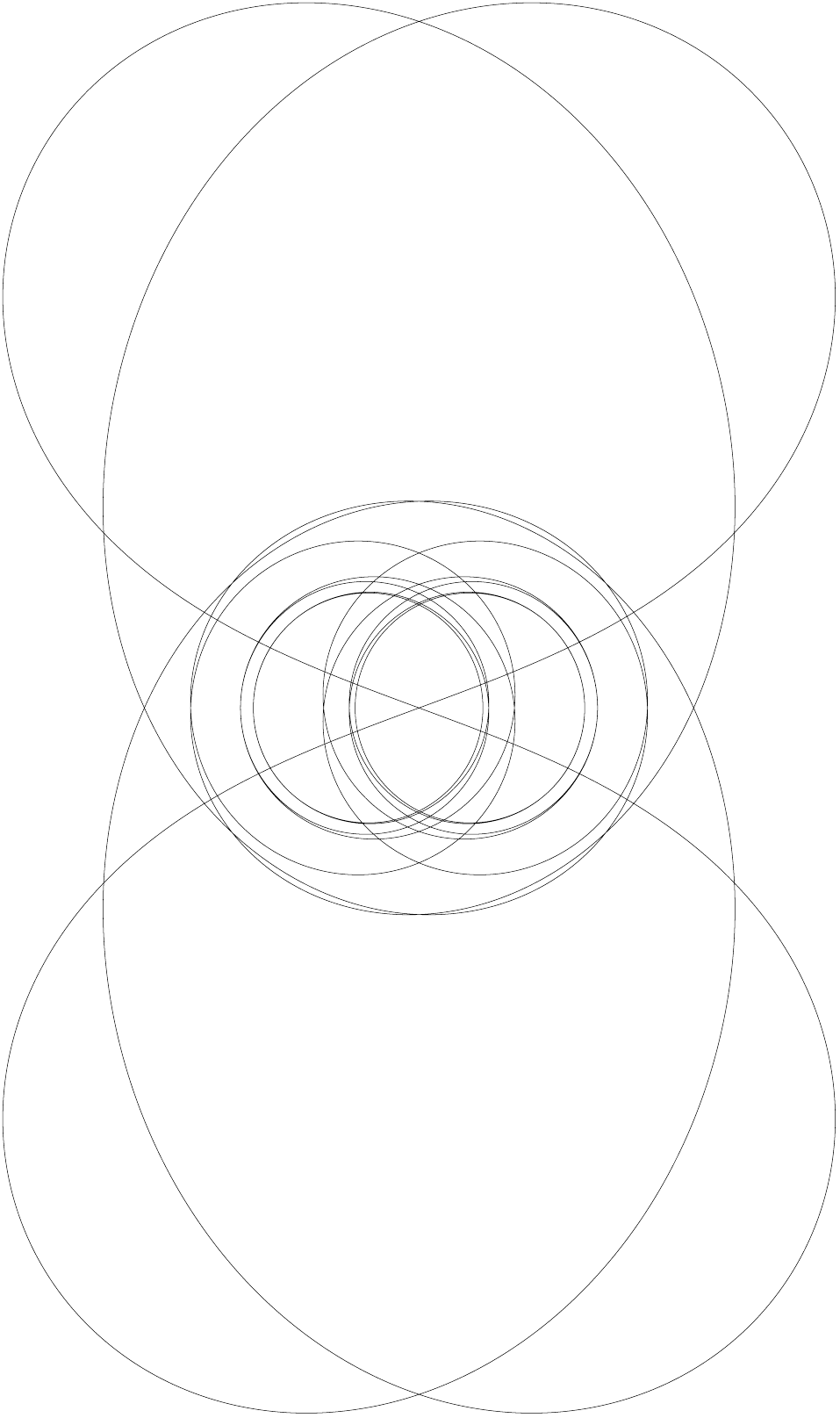}%
  \par\nointerlineskip%
  \SelectedCriticalPlot{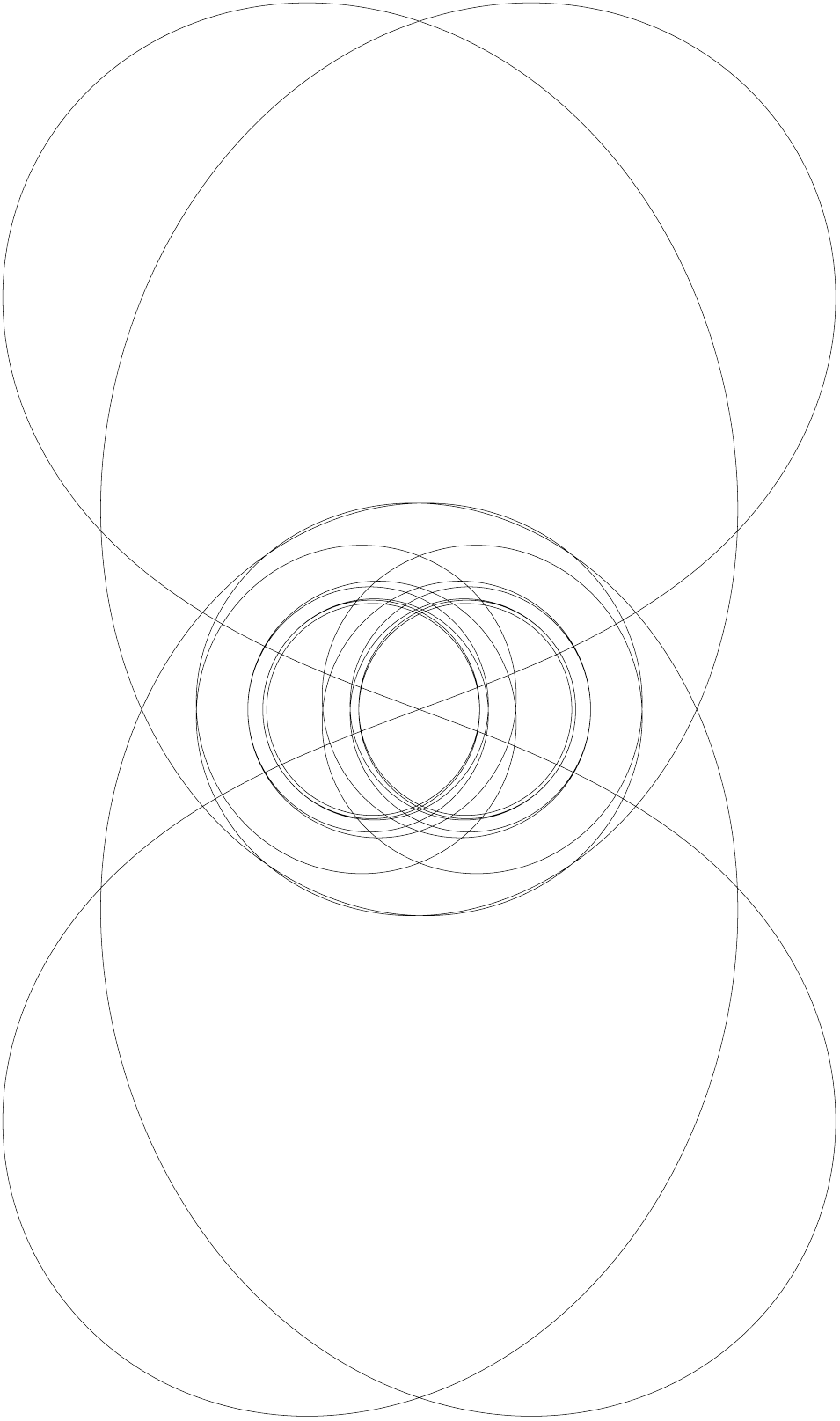}%
  \SelectedCriticalPlot{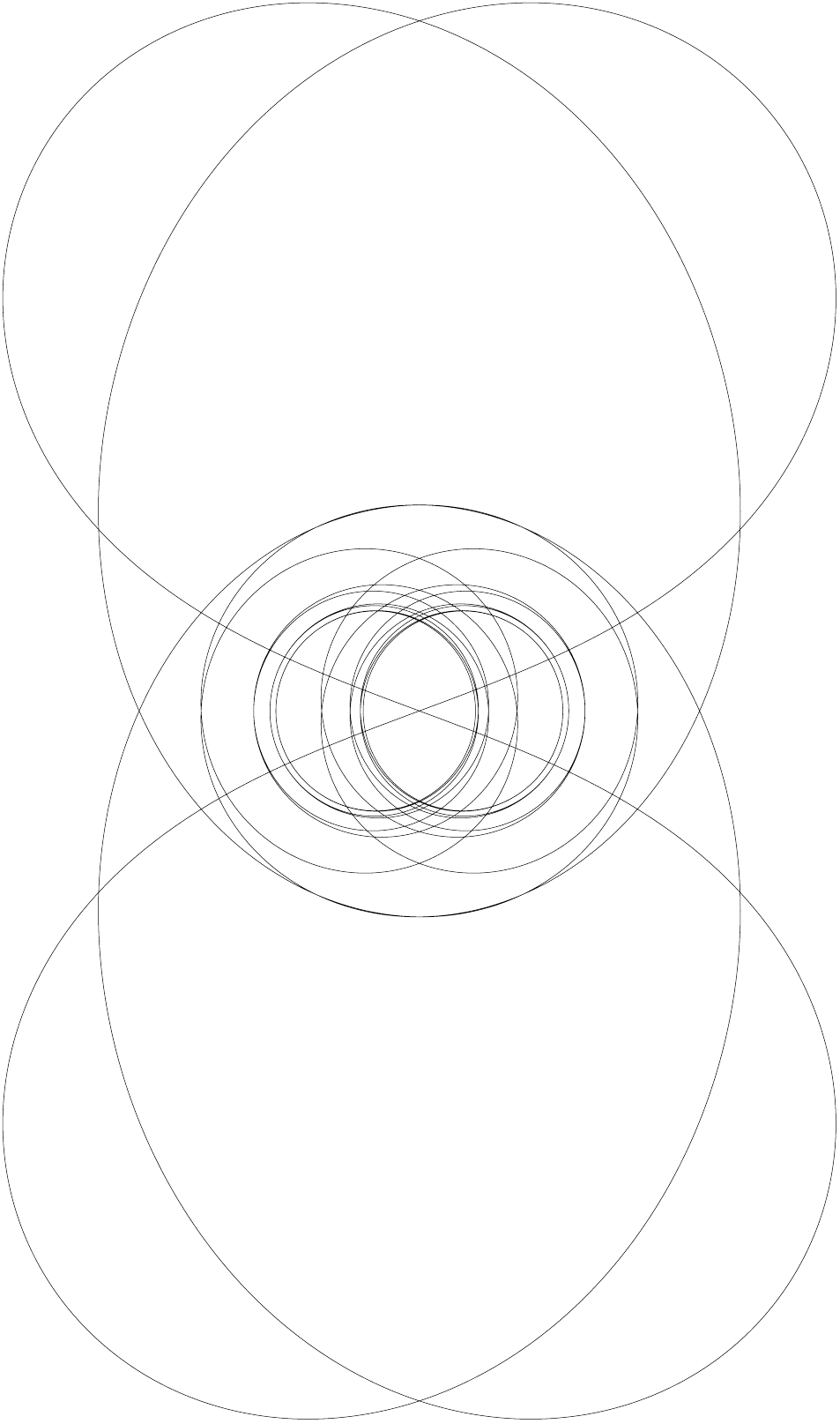}%
  \SelectedCriticalPlot{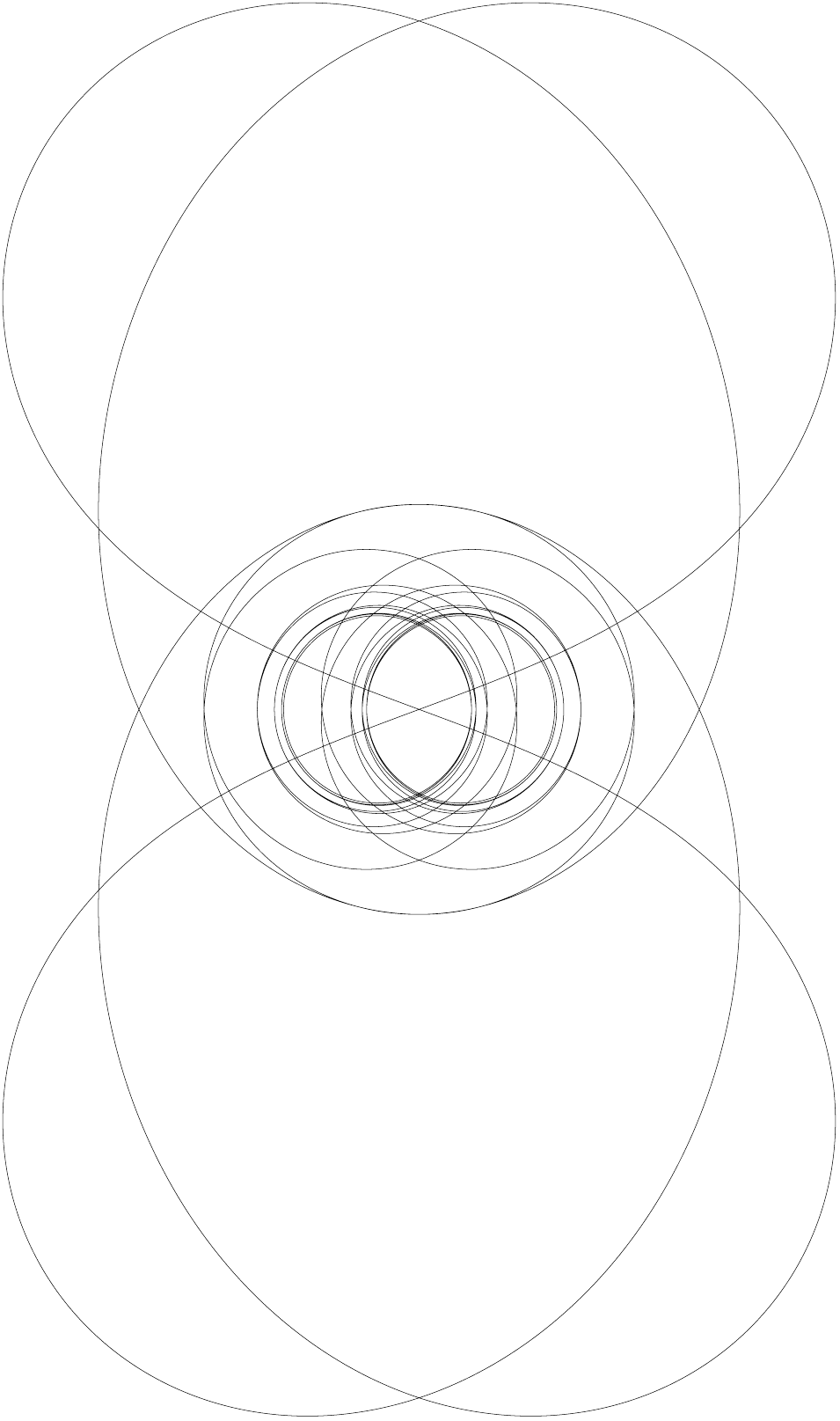}%
  \par\nointerlineskip%
  \SelectedCriticalPlot{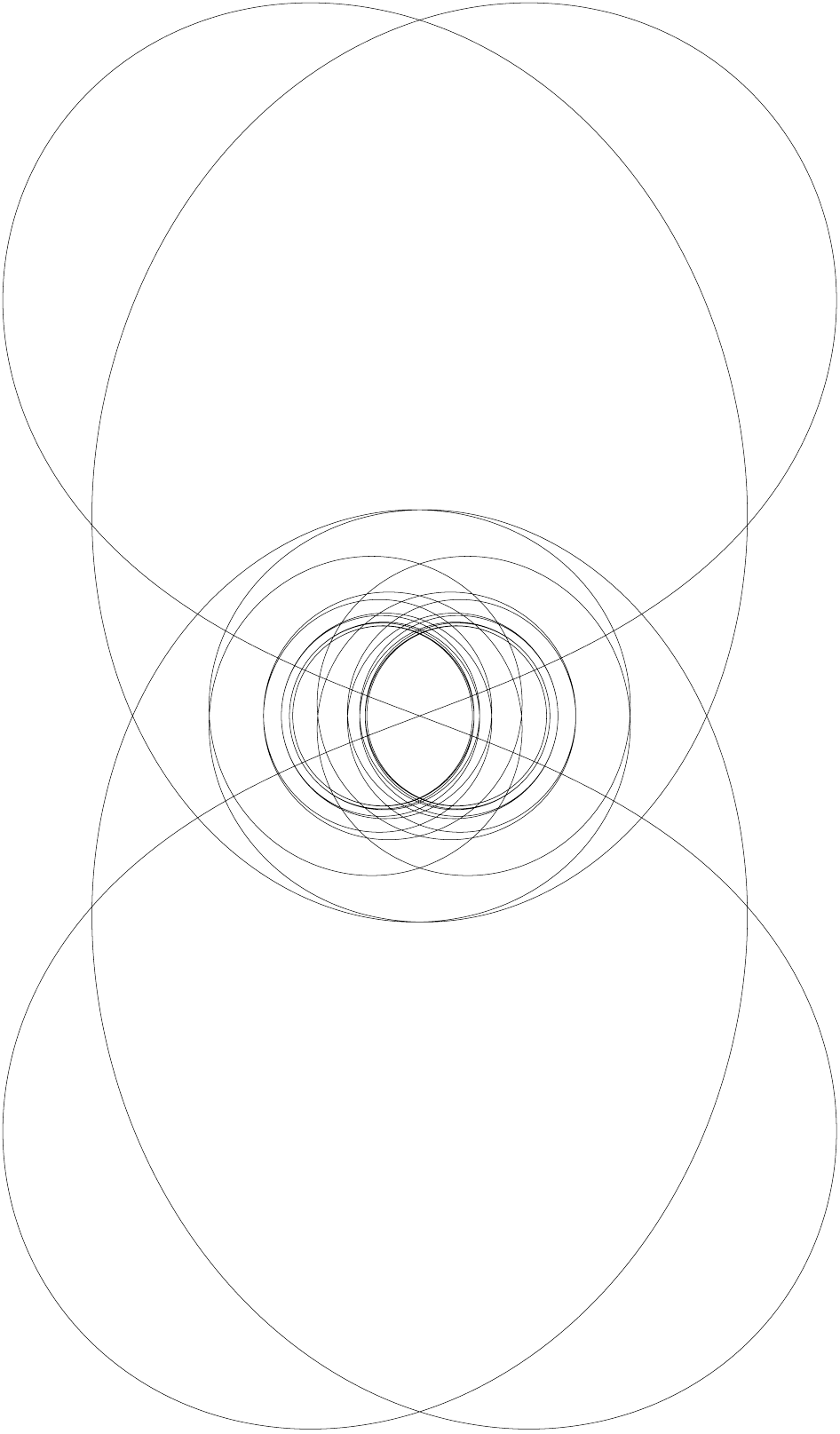}%
  \SelectedCriticalPlot{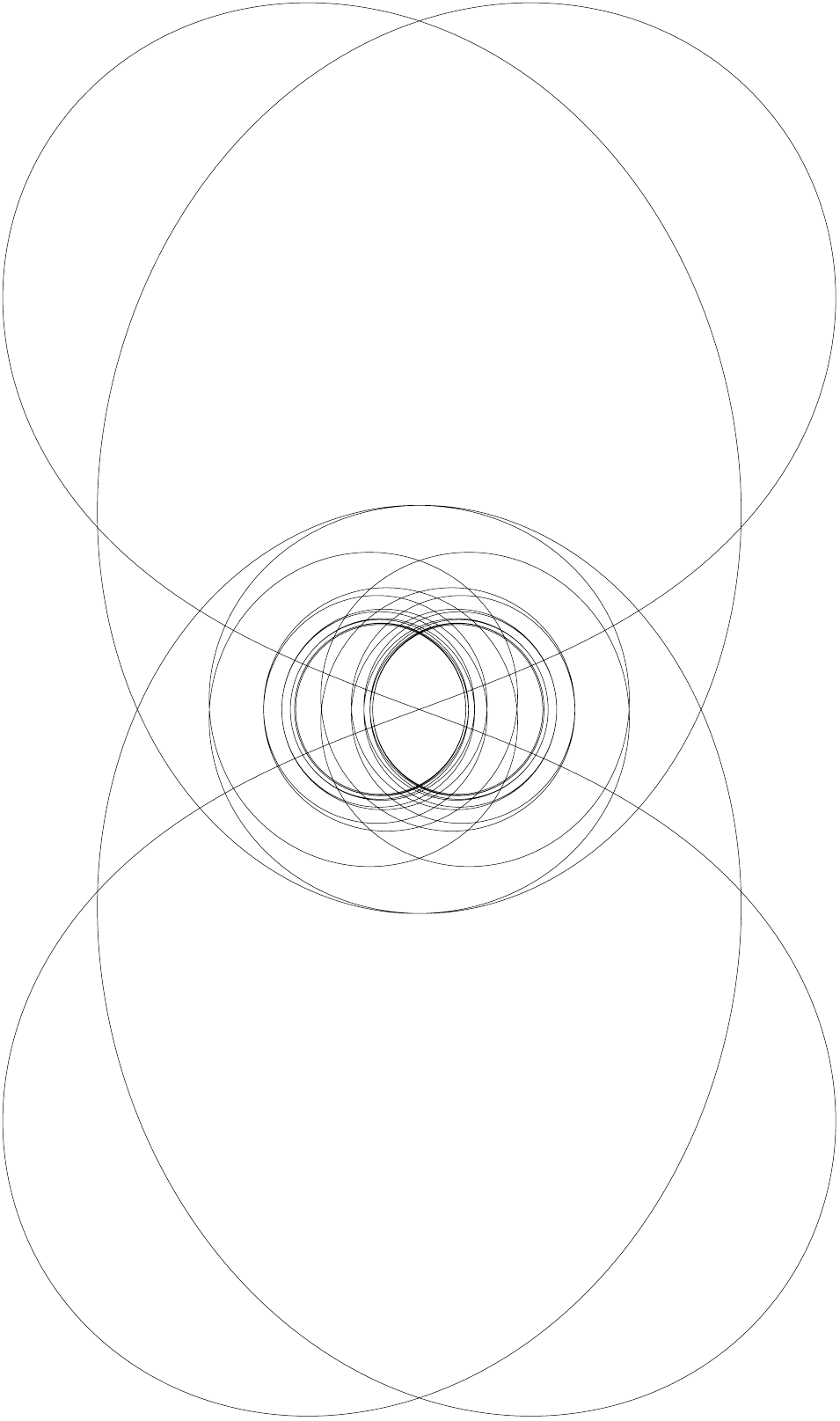}%
  \SelectedCriticalPlot{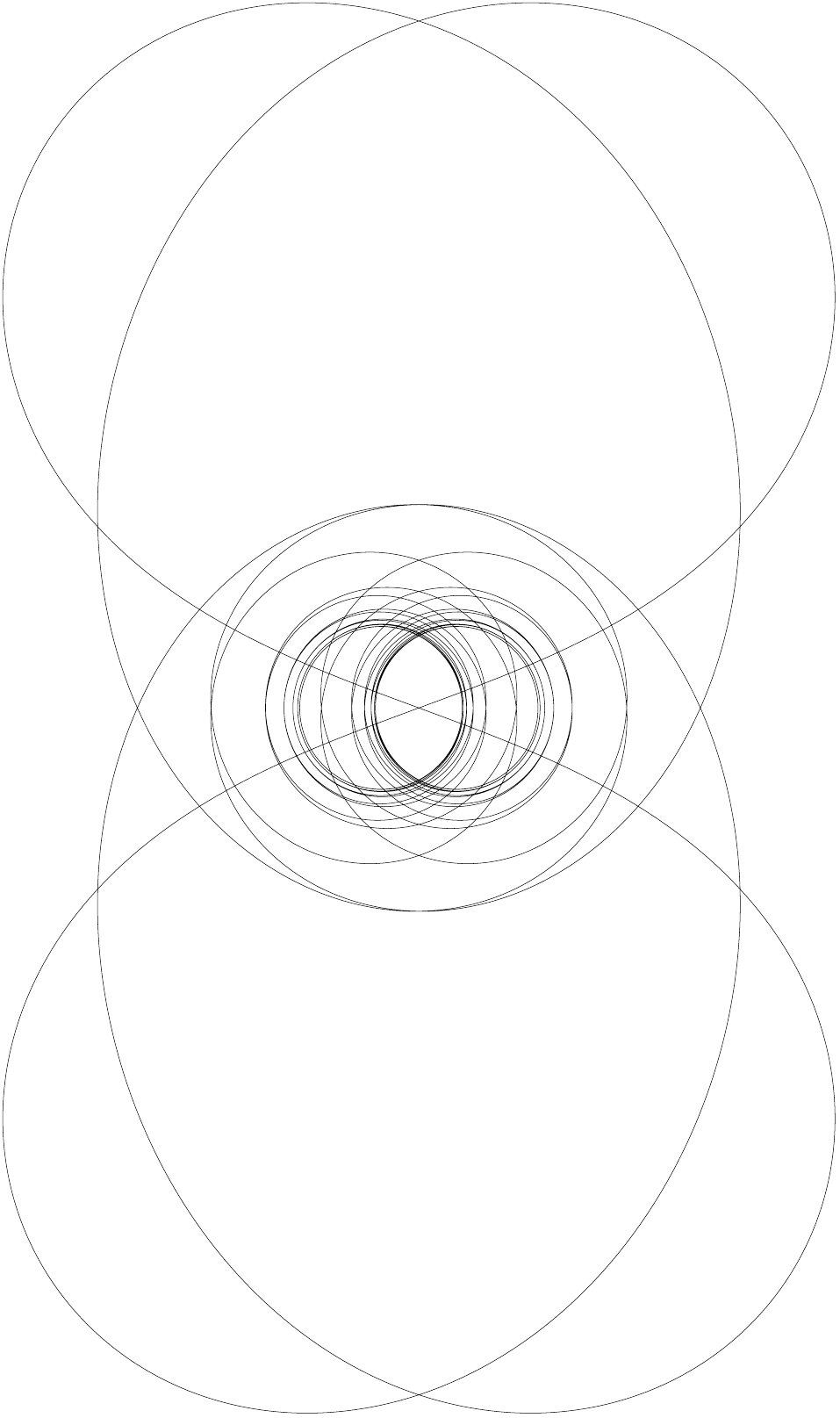}%
  \par\nointerlineskip%
  \SelectedCriticalPlot{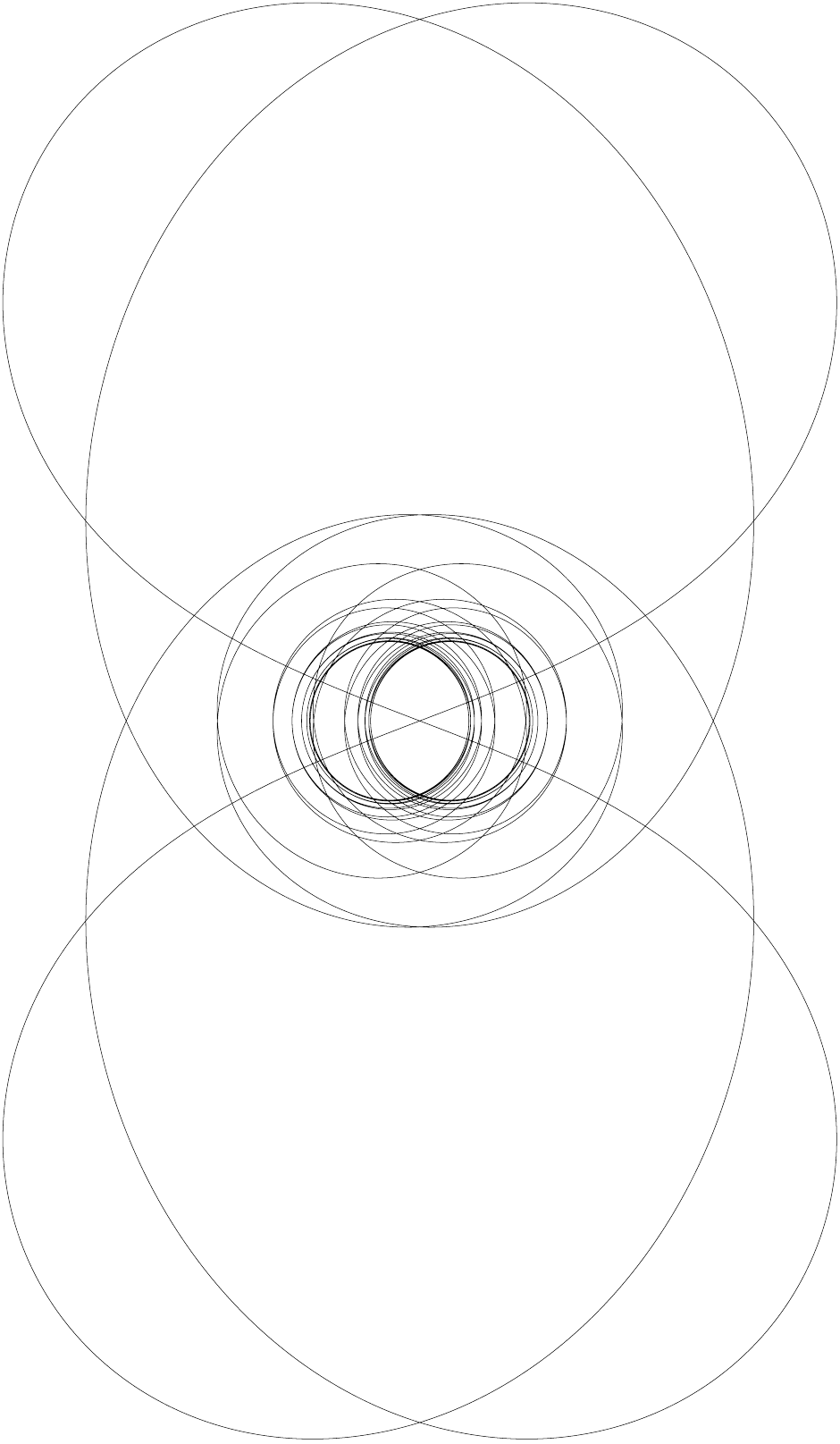}%
  \SelectedCriticalPlot{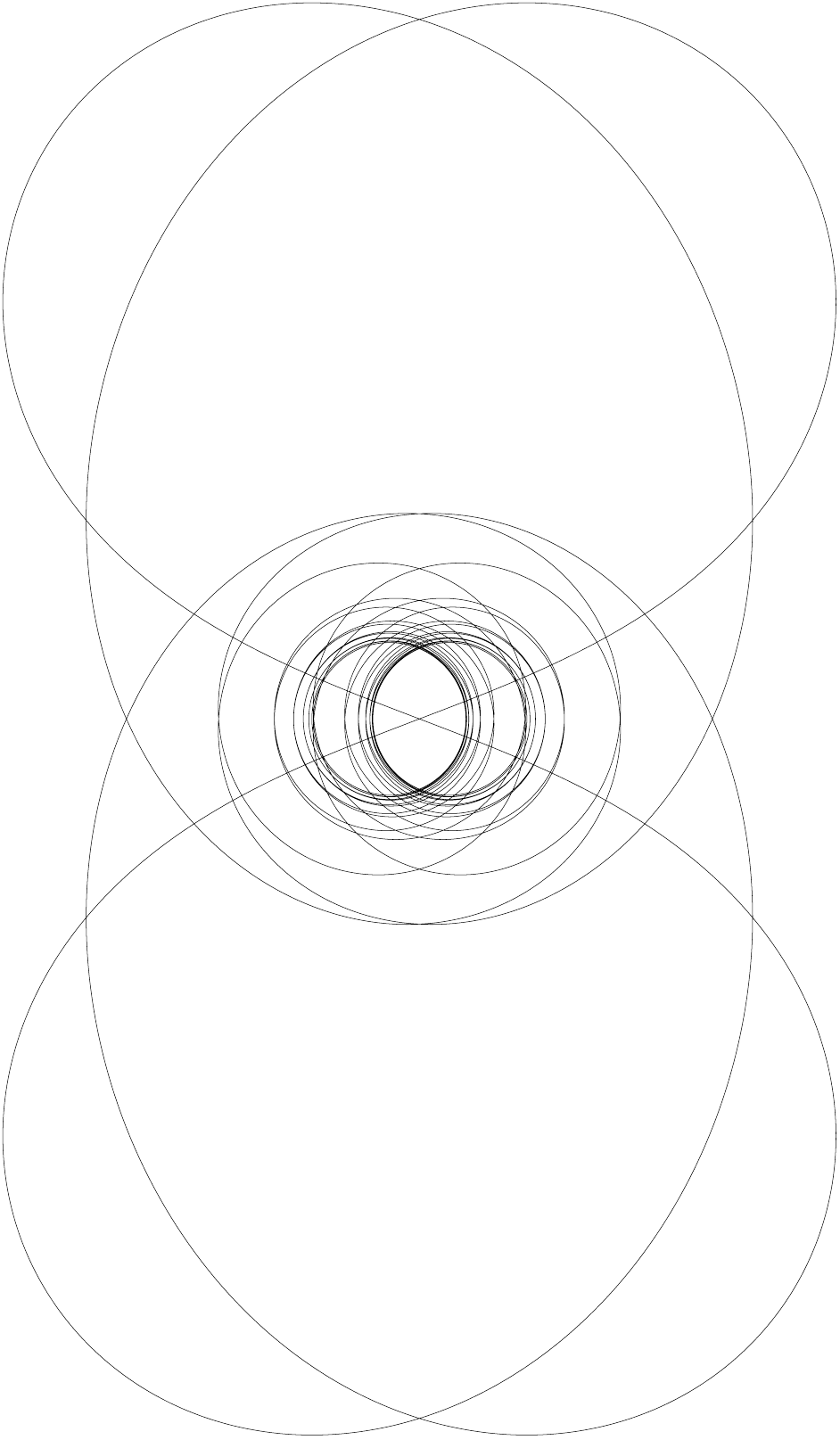}%
  \SelectedCriticalPlot{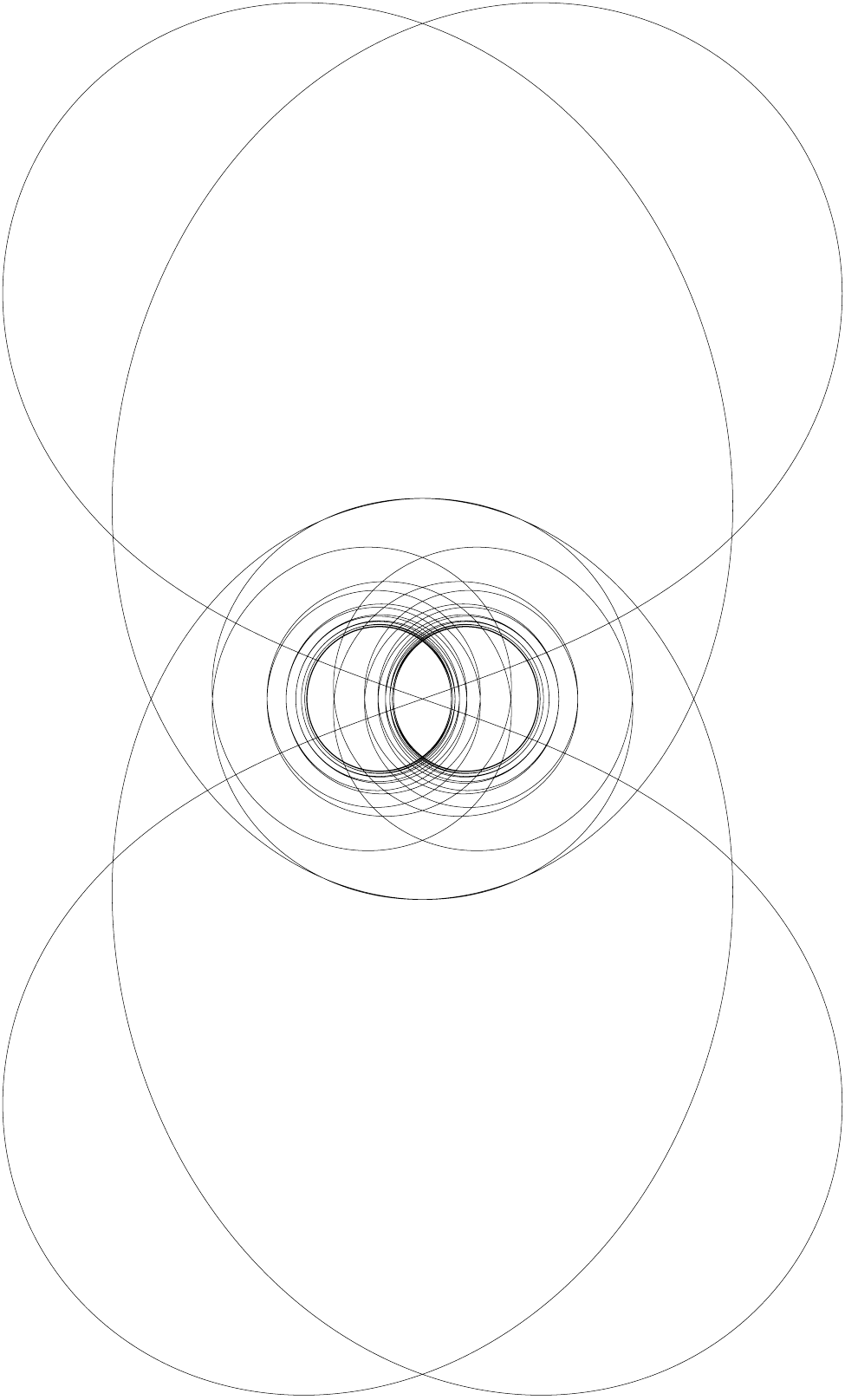}%
  \par\nointerlineskip%
  \SelectedCriticalPlot{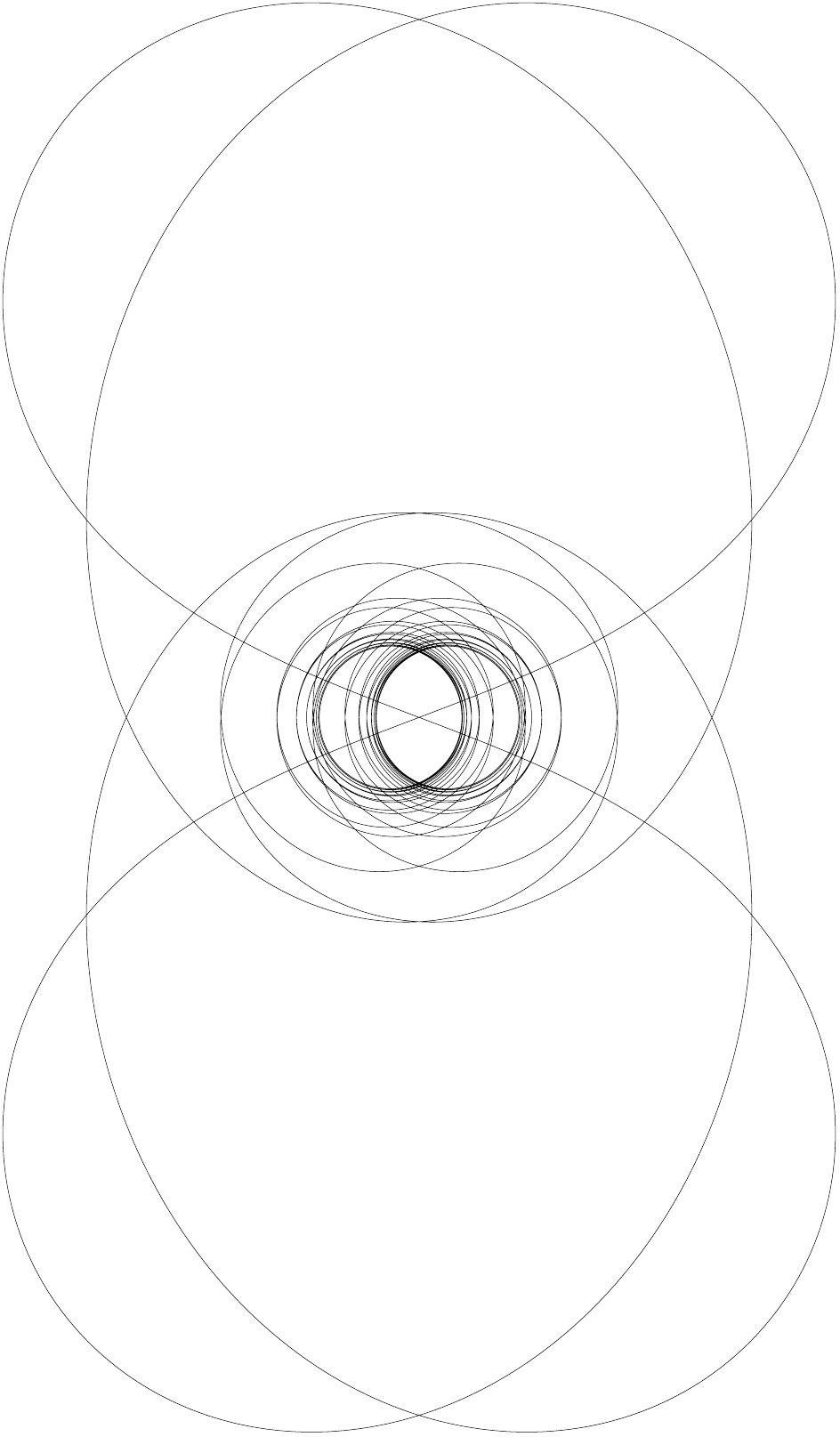}%
  \SelectedCriticalPlot{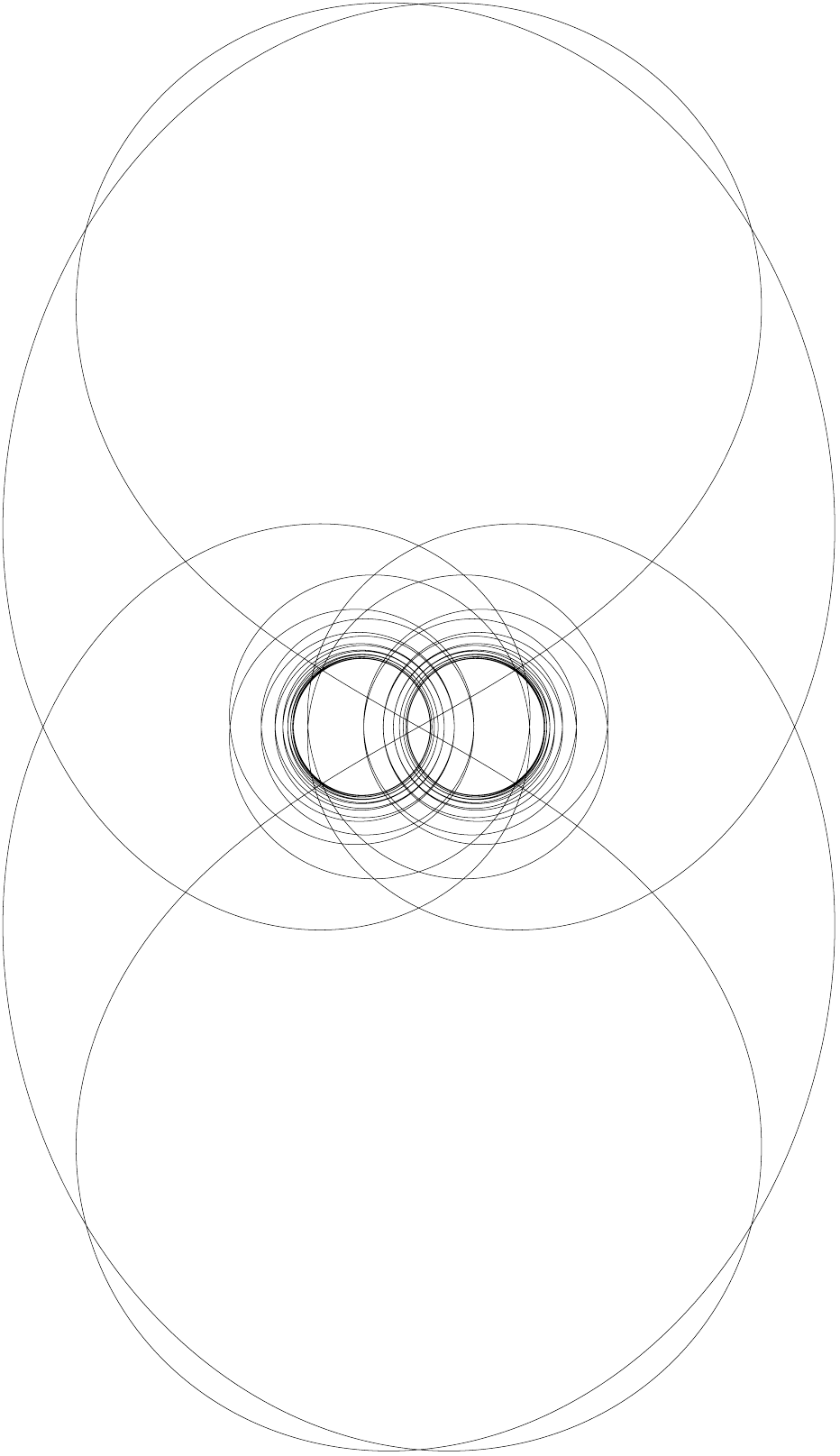}%
  \SelectedCriticalPlot{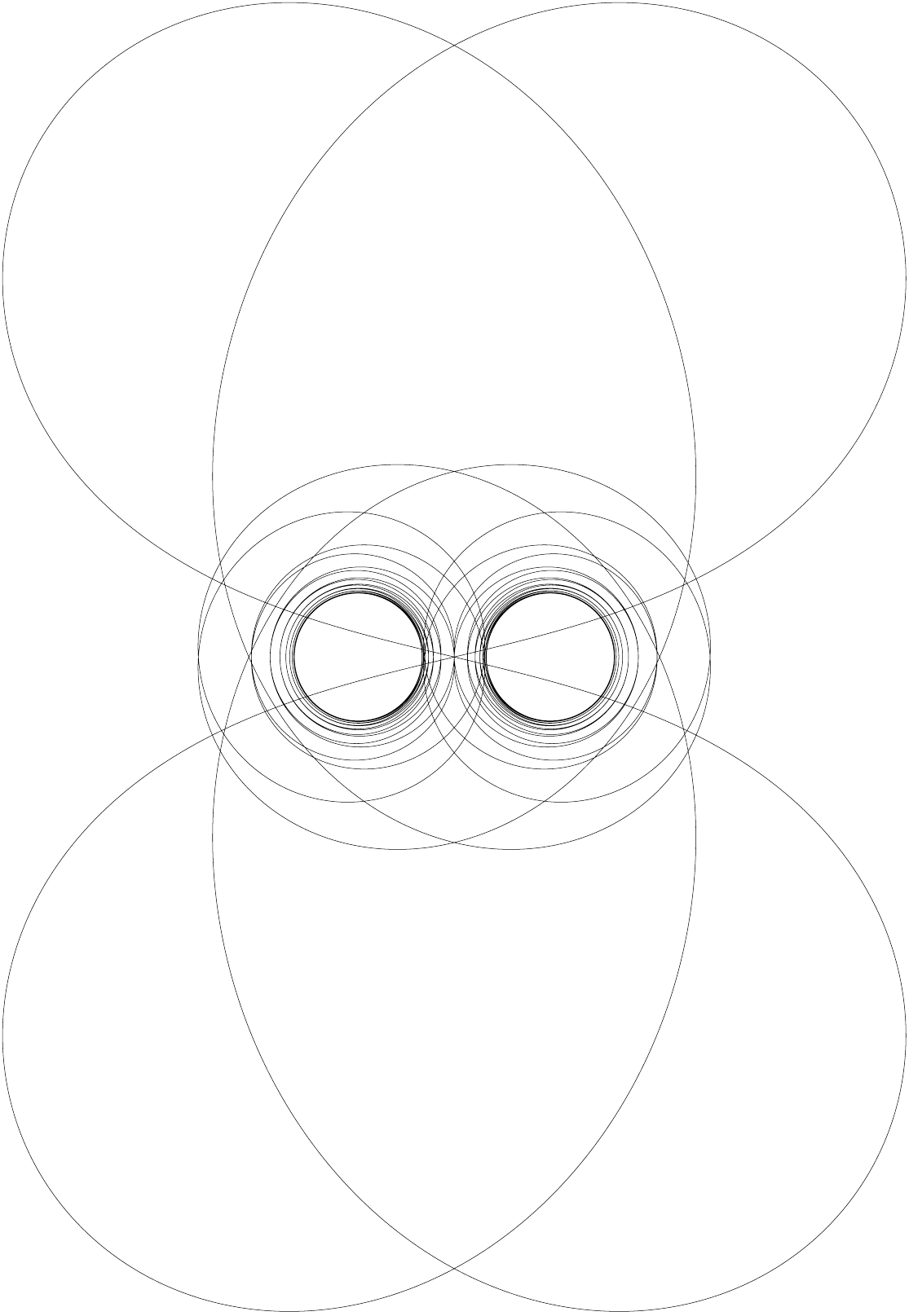}%
  \par\nointerlineskip%
  \SelectedCriticalPlot{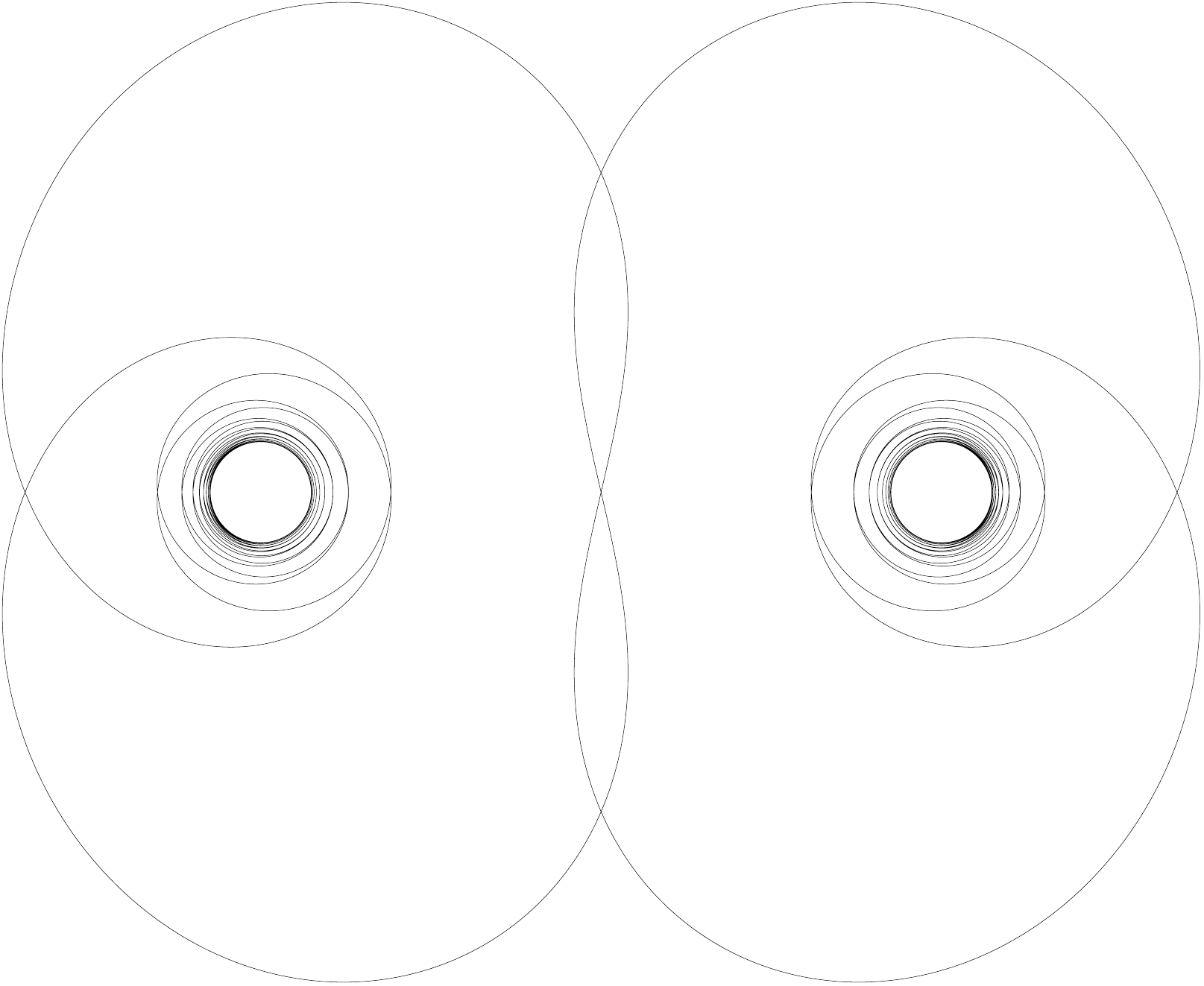}%
  \caption{Selected butterfly-type critical point candidates for \(J_1\).    The plots are ordered by their index in the numerical scan, and thus have increasing energy, from top-left to bottom-right.}
  \label{fig:selected-butterfly-type-candidates}
\end{figure}
\clearpage
\begin{longtable}{r l r r r r}
\caption{Diagnostics for the  critical points illustrated in Figures~\ref{fig:selected-first-pair}, \ref{fig:selected-lemniscate-type-candidates}, and~\ref{fig:selected-butterfly-type-candidates}. 
}\label{tab:selected-critical-point-data}\\
\toprule
$j$ & Type & $\Phi/\pi$ & quarter-arc angle $\Phi$ & total length $L[\Gamma]$ & energy $J_1[\Gamma]$\\
\midrule
\endfirsthead
\toprule
$j$ & Type & $\Phi/\pi$ & quarter-arc angle $\Phi$ & total length $L[\Gamma]$ & energy $J_1[\Gamma]$\\
\midrule
\endhead
\midrule
\multicolumn{6}{r}{\emph{Continued on next page}}\\
\endfoot
\bottomrule
\endlastfoot
\multicolumn{6}{l}{\textbf{Lemniscate-type}}\\
1 & Lemniscate-type & 0.760574 & 2.389413 & 9.063416 & 12.084554\\
3 & Lemniscate-type & 1.755809 & 5.516035 & 13.707291 & 18.276388\\
5 & Lemniscate-type & 2.754633 & 8.653936 & 17.153077 & 22.870769\\
7 & Lemniscate-type & 3.754014 & 11.793581 & 20.017515 & 26.690020\\
9 & Lemniscate-type & 4.753601 & 14.933879 & 22.521744 & 30.028992\\
11 & Lemniscate-type & 5.753298 & 18.074518 & 24.774702 & 33.032936\\
13 & Lemniscate-type & 6.753061 & 21.215366 & 26.839526 & 35.786035\\
15 & Lemniscate-type & 7.752869 & 24.356356 & 28.756666 & 38.342222\\
17 & Lemniscate-type & 8.752709 & 27.497446 & 30.553878 & 40.738505\\
19 & Lemniscate-type & 9.752573 & 30.638612 & 32.251185 & 43.001580\\
21 & Lemniscate-type & 10.752456 & 33.779836 & 33.863590 & 45.151454\\
23 & Lemniscate-type & 11.752353 & 36.921106 & 35.402683 & 47.203578\\
25 & Lemniscate-type & 12.752263 & 40.062415 & 36.877634 & 49.170179\\
27 & Lemniscate-type & 13.752180 & 43.203748 & 38.295849 & 51.061132\\
29 & Lemniscate-type & 14.752108 & 46.345115 & 39.663409 & 52.884546\\
31 & Lemniscate-type & 15.752043 & 49.486503 & 40.985382 & 54.647175\\
33 & Lemniscate-type & 16.751979 & 52.627895 & 42.266041 & 56.354722\\
35 & Lemniscate-type & 17.751908 & 55.769265 & 43.509035 & 58.012046\\
37 & Lemniscate-type & 18.751900 & 58.910830 & 44.717501 & 59.623335\\
39 & Lemniscate-type & 19.751997 & 62.052730 & 45.894172 & 61.192230\\
40 & Lemniscate-type & 20.751560 & 65.192948 & 47.041426 & 62.721902\\
42 & Lemniscate-type & 21.749793 & 68.328988 & 48.162125 & 64.216166\\
\addlinespace
\multicolumn{6}{l}{\textbf{Butterfly-type}}\\
2 & Butterfly-type & 1.442710 & 4.532409 & 17.152740 & 22.870320\\
4 & Butterfly-type & 2.398552 & 7.535272 & 22.823664 & 30.431552\\
6 & Butterfly-type & 3.391938 & 10.656087 & 26.850900 & 35.801199\\
8 & Butterfly-type & 4.389726 & 13.790729 & 30.049660 & 40.066213\\
10 & Butterfly-type & 5.388605 & 16.928803 & 32.773241 & 43.697655\\
12 & Butterfly-type & 6.387911 & 20.068214 & 35.183351 & 46.911134\\
14 & Butterfly-type & 7.387433 & 23.208307 & 37.367590 & 49.823453\\
16 & Butterfly-type & 8.387085 & 26.348804 & 39.379323 & 52.505764\\
18 & Butterfly-type & 9.386687 & 29.489148 & 41.253810 & 55.005080\\
20 & Butterfly-type & 10.387074 & 32.631955 & 43.015747 & 57.354329\\
22 & Butterfly-type & 11.385131 & 35.767443 & 44.683198 & 59.577598\\
24 & Butterfly-type & 12.387204 & 38.915548 & 46.269991 & 61.693321\\
26 & Butterfly-type & 13.387508 & 42.058096 & 47.786684 & 63.715579\\
28 & Butterfly-type & 14.383700 & 45.187726 & 49.241789 & 65.655719\\
30 & Butterfly-type & 15.384023 & 48.330333 & 50.642486 & 67.523315\\
32 & Butterfly-type & 16.392696 & 51.499173 & 51.992532 & 69.323376\\
34 & Butterfly-type & 17.384397 & 54.614693 & 53.301977 & 71.069302\\
36 & Butterfly-type & 18.336099 & 57.604554 & 54.479664 & 72.639552\\
38 & Butterfly-type & 19.425451 & 61.026854 & 55.737865 & 74.317153\\
41 & Butterfly-type & 21.072001 & 66.199642 & 54.811710 & 73.082280\\
43 & Butterfly-type & 22.162572 & 69.625773 & 57.588237 & 76.784316\\
\end{longtable}

\newpage

\section{Auxiliary estimates and symmetry calculus}\label{sec:aux}

This section collects the analytic inequalities, scaling calculations, and reflection/gluing lemmata
that will be used in the paper.

\subsection{Scaling of $J_\lambda$ and the virial identity}\label{subsec:aux-scaling}

The functional $J_\lambda$ is not scale invariant; the competition between the curvature-derivative term and length
selects a preferred scale.  We will use this both conceptually and quantitatively.

\begin{lemma}[Scaling law]\label{lem:scaling}
Let $\gamma$ be a unit-speed curve of length $L$, and let $\tilde\gamma=r\gamma$ be its dilation by $r>0$.
Then
\begin{equation}\label{eq:scaling-law}
L[\tilde\gamma]=r\,L[\gamma],
\qquad
\int_{\tilde\gamma} k_s^2\,ds = r^{-3}\int_\gamma k_s^2\,ds,
\qquad
J_\lambda[\tilde\gamma]=r^{-3}\Big(\frac12\int_\gamma k_s^2\,ds\Big)+\lambda\,r\,L[\gamma].
\end{equation}
\end{lemma}

\begin{proof}
Under $x\mapsto r x$, arclength scales as $ds\mapsto r\,ds$.
Curvature scales as $k\mapsto r^{-1}k$, hence $k_s=dk/ds$ scales as
\[
k_s \mapsto \frac{d(r^{-1}k)}{r\,ds}=r^{-2}k_s.
\]
Therefore $k_s^2\,ds$ scales as $(r^{-4})(r\,ds)=r^{-3}(k_s^2\,ds)$, proving \eqref{eq:scaling-law}.
\end{proof}

\begin{corollary}[Virial identity at smooth critical points]\label{cor:virial}
If $\gamma$ is a smooth critical point of $J_\lambda$ with respect to ambient dilations then
\begin{equation}\label{eq:virial}
\frac32\int_\gamma k_s^2\,ds=\lambda\,L[\gamma].
\end{equation}
\end{corollary}

\begin{proof}
By \eqref{eq:scaling-law},
\[
J_\lambda[r\gamma]=r^{-3}\Big(\frac12\int_\gamma k_s^2\,ds\Big)+\lambda\,r\,L[\gamma].
\]
Differentiate at $r=1$ and set the derivative to $0$:
\[
0=-3\Big(\frac12\int_\gamma k_s^2\,ds\Big)+\lambda\,L[\gamma],
\]
which is exactly \eqref{eq:virial}.
\end{proof}

\subsection{A sharp turning-energy inequality}\label{subsec:aux-coercive}

A key estimate is that, once we impose a seam condition $k(L)=0$ at the end of an arc,
the turning $\Phi=\int_0^L k\,ds$ controls $\|k_s\|_{L^2}$ sharply.  This yields coercivity and
uniform length bounds for minimising sequences.

\begin{lemma}[Sharp turning controls $\|k_s\|_{L^2}$]\label{lem:turning-controls-ks}
Let $L>0$ and let $k\in H^1(0,L)$ satisfy the endpoint seam constraint $k(L)=0$.
Define the \emph{unwrapped turning}
\[
\Phi:=\int_0^L k(s)\,ds.
\]
Then
\begin{equation}\label{eq:turning-ineq}
\int_0^L k_s^2\,ds \;\ge\; \frac{3\Phi^2}{L^3}.
\end{equation}
Moreover, the constant $3$ is sharp, and equality holds in \eqref{eq:turning-ineq} if and only if
\begin{equation}\label{eq:equality-ks}
k_s(s)=-\frac{3\Phi}{L^3}\,s\quad\text{for a.e. }s\in(0,L),
\end{equation}
equivalently,
\begin{equation}\label{eq:equality-k}
k(s)=\frac{3\Phi}{2L^3}\,(L^2-s^2)\quad\text{for a.e. }s\in(0,L).
\end{equation}
\end{lemma}

\begin{proof}
Since $k(L)=0$ and $k\in H^1(0,L)$, for a.e.\ $s\in(0,L)$,
\[
k(s)=k(s)-k(L)=-\int_s^L k_s(t)\,dt.
\]
Integrating in $s$ and using Fubini,
\[
\Phi=\int_0^L k(s)\,ds
=-\int_0^L\int_s^L k_s(t)\,dt\,ds
=-\int_0^L t\,k_s(t)\,dt.
\]
By Cauchy-Schwarz,
\[
|\Phi|
\le \Big(\int_0^L t^2\,dt\Big)^{1/2}\Big(\int_0^L k_s(t)^2\,dt\Big)^{1/2}
=\Big(\frac{L^3}{3}\Big)^{1/2}\|k_s\|_{L^2(0,L)}.
\]
Squaring and rearranging yields \eqref{eq:turning-ineq}.

Equality holds in Cauchy-Schwarz if and only if $k_s(t)=\alpha t$ a.e.\ for some constant $\alpha$.
Imposing $k(L)=0$ gives
\[
k(s)=-\int_s^L \alpha t\,dt=\frac{\alpha}{2}(s^2-L^2),
\]
and integrating this expression over $[0,L]$ yields $\Phi=-\alpha L^3/3$, hence
$\alpha=-3\Phi/L^3$. This gives \eqref{eq:equality-ks} and \eqref{eq:equality-k}.
Conversely, any $k$ of the form \eqref{eq:equality-k} achieves equality in \eqref{eq:turning-ineq}.
\end{proof}

Note that we use the terminology `unwrapped' to describe $\Phi$ as this angle is lifted to the universal cover $\R$ of $\S$ and not $\S$ itself.

\begin{figure}[t]
  \centering
  \includegraphics[clip=true,trim=1cm 10cm 1cm 10cm,width=0.38\linewidth]{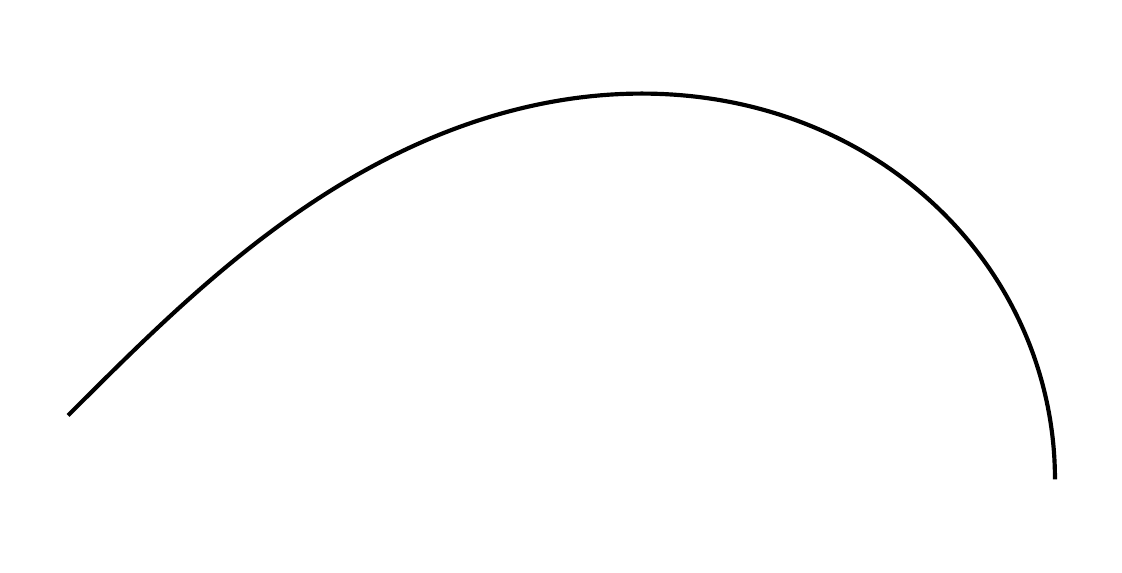}
  \caption{Illustration of the ``equality-case'' fundamental arc obtained by integrating the curvature profile
  \(
    k(s)=\frac{3\Phi}{2L^3}(L^2-s^2)
  \)
  on $s\in[0,L]$ (here plotted with representative parameters $L=1$, $\Phi=3\pi/4$, and initial tangent angle
  $\theta(0)=\pi/2$).  This parabolic curvature saturates the sharp turning inequality
  \(
    \int_0^L k_s^2\,ds \ge 3\Phi^2/L^3
  \)
  under the seam condition $k(L)=0$. Interestingly, it already appears close to the fundamental arc of the lemniscate in Figure \ref{fig:selected-first-pair}.}
  \label{fig:equality-arc}
\end{figure}

\begin{corollary}[Coercive lower bound for the arc energy]\label{cor:arc-lower-bound}
Let $\gamma:[0,L]\to\R^2$ be unit-speed with curvature $k\in H^1(0,L)$ satisfying $k(L)=0$ and turning $\Phi$.
Then
\begin{equation}\label{eq:J-lower-coercive}
J_\lambda[\gamma]
=\frac12\int_0^L k_s^2\,ds+\lambda L
\;\ge\;
\frac{3\Phi^2}{2L^3}+\lambda L.
\end{equation}
Moreover, if $\Phi\ne0$, then for fixed $(\lambda,\Phi)$ the right-hand side is minimised over $L>0$ at
\begin{equation}\label{eq:L-opt}
L_\ast=\Big(\frac{9\Phi^2}{2\lambda}\Big)^{1/4},
\end{equation}
and one obtains the explicit sharp bound
\begin{equation}\label{eq:arc-min-value}
J_\lambda[\gamma]
\;\ge\;
\min_{L>0}\Big(\frac{3\Phi^2}{2L^3}+\lambda L\Big)
=
\frac{4}{3}\Big(\frac{9}{2}\Big)^{\!1/4}\lambda^{3/4}|\Phi|^{1/2}.
\end{equation}
\end{corollary}

\begin{proof}
The inequality \eqref{eq:J-lower-coercive} is immediate from Lemma~\ref{lem:turning-controls-ks}.
Assume $\Phi\ne0$.  To minimise $f(L):=\frac{3\Phi^2}{2L^3}+\lambda L$ over $L>0$, compute
\[
f'(L)= -\frac{9\Phi^2}{2L^4}+\lambda,\qquad
f''(L)=\frac{18\Phi^2}{L^5}>0.
\]
Hence $f$ is strictly convex and its unique minimiser solves $f'(L)=0$, i.e.
\(
L^4=\frac{9\Phi^2}{2\lambda}
\), 
which is \eqref{eq:L-opt}.  Substituting $L=L_\ast$ gives
\[
\frac{3\Phi^2}{2L_\ast^3}
=\frac{3\Phi^2}{2}\Big(\frac{9\Phi^2}{2\lambda}\Big)^{-3/4}
=\lambda^{3/4}|\Phi|^{1/2}\cdot \frac{3}{2}\Big(\frac{9}{2}\Big)^{-3/4}
=\lambda^{3/4}|\Phi|^{1/2}\cdot\frac{1}{3}\Big(\frac{9}{2}\Big)^{1/4},
\]
and
\[
\lambda L_\ast=\lambda\Big(\frac{9\Phi^2}{2\lambda}\Big)^{1/4}
=\lambda^{3/4}|\Phi|^{1/2}\cdot \Big(\frac{9}{2}\Big)^{1/4}.
\]
Therefore
\[
f(L_\ast)=\lambda^{3/4}|\Phi|^{1/2}\Big(\frac{1}{3}+\;1\Big)\Big(\frac{9}{2}\Big)^{1/4}
=\frac{4}{3}\Big(\frac{9}{2}\Big)^{1/4}\lambda^{3/4}|\Phi|^{1/2},
\]
which is \eqref{eq:arc-min-value}.
\end{proof}

\subsection{Rescaling to a fixed domain}\label{subsec:aux-rescaling}

When working with tangent angles, it is convenient to pass from $[0,L]$ to $[0,1]$.
This isolates the coercive $L^{-3}$ scaling of the curvature-derivative term.
The proof is the same scaling argument used previously in Lemma \ref{lem:scaling}, so we omit it.

\begin{lemma}[Angle rescaling formulas]\label{lem:angle-rescaling}
Let $L>0$ and let $\theta\in H^2(0,L)$.  Define $\phi\in H^2(0,1)$ by
\[
\phi(t):=\theta(Lt),\qquad t\in[0,1].
\]
Then, with $s=Lt$,
\begin{equation}\label{eq:rescaling-derivatives}
\theta_s(Lt)=\frac1L\,\phi'(t),\qquad
\theta_{ss}(Lt)=\frac1{L^2}\,\phi''(t),\qquad ds=L\,dt,
\end{equation}
and consequently
\begin{equation}\label{eq:rescaled-energy}
\int_0^L \theta_{ss}^2\,ds
=\frac{1}{L^3}\int_0^1 (\phi'')^2\,dt.
\end{equation}
Moreover, the endpoint height constraint
\[
y(L)=\int_0^L \sin\theta(s)\,ds
\]
rescales to
\begin{equation}\label{eq:rescaled-closure}
y(L)=L\int_0^1 \sin\phi(t)\,dt.
\end{equation}
\end{lemma}

\subsection{First variation with a single geometric closure constraint}\label{subsec:aux-EL}

Later we will minimise $J_\lambda$ over an admissible class of arcs described by $(L,\theta)$ with fixed endpoint
angles and the scalar closure constraint $\int_0^L\sin\theta\,ds=0$.
In this subsection we compute the Euler-Lagrange equation in the $\theta$-variable and the natural boundary
condition at the free end $s=0$.

\begin{remark}[Angle variations are fixed-length variations]
For fixed \(L>0\), the tangent angle is used here as a coordinate on
unit-speed arcs of length \(L\).  Given \(\theta\in H^2(0,L)\), we reconstruct
\[
\gamma_\theta(s)=\gamma_\theta(0)+\int_0^s(\cos\theta(\sigma),\sin\theta(\sigma))\,d\sigma .
\]
Thus every nearby angle \(\theta+\varepsilon\eta\) reconstructs a unit-speed
arc on the same interval \([0,L]\).  In this angle-space problem the measure
\(ds\) is therefore fixed by construction.  The resulting Euler--Lagrange
equation is a fixed-length, inextensible first variation.  It is not, by
itself, the full geometric first variation of \(J_\lambda\); length-changing
variations are accounted for separately through the length variable \(\ell\)
and, a posteriori, through the Hamiltonian identity below.
\end{remark}

\begin{lemma}[Euler-Lagrange equation and natural boundary condition]\label{lem:EL-and-natural}
Fix $L>0$ and $\theta_0,\theta_1\in\R$.
Consider the  functional
\[
\mathcal{I}[\theta]
:=\frac12\int_0^L \theta_{ss}^2\,ds
\quad\text{subject to}\quad
\int_0^L \sin\theta\,ds=0,
\ \theta(0)=\theta_0,
\ \theta(L)=\theta_1, 
\ \theta_s(L)=0,   
\quad \theta\in H^2(0,L).
\]
If $\theta$ is a minimiser, then there exists a constant $a\in\R$ such that:
\begin{enumerate}[label=\textup{(\roman*)},leftmargin=2.0em]
\item $\theta$ satisfies the weak Euler-Lagrange equation
\begin{equation}\label{eq:EL-aux}
\theta_{ssss}+a\cos\theta=0\qquad\text{on }(0,L);
\end{equation}
\item the natural boundary condition
\begin{equation}\label{eq:natural-aux}
\theta_{ss}(0)=0
\end{equation}
holds.
\end{enumerate}
If, in addition, $\cos(\theta_0)=0$ (e.g.\ $\theta_0=\pi/2$), then evaluating \eqref{eq:EL-aux} at $s=0$ yields
\begin{equation}\label{eq:third-seam-aux}
\theta_{ssss}(0)=0.
\end{equation}
\end{lemma}

\begin{proof}
Let $\eta\in H^2(0,L)$ be an admissible variation preserving the endpoint/seam constraints to first order, i.e.
\[
\eta(0)=\eta(L)=0,\qquad \eta_s(L)=0.
\]
For $\varepsilon$ small set $\theta_\varepsilon:=\theta+\varepsilon\eta$.
By the Lagrange multiplier principle, there exists $a\in\R$ such that
\begin{equation}
\label{eqn:lamult}
\frac{d}{d\varepsilon}\Big|_{\varepsilon=0}
\left[
\frac12\int_0^L (\theta_\varepsilon)_{ss}^2\,ds
+a\int_0^L \sin(\theta_\varepsilon)\,ds
\right]=0
\quad\text{for all such }\eta.
\end{equation}
To justify the multiplier, define the constraint map $C:H^2(0,L)\to\R$ by
\[
C(\theta):=\int_0^{L}\sin\theta\,ds.
\]
Then $C$ is $C^1$ and
\[
DC(\theta)[\eta]=\int_0^{L}\cos(\theta)\,\eta\,ds.
\]
Since $C(\theta)=0$, we cannot have $\cos(\theta)\equiv 0$ on $[0,L]$ (otherwise $\theta\equiv \pi/2$ or $3\pi/2$ by
continuity, which would force $C(\theta)=\pm L\neq 0$).
Choose $\chi(s):=s^2(L-s)^2$ and set $\eta_0(s):=\chi(s)\cos(\theta(s))$. Then $\eta_0$ is admissible
($\eta_0(0)=\eta_0(L)=0$ and $\eta_{0,s}(L)=0$), and
\[
DC(\theta)[\eta_0]=\int_0^{L}\chi\,\cos^2(\theta)\,ds>0.
\]
Thus the constraint qualification holds at $\theta$, and the Lagrange multiplier principle gives $a\in\R$ such that
\eqref{eqn:lamult} holds for all admissible $\eta$.

Compute the first variation:
\begin{equation}\label{eq:first-var-aux}
\int_0^L \theta_{ss}\,\eta_{ss}\,ds + a\int_0^L \cos(\theta)\,\eta\,ds = 0.
\end{equation}
Integrate the first term by parts twice:
\[
\int_0^L \theta_{ss}\eta_{ss}\,ds
=\Big[\theta_{ss}\eta_s-\theta_{sss}\eta\Big]_0^L+\int_0^L \theta_{ssss}\eta\,ds.
\]
At $s=L$ we have $\eta(L)=0$ and $\eta_s(L)=0$, hence the boundary contribution at $L$ vanishes.
At $s=0$ we have $\eta(0)=0$ but $\eta_s(0)$ is unconstrained, so the boundary term reduces to
\[
\Big[\theta_{ss}\eta_s-\theta_{sss}\eta\Big]_0^L=-\theta_{ss}(0)\,\eta_s(0).
\]
Substituting into \eqref{eq:first-var-aux} yields
\[
-\theta_{ss}(0)\,\eta_s(0)+\int_0^L\big(\theta_{ssss}+a\cos\theta\big)\eta\,ds=0.
\]
Since $\eta_s(0)$ can be chosen arbitrarily, we must have $\theta_{ss}(0)=0$, proving \eqref{eq:natural-aux}.
With \eqref{eq:natural-aux} imposed, the remaining condition becomes
\[
\int_0^L\big(\theta_{ssss}+a\cos\theta\big)\eta\,ds=0
\quad\text{for all admissible }\eta,
\]
which is the weak form of \eqref{eq:EL-aux}.
Since $\theta\in H^2(0,L)\hookrightarrow C^1([0,L])$, the right-hand side
$-a\cos\theta$ is continuous.  The weak equation therefore gives $\theta\in C^4([0,L])$, and standard
one-dimensional bootstrapping gives $\theta\in C^\infty([0,L])$.  Thus, if $\cos(\theta_0)=0$, evaluating
\eqref{eq:EL-aux} at $s=0$ gives
$\theta_{ssss}(0)=-a\cos(\theta_0)=0$, proving \eqref{eq:third-seam-aux}.
\end{proof}

\subsection{Symmetries}

The second symmetry used to close the doubled arc is reflection across the $x$-axis, which is smooth at a seam
provided the tangent is vertical and the curvature jet has the correct parity.

\begin{lemma}[One-sided reflection principle at a vertical tangent]\label{lem:reflection-principle}
Let $h>0$, let $\theta_0\in\frac\pi2+\pi\Z$, and let $\theta\in C^4([0,h))$ solve
\[
\theta_{ssss}=c\cos\theta\qquad\text{on }(0,h)
\]
for some $c\in\R$.  Assume $\theta(0)=\theta_0$ and $\theta_{ss}(0)=0$.  Define
\[
\widehat\theta(s):=
\begin{cases}
\theta(s),&0\le s<h,\\
2\theta_0-\theta(-s),&-h<s\le0.
\end{cases}
\]
Then $\widehat\theta\in C^\infty((-h,h))$ and satisfies
$\widehat\theta_{ssss}=c\cos\widehat\theta$ throughout $(-h,h)$.  The analogous statement holds when the given
solution lies on the left of the seam.
\end{lemma}

\begin{proof}
The reflected traces of derivatives of orders one and three agree automatically, while the second-order traces agree
because $\theta_{ss}(0)=0$.  The ODE and $\cos\theta_0=0$ give $\theta_{ssss}(0)=0$, so the fourth-order traces agree as
well.  Hence $\widehat\theta\in C^4((-h,h))$.  Since $2\theta_0$ is an odd multiple of $\pi$,
\(
\cos(2\theta_0-\alpha)=-\cos\alpha
\),
and four differentiations of $2\theta_0-\theta(-s)$ give $-\theta_{ssss}(-s)$.  Thus the reflected function satisfies
the same ODE on the other side of the seam.  Bootstrapping the smooth ODE gives
$\widehat\theta\in C^\infty((-h,h))$.  Reflection of the independent variable gives the left-sided version.
\end{proof}

In the variational construction used later, point-reflection symmetry is imposed at the level of the tangent angle
on the {doubled} interval.  

Recall that a function $f:[0,2\ell]\to\R$ is \emph{even about $\ell$} if
\(
f(\ell+s)=f(\ell-s)
\) for a.e. $s\in(0,\ell)$.

\begin{lemma}\label{lem:midpoint-seam-from-evenness}
Let $\ell>0$ and $\theta\in H^2(0,2\ell)$ be even about $\ell$.
Then $\theta\in C^1([0,2\ell])$, $\theta_s\in H^1(0,2\ell)$, and
\(\theta_s(\ell)=0\). 
\end{lemma}

\begin{proof}
Since $\theta\in H^2(0,2\ell)$ in one dimension, the Sobolev embedding gives $\theta\in C^1([0,2\ell])$.
The evenness identity therefore holds pointwise.  Differentiating it for a.e.\ $s\in(0,\ell)$ gives
\[
\theta_s(\ell+s)=-\theta_s(\ell-s).
\]
Since $\theta_s\in H^1(0,2\ell)\hookrightarrow C^0([0,2\ell])$, letting $s\downarrow0$ yields
$\theta_s(\ell)=-\theta_s(\ell)$, hence $\theta_s(\ell)=0$.
\end{proof}

The above oddness and evenness properties imply appropriate 
higher seam jet conditions needed for smooth point-reflection doubling.

\begin{lemma}[Even tangent angle implies point symmetry of the curve]\label{lem:even-theta-point-symmetry}
Let $\ell>0$ and $\theta\in H^2(0,2\ell)$ be even about $\ell$.
Define the unit-speed curve $\gamma:[0,2\ell]\to\R^2$ by 
\eqref{eq:defgamma}.
Then $\gamma$ is point-symmetric about $B:=\gamma(\ell)$:
\begin{equation}\label{eq:point-symmetry-gamma}
\gamma(\ell+s)=2B-\gamma(\ell-s)\qquad\text{for all }s\in[0,\ell].
\end{equation}
In particular, if $y(\ell)=0$ then also $y(2\ell)=0$.
\end{lemma}

\begin{proof}
Let $T=(\cos\theta,\sin\theta)=\gamma_s$.
Evenness of $\theta$ about $\ell$ implies evenness of $T$ about $\ell$:
\[
T(\ell+s)=T(\ell-s)\quad\text{for a.e. }s\in(0,\ell).
\]
For $s\in[0,\ell]$ we compute
\[
\gamma(\ell+s)-\gamma(\ell)=\int_\ell^{\ell+s}T(\sigma)\,d\sigma
=\int_0^{s}T(\ell+u)\,du
=\int_0^s T(\ell-u)\,du
= \int_{\ell-s}^{\ell}T(\sigma)\,d\sigma
=-(\gamma(\ell-s)-\gamma(\ell)),
\]
which is \eqref{eq:point-symmetry-gamma}.  The final statement follows by taking $y$-components.
\end{proof}

\subsection{Energy additivity under symmetry actions}\label{subsec:aux-additivity}

We first note the following elementary fact.

\begin{lemma}[Signed curvature under arclength reversal]\label{lem:curvature-reversal}
Let $\gamma:[0,L]\to\R^2$ be a $C^2$ unit-speed immersion. 
Define the arclength-reversed curve $\tilde\gamma:[0,L]\to\R^2$ by
\(
\tilde\gamma(s):=\gamma(L-s)
\).
Then $\tilde\gamma$ is unit-speed and $C^2$, and its signed curvature $\tilde k$
satisfies
\(
\tilde k(s)=-k(L-s)
\)
and (wherever $k_s$ exists)
$\tilde k_s(s)=k_s(L-s)$ 
for all $s\in[0,L]$.
\end{lemma}

We will build closed curves from one fundamental arc by applying isometries and parameter reversals.
The following bookkeeping lemma reduces the energy of the closed curve to the energy of one arc.

\begin{lemma}[Additivity of $J_\lambda$ under gluing congruent pieces]\label{lem:additivity}
Let $\gamma:[0,L]\to\R^2$ be unit-speed with $k\in H^1$.
Suppose $\Gamma$ is obtained by concatenating $m$ copies of $\gamma$, each copy transformed by a rigid motion of $\R^2$
and possibly a parameter reversal, and glued so that $\Gamma$ is a unit-speed immersion with curvature in $H^1$.
Then
\begin{equation}\label{eq:additivity}
J_\lambda[\Gamma]=m\,J_\lambda[\gamma].
\end{equation}
\end{lemma}

\begin{proof}
Every Euclidean isometry preserves arclength; an orientation-reversing isometry changes the sign of signed curvature
and of its arclength derivative, but still preserves $k_s^2\,ds$.  Parameter reversal
$s\mapsto L-s$ sends $k\mapsto -k$ and $k_s\mapsto k_s$ but leaves $k_s^2\,ds$ invariant (see Lemma \ref{lem:curvature-reversal}).
Therefore each congruent copy contributes exactly the same amount to $\int k_s^2\,ds$ and to the length.
Since $\Gamma$ consists of $m$ such copies \eqref{eq:additivity} follows.
\end{proof}

\subsection{Energetic separation of the closed family}\label{subsec:aux-distinctness}

Combining the coercive arc bound with the additivity lemma yields a quantitative energy estimate.  Since $J_\lambda$
is invariant under Euclidean congruences and reparametrisation, divergence of these energies will later allow us to
extract infinitely many pairwise non-congruent closed critical points without assuming uniqueness of a fundamental-arc
decomposition.

\begin{proposition}[Energetic lower bound by angle]\label{prop:closed-lb}
Let $\Gamma$ be a symmetric null-turning closed curve obtained from a unit-speed fundamental arc
$\gamma:[0,L]\to\R^2$ with curvature $k\in H^1(0,L)$.
Let $\Phi:=\int_\gamma k\,ds$ be the unwrapped turning of $\gamma$.
Then
\begin{equation}\label{eq:closed-lb}
J_\lambda[\Gamma]
\;\ge\;
\frac{16}{3}\Big(\frac{9}{2}\Big)^{\!1/4}\lambda^{3/4}|\Phi|^{1/2}.
\end{equation}
\end{proposition}

\begin{proof}
The construction uses $4$ congruent copies of the fundamental arc, hence by Lemma~\ref{lem:additivity},
\(
J_\lambda[\Gamma]=4J_\lambda[\gamma]
\).
Applying the arc lower bound \eqref{eq:arc-min-value} from Corollary~\ref{cor:arc-lower-bound} gives \eqref{eq:closed-lb}.
\end{proof}

\section{The direct method on doubled intervals}\label{sec:doubled-direct-method}

In this section we build point-reflection symmetry {into} the variational problem by working on the doubled
interval $[0,2\ell]$ and imposing evenness of the tangent angle about the midpoint $\ell$.
For a single prescribed lifted turning $\Phi$, the midpoint angle is fixed; the direct method therefore gives a
piecewise-smooth minimiser that may carry a midpoint balance defect.  The missing condition is
$(\theta)_{sss}(\ell^-)=0$.  The turning-window argument in Section~\ref{sec:infinite-family} selects minimisers for
which this condition holds, and only then does the doubled arc become smooth through the midpoint.

\subsection{Admissible class and the doubled functional}\label{subsec:doubled-admissible}

Fix $\lambda>0$ and a prescribed unwrapped turning $\Phi\in\R$.
We consider tangent angles $\theta$ on the doubled interval $[0,2\ell]$ satisfying:
\begin{itemize}[leftmargin=2.2em]
\item \textbf{endpoint and midpoint angles}
\[
\theta(0)=\frac{\pi}{2},\qquad \theta(\ell)=\frac{\pi}{2}+\Phi,\qquad \theta(2\ell)=\frac{\pi}{2};
\]
\item \textbf{point-reflection symmetry} (evenness about $\ell$)
\[
\theta(\ell+s)=\theta(\ell-s)\quad\text{for a.e. }s\in(0,\ell);
\]
\item \textbf{midpoint-on-axis closure constraint}
\[
\int_0^\ell \sin\theta(s)\,ds=0.
\]
\end{itemize}
Given $(\ell,\theta)$ we reconstruct the unit-speed curve $\gamma:[0,2\ell]\to\R^2$ by \eqref{eq:defgamma}.
Then $\gamma(0)=(0,0)$ and the closure constraint makes the midpoint $\gamma(\ell)$ lie on the $x$-axis.
By Lemma~\ref{lem:even-theta-point-symmetry}, $\gamma$ is point-symmetric about $\gamma(\ell)$, hence also
$\gamma(2\ell)\in\{y=0\}$.

We work in the $(\ell,\theta)$ variables and define the doubled admissible set $\mathcal{A}^{(2)}(\Phi)$ as in \eqref{eq:Adouble}, the energy 
$J_\lambda^{(2)}(\ell,\theta)$ as in \eqref{eq:Jdouble},  
and the turning-window class 
$\mathcal{A}^{(2)}(I)$, where $I$ is a compact interval, as in \eqref{eq:Adouble-window}.
Note that $(\ell,\theta)\in\mathcal{A}^{(2)}(I)$ satisfies all conditions in \eqref{eq:Adouble} except that the midpoint turning is allowed to vary subject to
\[
\theta(\ell)-\frac\pi2\in I.
\]
This small change is essential: when the selected turning lies in $\operatorname{int}I$, first variations may change
$\theta(\ell)$ and therefore enforce the natural midpoint transversality condition.

\subsection{Non-emptiness}\label{subsec:doubled-nonempty}

\begin{lemma}[Non-emptiness of the doubled admissible class]\label{lem:Adouble-nonempty}
For every $\Phi\in\R$ there exists at least one pair $(\ell,\theta)\in\mathcal{A}^{(2)}(\Phi)$ with $\theta\in C^\infty([0,2\ell])$.
\end{lemma}

\begin{proof}
Choose smooth transition profiles $\tau_+,\tau_\Phi\in C^\infty([0,1])$, flat to all orders at both endpoints, such
that
\[
\tau_+(0)=\frac\pi2,\quad \tau_+(1)=\frac{3\pi}{2},\qquad
\tau_\Phi(0)=\frac{3\pi}{2},\quad \tau_\Phi(1)=\frac\pi2+\Phi.
\]
Put
\[
C_\Phi:=\int_0^1\sin\tau_+(s)\,ds+\int_0^1\sin\tau_\Phi(s)\,ds,\qquad
U:=1+\max\{0,-C_\Phi\},\qquad D:=1+\max\{0,C_\Phi\}.
\]
Then $U,D>0$ and $U-D+C_\Phi=0$.  Set $\ell:=U+D+2$ and concatenate, on $[0,\ell]$, a plateau of
angle $\pi/2$ and length $U$, the transition $\tau_+$, a plateau of angle $3\pi/2$ and length $D$, and the transition
$\tau_\Phi$.  Flatness at the transition endpoints makes the resulting half-angle, denoted $\theta_\ast$, smooth and
constant to all orders at $s=\ell$.  Moreover,
\[
\int_0^\ell\sin\theta_\ast(s)\,ds=U-D+C_\Phi=0.
\]
Now define $\theta$ on $[0,2\ell]$ by even reflection about $\ell$:
\[
\theta(s):=
\begin{cases}
\theta_\ast(s), & 0\le s\le \ell,\\
\theta_\ast(2\ell-s), & \ell\le s\le 2\ell.
\end{cases}
\]
Then $\theta\in C^\infty([0,2\ell])$, $\theta$ is even about $\ell$, and it satisfies the endpoint/midpoint angle conditions.
Moreover, the closure constraint holds by construction. Hence $(\ell,\theta)\in\mathcal{A}^{(2)}(\Phi)$.
\end{proof}

\subsection{Coercivity and length bounds}\label{subsec:doubled-coercive}

\begin{lemma}[Turning controls $\|k_s\|_{L^2}$ on the doubled class]\label{lem:doubled-coercive}
Let $(\ell,\theta)\in\mathcal{A}^{(2)}(\Phi)$.
Then $k=\theta_s\in H^1(0,2\ell)$ and $k(\ell)=0$. Moreover,
\begin{equation}\label{eq:doubled-turning-bound}
\int_0^{2\ell} k_s^2\,ds \;\ge\; \frac{6\Phi^2}{\ell^3},
\end{equation}
and consequently
\begin{equation}\label{eq:doubled-J-lower}
J_\lambda^{(2)}(\ell,\theta)\;\ge\;\frac{3\Phi^2}{\ell^3}+2\lambda\,\ell.
\end{equation}
In particular, any minimising sequence for $J_\lambda^{(2)}$ in $\mathcal{A}^{(2)}(\Phi)$ has $\ell$ bounded above and bounded away from $0$.
\end{lemma}

\begin{proof}
Since $\theta\in H^2(0,2\ell)$, we have $k=\theta_s\in H^1(0,2\ell)$.
By Lemma~\ref{lem:midpoint-seam-from-evenness}, $\theta_s(\ell)=0$, hence $k(\ell)=0$.

Restrict to the first half $[0,\ell]$.
On $[0,\ell]$, $k\in H^1(0,\ell)$ with $k(\ell)=0$, and the turning on the half-interval is
\[
\int_0^\ell k\,ds=\theta(\ell)-\theta(0)=\Phi.
\]
Applying Lemma~\ref{lem:turning-controls-ks} on $[0,\ell]$ gives
\[
\int_0^\ell k_s^2\,ds \ge \frac{3\Phi^2}{\ell^3}.
\]
Since $\theta$ is even about $\ell$, $k_s=\theta_{ss}$ is also even about $\ell$, hence $k_s^2$ is even and
\[
\int_0^{2\ell}k_s^2\,ds = 2\int_0^\ell k_s^2\,ds \ge 2\cdot\frac{3\Phi^2}{\ell^3}=\frac{6\Phi^2}{\ell^3},
\]
which is \eqref{eq:doubled-turning-bound}.  Multiplying by $\frac12$ and adding $2\lambda\ell$ yields \eqref{eq:doubled-J-lower}.

If $\Phi\ne0$, the right-hand side $\frac{3\Phi^2}{\ell^3}+2\lambda\ell$ blows up as $\ell\downarrow0$.  For
$\Phi=0$, rescale the first half by $\phi(t)=\theta(\ell t)$ and put $u=\phi-\pi/2$.  Then
$u(0)=0$, $u'(1)=0$, and the closure constraint is $\int_0^1\cos u\,dt=0$.  Hence
$\|u\|_{L^\infty}\ge\pi/2$; otherwise $\cos u$ would be strictly positive.  On the other hand,
\[
u(t)=-\int_0^1\min\{t,r\}\,u''(r)\,dr,
\qquad
\|u\|_{L^\infty}\le\frac1{\sqrt3}\|u''\|_{L^2(0,1)}.
\]
Thus $\int_0^1(\phi'')^2\,dt\ge3\pi^2/4$, and evenness gives
$\int_0^{2\ell}k_s^2\,ds\ge(3\pi^2/2)\ell^{-3}$.  This again rules out $\ell\downarrow0$ along a minimising
sequence.  Finally, $2\lambda\ell\le J_\lambda^{(2)}$ controls $\ell$ from above, so in every case
$0<\ell_-\le\ell\le\ell_+$.
\end{proof}

\subsection{Compactness on a fixed domain}\label{subsec:doubled-compactness}

Introduce the rescaling to $t\in[0,2]$ by $s=\ell t$ and set
\(
\phi(t):=\theta(\ell t)\),
\( t\in[0,2]\).
Then (exactly as in Lemma~\ref{lem:angle-rescaling}) one has
\begin{equation}\label{eq:doubled-rescale-energy}
\int_0^{2\ell}\theta_{ss}^2\,ds
=\frac{1}{\ell^3}\int_0^2 (\phi'')^2\,dt.
\end{equation}
Evenness about $\ell$ becomes the fixed-domain symmetry
\begin{equation}\label{eq:phi-even-about-1}
\phi(1+t)=\phi(1-t)\quad\text{for a.e. }t\in(0,1),
\end{equation}
and the closure constraint becomes
\begin{equation}\label{eq:doubled-closure-fixed}
\int_0^1 \sin\phi(t)\,dt=0.
\end{equation}

\begin{lemma}[Compactness and constraint stability]\label{lem:doubled-compact}
Let $\{(\ell_j,\theta_j)\}\subset\mathcal{A}^{(2)}(\Phi)$ be a minimising sequence for $J_\lambda^{(2)}$.
Then after passing to a subsequence,
\[
\ell_j\to \ell\in[\ell_-,\ell_+],\qquad
\phi_j\rightharpoonup \phi\ \text{weakly in }H^2(0,2),\qquad
\phi_j\to \phi\ \text{strongly in }C^1([0,2]),
\]
where $\phi_j(t):=\theta_j(\ell_j t)$.
Moreover, the limit pair $(\ell,\theta)$ with $\theta(s):=\phi(s/\ell)$ lies in $\mathcal{A}^{(2)}(\Phi)$.
\end{lemma}

\begin{proof}
By Lemma~\ref{lem:doubled-coercive}, $\ell_j\in[\ell_-,\ell_+]$.
Boundedness of $J_\lambda^{(2)}(\ell_j,\theta_j)$ and \eqref{eq:doubled-rescale-energy} yield a uniform bound on
$\{\phi_j\}$ in $H^2(0,2)$.
Hence (by reflexivity) $\phi_j\rightharpoonup\phi$ weakly in $H^2(0,2)$ along a subsequence.
By the compact embedding $H^2(0,2)\hookrightarrow C^1([0,2])$, we also have $\phi_j\to\phi$ in $C^1([0,2])$ along a further subsequence.
We also assume $\ell_j\to\ell\in[\ell_-,\ell_+]$.

The boundary conditions pass to the limit by $C^0$ convergence:
\[
\phi(0)=\frac{\pi}{2},\qquad \phi(1)=\frac{\pi}{2}+\Phi,\qquad \phi(2)=\frac{\pi}{2}.
\]
The symmetry \eqref{eq:phi-even-about-1} is a closed linear condition in $H^2(0,2)$ (it is invariance under the bounded linear reflection
operator $t\mapsto 2-t$ restricted to $[0,2]$ around $t=1$), hence it passes to the weak limit: $\phi$ is even about $1$.

Finally, \eqref{eq:doubled-closure-fixed} passes to the limit because $\phi_j\to\phi$ uniformly:
\[
0=\int_0^1\sin\phi_j\,dt\to \int_0^1\sin\phi\,dt.
\]
Defining $\theta(s):=\phi(s/\ell)$ on $[0,2\ell]$ gives $(\ell,\theta)\in\mathcal{A}^{(2)}(\Phi)$.
\end{proof}

\subsection{Existence, Euler-Lagrange equation, and the midpoint balance defect}\label{subsec:doubled-existence-smoothness}

\begin{theorem}[Existence of a minimiser on the doubled interval]\label{thm:doubled-existence}
The minimisation problem
\[
\inf_{(\ell,\theta)\in\mathcal{A}^{(2)}(\Phi)} J_\lambda^{(2)}(\ell,\theta)
\]
admits a minimiser $(\ell_\ast,\theta_\ast)\in\mathcal{A}^{(2)}(\Phi)$.
\end{theorem}

\begin{proof}
Let $(\ell_j,\theta_j)$ be a minimising sequence.
By Lemma~\ref{lem:doubled-compact}, after extracting a subsequence we obtain a limit
$(\ell,\theta)\in\mathcal{A}^{(2)}(\Phi)$.

Lower semicontinuity follows as in the standard direct method:
using \eqref{eq:doubled-rescale-energy} and weak lower semicontinuity of the $L^2$ norm,
\[
\liminf_{j\to\infty}\frac12\int_0^{2\ell_j}\theta_{j,ss}^2\,ds
=
\liminf_{j\to\infty}\frac12\cdot\frac1{\ell_j^3}\int_0^2(\phi_j'')^2\,dt
\ \ge\
\frac12\cdot\frac1{\ell^3}\int_0^2(\phi'')^2\,dt
=
\frac12\int_0^{2\ell}\theta_{ss}^2\,ds.
\]
Also $2\lambda\ell_j\to 2\lambda\ell$.
Thus $J_\lambda^{(2)}(\ell,\theta)\le \liminf_j J_\lambda^{(2)}(\ell_j,\theta_j)$, so $(\ell,\theta)$ attains the infimum.
Rename $(\ell_\ast,\theta_\ast):=(\ell,\theta)$.
\end{proof}

\begin{theorem}[Euler-Lagrange equation and the midpoint balance defect]\label{thm:doubled-EL}
Let $(\ell_\ast,\theta_\ast)$ be a minimiser from Theorem~\ref{thm:doubled-existence} for a fixed prescribed turning $\Phi$.
Then there exists $a\in\R$ such that $\theta_\ast$ satisfies
\begin{equation}\label{eq:doubled-EL-eq}
(\theta_\ast)_{ssss}+a\cos(\theta_\ast)=0
\end{equation}
on each open half-interval $(0,\ell_\ast)$ and $(\ell_\ast,2\ell_\ast)$.  Moreover,
\begin{equation}\label{eq:doubled-natural-ends}
(\theta_\ast)_{ss}(0)=0,
\qquad
(\theta_\ast)_{ss}(2\ell_\ast)=0.
\end{equation}
The minimiser is $C^\infty$ on each closed half-interval up to the midpoint from the appropriate side.  Evenness gives
\[
(\theta_\ast)_s(\ell_\ast)=0,
\qquad
(\theta_\ast)_{sss}(\ell_\ast^+)=- (\theta_\ast)_{sss}(\ell_\ast^-),
\]
whenever the one-sided traces are taken.  The number
\begin{equation}\label{eq:midpoint-Q-fixed}
Q_\Phi:=(\theta_\ast)_{sss}(\ell_\ast^-)
\end{equation}
is the midpoint balance defect.  Distributionally on $(0,2\ell_\ast)$,
\[
(\theta_\ast)_{ssss}+a\cos\theta_\ast=-2Q_\Phi\,\delta_{\ell_\ast},
\]
so $-2Q_\Phi$ is the actual Dirac-mass coefficient.  If $Q_\Phi=0$, then the two one-sided solutions glue smoothly
across $s=\ell_\ast$.
\end{theorem}

\begin{proof}
The first variation for a fixed value of $\Phi$ is tested against even functions
\[
\mathcal V_{\ell_\ast}^{\mathrm{ev}}:=\Big\{\eta\in H^2(0,2\ell_\ast):
\eta(0)=\eta(\ell_\ast)=\eta(2\ell_\ast)=0,\ \eta\text{ even about }\ell_\ast\Big\}.
\]
The condition $\eta(\ell_\ast)=0$ is exactly the fixed-midpoint-angle constraint.  As in Lemma~\ref{lem:EL-and-natural}, the closure constraint is regular because $\cos\theta_\ast$ cannot vanish identically on $[0,\ell_\ast]$.  Hence there is $a\in\R$ such that
\begin{equation}\label{eq:doubled-firstvar}
\int_0^{2\ell_\ast}(\theta_\ast)_{ss}\eta_{ss}\,ds
+
a\int_0^{\ell_\ast}\cos(\theta_\ast)\eta\,ds=0
\qquad\text{for all }\eta\in\mathcal V_{\ell_\ast}^{\mathrm{ev}}.
\end{equation}
Since the integrand is even, its integral over $[0,2\ell_\ast]$ is twice its integral over
$[0,\ell_\ast]$.  After replacing $a/2$ by $a$, we use the same letter for the coefficient in the corresponding
full-interval equation.

Using test functions supported in $(0,\ell_\ast)$ and then reflecting them evenly gives the weak equation on the first
half; symmetry gives the same equation on the second half.  Standard one-dimensional bootstrapping then gives
smoothness on each open half-interval and up to the midpoint from either side.

Integrating by parts on the two half-intervals, the fixed midpoint value $\eta(\ell_\ast)=0$ removes the possible
$\theta_{sss}$ boundary term at the midpoint.  At the outer endpoints the boundary contribution is
\[
(\theta_\ast)_{ss}(2\ell_\ast)\eta_s(2\ell_\ast)-(\theta_\ast)_{ss}(0)\eta_s(0).
\]
Because $\eta$ is even about $\ell_\ast$, $\eta_s(2\ell_\ast)=-\eta_s(0)$; because $\theta_\ast$ is even, the one-sided endpoint traces of $\theta_{ss}$ agree.  Since $\eta_s(0)$ is arbitrary, \eqref{eq:doubled-natural-ends} follows.

Finally, evenness gives $(\theta_\ast)_s(\ell_\ast)=0$ and the one-sided parity
$(\theta_\ast)_{sss}(\ell_\ast^+)=-(\theta_\ast)_{sss}(\ell_\ast^-)$.  Thus the jump in the third derivative is
$-2Q_\Phi$, which proves the displayed distributional identity.  If $Q_\Phi=0$, then the one-sided Cauchy data through
order three match at $\ell_\ast$; uniqueness and bootstrapping for the analytic ODE give smooth gluing across the midpoint.
\end{proof}

\subsection{Closing by $x$-axis reflection and energy bookkeeping}\label{subsec:doubled-close}

Let $(\ell_\ast,\theta_\ast)$ be as above and let $\gamma_\ast$ be the reconstructed curve \eqref{eq:defgamma}.
By Lemma~\ref{lem:even-theta-point-symmetry}, $\gamma_\ast$ is point-symmetric about $B_\ast:=\gamma_\ast(\ell_\ast)$.
Since $\int_0^{\ell_\ast}\sin\theta_\ast\,ds=0$, we have $B_\ast\in\{y=0\}$ and hence also $\gamma_\ast(2\ell_\ast)\in\{y=0\}$.

Define the closed curve $\Gamma_\ast:[0,4\ell_\ast]\to\R^2$ by reflecting the doubled arc across the $x$-axis:
\begin{equation}\label{eq:doubled-xaxis-close}
\Gamma_\ast(s):=
\begin{cases}
\gamma_\ast(s), & 0\le s\le 2\ell_\ast,\\[2pt]
(x_\ast(4\ell_\ast-s),-y_\ast(4\ell_\ast-s)), & 2\ell_\ast\le s\le 4\ell_\ast,
\end{cases}
\end{equation}
where $\gamma_\ast=(x_\ast,y_\ast)$.

\begin{proposition}[Smooth $x$-axis closure for balanced doubled arcs]\label{prop:doubled-xaxis-smooth}
Assume that the doubled angle $\theta_\ast$ is balanced at the midpoint, i.e.
\[
(\theta_\ast)_{sss}(\ell_\ast^-)=0.
\]
Then the curve $\Gamma_\ast$ defined by \eqref{eq:doubled-xaxis-close} is a $C^\infty$ closed unit-speed immersion.
\end{proposition}

\begin{proof}
By Theorem~\ref{thm:doubled-EL}, the balance condition removes the midpoint Dirac mass, so the reconstructed doubled arc is smooth through $s=\ell_\ast$.  Continuity and closure follow from $\gamma_\ast(0)=(0,0)$ and $\gamma_\ast(2\ell_\ast)\in\{y=0\}$.  Unit speed is preserved because reflection and parameter reversal preserve $|d\gamma/ds|$.

At the $x$-axis reflection seams $s=2\ell_\ast$ and $s=0\sim4\ell_\ast$, $C^1$ matching follows from the endpoint angle conditions $\theta_\ast(0)=\theta_\ast(2\ell_\ast)=\pi/2$.  For higher smoothness, Theorem~\ref{thm:doubled-EL} gives the endpoint natural conditions $(\theta_\ast)_{ss}(0)=(\theta_\ast)_{ss}(2\ell_\ast)=0$ and the angle ODE near each endpoint.  Lemma~\ref{lem:reflection-principle} therefore applies at both vertical seams and gives precisely the jet compatibility needed for the reflection across the $x$-axis.  Hence \eqref{eq:doubled-xaxis-close} is $C^\infty$.
\end{proof}

\begin{lemma}[Energy additivity]\label{lem:doubled-energy-additivity}
Let $\gamma_\ast$ and $\Gamma_\ast$ be as above.
Then
\begin{equation}\label{eq:doubled-energy-identity}
J_\lambda[\Gamma_\ast]=2\,J_\lambda^{(2)}(\ell_\ast,\theta_\ast)=4\,J_\lambda[\gamma_\ast|_{[0,\ell_\ast]}].
\end{equation}
\end{lemma}

\begin{proof}
The closed curve $\Gamma_\ast$ consists of two congruent copies of $\gamma_\ast$ (the second obtained by reflection and parameter reversal),
so Lemma~\ref{lem:additivity} with $m=2$ gives $J_\lambda[\Gamma_\ast]=2J_\lambda[\gamma_\ast]=2J_\lambda^{(2)}(\ell_\ast,\theta_\ast)$.
Moreover, $\gamma_\ast$ itself consists of two congruent copies of its first half by point symmetry, so again by Lemma~\ref{lem:additivity},
$J_\lambda[\gamma_\ast]=2J_\lambda[\gamma_\ast|_{[0,\ell_\ast]}]$, yielding \eqref{eq:doubled-energy-identity}.
\end{proof}

\section{Turning-window selection and proof of Theorem \ref{thm:main-infinite}}\label{sec:infinite-family}

The fixed-turning direct method alone does not make the midpoint angle free.  This section repairs the construction by minimising over compact lifted-turning windows.  The proof has four parts: a fixed-domain value function, a boundary-layer estimate that creates infinitely many good windows, the midpoint transversality condition for interior window minimisers, and the geometric upgrade from the reduced angle equation to the full Euler-Lagrange equation for $J_\lambda$.

\subsection{The fixed-domain shape value and the length variable}\label{subsec:tw-shape-value}

For a lifted midpoint turning $\Phi$, define
\begin{equation}\label{eq:tw-BPhi}
\mathcal{B}(\Phi):=
\left\{
\phi\in H^2(0,1):
\phi(0)=\frac\pi2,
\phi(1)=\frac\pi2+\Phi,
\phi'(1)=0,
\int_0^1\sin\phi(t)\,dt=0
\right\}.
\end{equation}
Define the fixed-domain shape value
\begin{equation}\label{eq:tw-APhi}
A(\Phi):=\inf_{\phi\in\mathcal{B}(\Phi)}\int_0^1(\phi''(t))^2\,dt.
\end{equation}

\begin{lemma}[Existence and continuity of the shape value]\label{lem:tw-A-continuity}
For every $\Phi\in\R$, the class $\mathcal{B}(\Phi)$ is nonempty and the infimum defining $A(\Phi)$ is attained.  Moreover $A:\R\to\R$ is continuous.
\end{lemma}

\begin{proof}
Nonemptiness follows by taking the first half of the explicit plateau construction in
Lemma~\ref{lem:Adouble-nonempty} and rescaling it to $[0,1]$.

For existence, let $\phi_m\in\mathcal{B}(\Phi)$ be a minimising sequence.  Since $\phi_m'(1)=0$ and $\phi_m(0)$ is fixed, a one-dimensional Poincare inequality gives a uniform $H^2(0,1)$ bound from the bound on $\|\phi_m''\|_{L^2}$.  After passing to a subsequence, $\phi_m\rightharpoonup\phi$ in $H^2$ and $\phi_m\to\phi$ in $C^1([0,1])$.  The endpoint conditions and the nonlinear closure constraint pass to the limit, and weak lower semicontinuity gives attainment of $A(\Phi)$.

We prove continuity.  Fix $\Phi_0$ and let $\Phi_m\to\Phi_0$.  Let $\phi_0$ be a minimiser for $A(\Phi_0)$ and set
\[
H(t):=\frac32t-\frac12t^3.
\]
Then $H(0)=0$, $H(1)=1$, and $H'(1)=0$.  The angles
\[
\phi_m^0(t):=\phi_0(t)+(\Phi_m-\Phi_0)H(t)
\]
have the correct endpoint data for $\Phi_m$, but may fail the closure constraint by $O(|\Phi_m-\Phi_0|)$.  Since $\cos\phi_0$ is not identically zero, choose $\psi\in C_c^\infty(0,1)$ with
\[
\int_0^1\cos\phi_0\,\psi\,dt\ne0.
\]
Define
\[
G(\delta,a):=\int_0^1\sin\bigl(\phi_0+\delta H+a\psi\bigr)\,dt.
\]
Then $G(0,0)=0$ and
$\partial_aG(0,0)=\int_0^1\cos\phi_0\,\psi\,dt\ne0$.  The implicit function theorem gives $a_m\to0$ such that
$G(\Phi_m-\Phi_0,a_m)=0$, equivalently $\phi_m^0+a_m\psi\in\mathcal{B}(\Phi_m)$.  Hence
\[
\limsup_{m\to\infty}A(\Phi_m)\le A(\Phi_0).
\]
Conversely, take minimisers $\phi_m\in\mathcal{B}(\Phi_m)$.  The upper semicontinuity just proved gives a uniform
bound on $\|\phi_m''\|_{L^2}$, while $\phi_m'(1)=0$ and the fixed value $\phi_m(0)=\pi/2$ give a uniform
$H^2(0,1)$ bound.  After passing to a subsequence,
\[
\phi_m\rightharpoonup\phi\quad\text{in }H^2(0,1),
\qquad
\phi_m\to\phi\quad\text{in }C^1([0,1]).
\]
Since $\Phi_m\to\Phi_0$, the endpoint conditions pass to the limit, and uniform convergence gives
\[
\int_0^1\sin\phi\,dt=\lim_{m\to\infty}\int_0^1\sin\phi_m\,dt=0.
\]
Thus $\phi\in\mathcal{B}(\Phi_0)$, and weak lower semicontinuity yields
\[
A(\Phi_0)\le\liminf_{m\to\infty}A(\Phi_m).
\]
The two inequalities prove continuity.
\end{proof}

For fixed $\phi\in\mathcal{B}(\Phi)$ and $\ell>0$, the doubled length-penalised energy is
\[
J_\lambda^{(2)}(\ell,\phi)=\ell^{-3}\int_0^1(\phi'')^2\,dt+2\lambda\ell.
\]
After minimising over shapes and then over $\ell$, define the fixed-turning value function
\begin{equation}\label{eq:tw-Mlambda}
M_\lambda(\Phi):=
\min_{\ell>0}\left(\frac{A(\Phi)}{\ell^3}+2\lambda\ell\right)
=\frac83\left(\frac32\right)^{1/4}\lambda^{3/4}A(\Phi)^{1/4},
\qquad
\ell_\lambda(\Phi)=\left(\frac{3A(\Phi)}{2\lambda}\right)^{1/4}.
\end{equation}
Thus $M_\lambda$ is continuous, and strict comparisons for $M_\lambda$ are equivalent to strict comparisons for $A$.

\subsection{Boundary-layer asymptotics and good windows}\label{subsec:tw-boundary-layer}

The sharp turning inequality identifies the unconstrained equality profile
\begin{equation}\label{eq:tw-H}
H(t):=\frac32t-\frac12t^3.
\end{equation}
If
\[
\phi(t)=\frac\pi2+\Phi H(t)+h(t),
\qquad
h\in X:=\{h\in H^2(0,1):h(0)=h(1)=h'(1)=0\},
\]
then the closure condition is
\begin{equation}\label{eq:tw-closure-cos}
\int_0^1\cos(\Phi H(t)+h(t))\,dt=0,
\end{equation}
and the bending energy splits exactly:
\begin{equation}\label{eq:tw-energy-splitting}
\int_0^1(\phi'')^2\,dt
=3\Phi^2+\int_0^1(h'')^2\,dt.
\end{equation}
Indeed, $H''(t)=-3t$, $\int_0^1(H'')^2dt=3$, and the cross term vanishes by integration by parts using $h(0)=h(1)=h'(1)=0$.

The equality profile's closure defect is controlled by the stationary endpoint $t=1$.
\begin{lemma}[Endpoint stationary phase]\label{lem:tw-stationary-phase}
As $\Phi\to+\infty$,
\begin{equation}\label{eq:tw-stationary-phase}
\int_0^1\cos(\Phi H(t))\,dt
=
\sqrt{\frac{\pi}{6\Phi}}\cos\left(\Phi-\frac\pi4\right)+o(\Phi^{-1/2}).
\end{equation}
\end{lemma}

\begin{proof}
Put $x=1-t$.  Since
\[
H(1-x)=1-\frac32x^2+\frac12x^3,
\]
the change of variables $y=\sqrt{\Phi}\,x$ gives
\[
\int_0^1 e^{i\Phi H(t)}\,dt
=
\Phi^{-1/2}e^{i\Phi}\int_0^{\sqrt\Phi}\exp\left(-\frac{3i}{2}y^2+\frac{i}{2}\Phi^{-1/2}y^3\right)dy.
\]
On bounded $y$-intervals the integrand converges uniformly to $e^{-3iy^2/2}$; on the tail, integration by parts using the nonvanishing derivative of the phase gives a uniform $O(R^{-1})$ bound after first cutting at $y=R$.  Therefore
\[
\int_0^1 e^{i\Phi H(t)}\,dt
=
\Phi^{-1/2}e^{i\Phi}\int_0^\infty e^{-3iy^2/2}\,dy+o(\Phi^{-1/2}).
\]
The Fresnel integral is
\[
\int_0^\infty e^{-3iy^2/2}\,dy=e^{-i\pi/4}\sqrt{\frac\pi6},
\]
and taking real parts gives \eqref{eq:tw-stationary-phase}.
\end{proof}

The next estimate says that any correction with energy $o(\Phi^{3/2})$ is invisible in the natural boundary layer of width $\Phi^{-1/2}$.
\begin{proposition}[Boundary-layer stability]\label{prop:tw-boundary-layer-stability}
Let $\Phi_m\to+\infty$, write $\Phi_m\equiv\rho_m\pmod{2\pi}$ with $\rho_m\to\rho\in[0,2\pi)$, and let $h_m\in X$ satisfy
\[
\int_0^1(h_m'')^2\,dt=o(\Phi_m^{3/2}).
\]
Then
\begin{equation}\label{eq:tw-boundary-layer-stability}
\sqrt{\Phi_m}\int_0^1\cos(\Phi_m H(t)+h_m(t))\,dt
\longrightarrow
\sqrt{\frac\pi6}\cos\left(\rho-\frac\pi4\right).
\end{equation}
\end{proposition}

\begin{proof}
Set
\[
g_m(y):=h_m\left(1-\frac{y}{\sqrt{\Phi_m}}\right),
\qquad 0\le y\le\sqrt{\Phi_m}.
\]
Since $h_m(1)=h_m'(1)=0$,
\[
g_m(0)=g_m'(0)=0,
\qquad
\int_0^{\sqrt{\Phi_m}}|g_m''(y)|^2\,dy
=
\Phi_m^{-3/2}\int_0^1|h_m''(t)|^2\,dt\to0.
\]
Thus $g_m\to0$ in $C^1_{\mathrm{loc}}([0,\infty))$.  The rescaled phase is
\[
\Psi_m(y):=\Phi_m H\left(1-\frac{y}{\sqrt{\Phi_m}}\right)+g_m(y)
=
\Phi_m-\frac32y^2+\frac12\Phi_m^{-1/2}y^3+g_m(y).
\]
On every fixed interval $[0,R]$, $\Psi_m(y)\to \rho-3y^2/2$ modulo $2\pi$, uniformly.  Hence the corresponding integrals converge on $[0,R]$.

For the tail, put $M_m:=\sqrt{\Phi_m}$ and
$\varepsilon_m:=\|g_m''\|_{L^2(0,M_m)}\to0$.  Then
\[
\Psi_m'(y)=-3y+\frac{3y^2}{2\sqrt{\Phi_m}}+g_m'(y),
\qquad |g_m'(y)|\le y^{1/2}\varepsilon_m.
\]
For fixed $R\ge1$ and all large $m$, $|\Psi_m'(y)|\ge c y$ on $R\le y\le M_m$.  Since
\[
\Psi_m''(y)=-3+\frac{3y}{\sqrt{\Phi_m}}+g_m''(y),
\]
integration by parts gives the explicit estimate
\[
\begin{aligned}
\left|\int_R^{M_m}e^{i\Psi_m(y)}\,dy\right|
&\le \frac{C}{R}+C\int_R^{M_m}\frac{1+|g_m''(y)|}{y^2}\,dy\\
&\le \frac{C}{R}
+C\varepsilon_m\left(\int_R^\infty y^{-4}\,dy\right)^{1/2}\\
&\le \frac{C}{R}+C\varepsilon_mR^{-3/2}.
\end{aligned}
\]
Consequently
$\limsup_{m\to\infty}|\int_R^{M_m}e^{i\Psi_m(y)}\,dy|\le C/R$.
Letting first $m\to\infty$ and then $R\to\infty$ yields
\[
\int_0^{\sqrt{\Phi_m}}\cos\Psi_m(y)\,dy
\to
\int_0^\infty\cos\left(\rho-\frac32y^2\right)dy
=\sqrt{\frac\pi6}\cos\left(\rho-\frac\pi4\right),
\]
which is \eqref{eq:tw-boundary-layer-stability} after undoing the change of variables.
\end{proof}

\begin{lemma}[Upper bound at the good phases]\label{lem:tw-upper-good}
Let $\rho\in\{3\pi/4,7\pi/4\}$ and $\Phi_n:=2\pi n+\rho$.  Then
\begin{equation}\label{eq:tw-upper-good}
A(\Phi_n)\le 3\Phi_n^2+o(\Phi_n^{3/2})
\qquad\text{as }n\to\infty.
\end{equation}
\end{lemma}

\begin{proof}
By Lemma~\ref{lem:tw-stationary-phase}, the uncorrected closure defect
\[
D_n:=\int_0^1\cos(\Phi_nH(t))\,dt
\]
satisfies $D_n=o(\Phi_n^{-1/2})$ at the two phases $\rho=3\pi/4,7\pi/4$.
Choose $\psi\in C_c^\infty([0,\infty))$ with $\psi(0)=\psi'(0)=0$ and
\[
L_\rho:=-\int_0^\infty\sin\left(\rho-\frac32y^2\right)\psi(y)\,dy\ne0.
\]
For $a\in\R$ define the boundary-layer perturbation
\[
h_{n,a}(t):=a\,\psi(\sqrt{\Phi_n}(1-t)).
\]
For all large $n$, $h_{n,a}\in X$.  Let
\[
F_n(a):=\int_0^1\cos(\Phi_nH(t)+h_{n,a}(t))\,dt.
\]
Then $F_n(0)=D_n$ and
\[
F_n'(0)=\Phi_n^{-1/2}L_\rho+O(\Phi_n^{-1}).
\]
Moreover,
\[
F_n''(a)=-\int_0^1
\cos\bigl(\Phi_nH(t)+h_{n,a}(t)\bigr)
\psi\bigl(\sqrt{\Phi_n}(1-t)\bigr)^2\,dt,
\]
and hence, uniformly in $a$,
\[
|F_n''(a)|\le\Phi_n^{-1/2}\int_0^\infty\psi(y)^2\,dy.
\]
Choose $b_n\downarrow0$ with $|D_n|/(b_n\Phi_n^{-1/2})\to0$.  Taylor's formula gives
\[
F_n(\pm b_n)=D_n\pm b_n\Phi_n^{-1/2}
\bigl(L_\rho+O(\Phi_n^{-1/2})\bigr)+O(b_n^2\Phi_n^{-1/2}).
\]
Because $L_\rho\ne0$, the two values have opposite signs for all large $n$.  Thus there is $a_n$ with
$|a_n|\le b_n$ and $F_n(a_n)=0$.  Finally, the perturbation has the exact energy
\[
\int_0^1(h_{n,a}'')^2\,dt
=a^2\Phi_n^{3/2}\int_0^\infty(\psi''(y))^2\,dy.
\]
By the exact splitting \eqref{eq:tw-energy-splitting},
\[
A(\Phi_n)\le 3\Phi_n^2+
\int_0^1(h_{n,a_n}'')^2\,dt
\le 3\Phi_n^2+C b_n^2\Phi_n^{3/2}
=3\Phi_n^2+o(\Phi_n^{3/2}).
\]
\end{proof}

\begin{lemma}[Lower bound away from the good phases]\label{lem:tw-lower-away}
Let $K\subset[0,2\pi)$ be compact and assume
\[
\cos\left(\rho-\frac\pi4\right)\ne0
\qquad\text{for every }\rho\in K.
\]
Then there exist $c_K>0$ and $N_K\in\N$ such that, whenever $\Phi\ge N_K$ and $\Phi\pmod{2\pi}\in K$,
\begin{equation}\label{eq:tw-lower-away}
A(\Phi)\ge 3\Phi^2+c_K\Phi^{3/2}.
\end{equation}
\end{lemma}

\begin{proof}
If the conclusion failed, then for every $m\in\N$ there would exist $\Phi_m\ge m$, with
$\rho_m:=\Phi_m\pmod{2\pi}\in K$, such that
\[
A(\Phi_m)<3\Phi_m^2+\frac1m\Phi_m^{3/2}.
\]
Let $\phi_m$ be a minimiser for $A(\Phi_m)$ and set
$h_m:=\phi_m-\pi/2-\Phi_mH$.  Then $h_m\in X$, it satisfies the exact closure condition
\eqref{eq:tw-closure-cos}, and the splitting \eqref{eq:tw-energy-splitting} gives
\[
\int_0^1(h_m'')^2\,dt=A(\Phi_m)-3\Phi_m^2
<\frac1m\Phi_m^{3/2}=o(\Phi_m^{3/2}).
\]
After passing to a subsequence, $\rho_m\to\rho\in K$.  Proposition~\ref{prop:tw-boundary-layer-stability} gives
\[
0=\sqrt{\Phi_m}\int_0^1\cos(\Phi_mH+h_m)\,dt
\to
\sqrt{\frac\pi6}\cos\left(\rho-\frac\pi4\right)\ne0,
\]
a contradiction.
\end{proof}

\begin{theorem}[Infinitely many good windows]\label{thm:tw-good-windows}
Let $\rho_1=3\pi/4$ and $\rho_2=7\pi/4$.  There exists $\delta>0$ such that, for every sufficiently large $n$ and each $j\in\{1,2\}$, the window
\[
I_{n,j}:=[2\pi n+\rho_j-\delta,\,2\pi n+\rho_j+\delta]
\]
is good for $A$ and therefore also good for $M_\lambda$:
\begin{equation}\label{eq:tw-good-window}
\min_{\Phi\in I_{n,j}}A(\Phi)
<
\min\{A(2\pi n+\rho_j-\delta),A(2\pi n+\rho_j+\delta)\}.
\end{equation}
\end{theorem}

\begin{proof}
Choose $\delta>0$ so small that the endpoint phases $\rho_j\pm\delta$ are all separated from the zero set of $\cos(\rho-\pi/4)$.  At the centre $\Phi_{n,j}:=2\pi n+\rho_j$, Lemma~\ref{lem:tw-upper-good} gives
\[
A(\Phi_{n,j})\le3\Phi_{n,j}^2+o(\Phi_{n,j}^{3/2}).
\]
At either endpoint $\Phi_{n,j}^{\pm}:=2\pi n+\rho_j\pm\delta$, Lemma~\ref{lem:tw-lower-away} gives
\[
A(\Phi_{n,j}^{\pm})
\ge 3(\Phi_{n,j}^{\pm})^2+c\Phi_{n,j}^{3/2}
\]
for all large $n$, with $c>0$ independent of $n$.  Since
\[
3(\Phi_{n,j}^{\pm})^2=3\Phi_{n,j}^2+O(\Phi_{n,j}),
\]
the positive $c\Phi_{n,j}^{3/2}$ term dominates the $O(\Phi_{n,j})$ change in the quadratic term and the centre error.  Thus the centre value is strictly below both endpoint values for all large $n$.  The equivalence for $M_\lambda$ follows from \eqref{eq:tw-Mlambda}.
\end{proof}

\subsection{Window minimisers and midpoint transversality}\label{subsec:tw-window-transversality}

\begin{proposition}[Existence of window minimisers]\label{prop:tw-window-existence}
Let $I\subset(0,\infty)$ be compact.  Then
\[
\inf_{(\ell,\theta)\in\mathcal{A}^{(2)}(I)}J_\lambda^{(2)}(\ell,\theta)
\]
is attained, and the minimum equals $\min_{\Phi\in I}M_\lambda(\Phi)$.
\end{proposition}

\begin{proof}
By Lemma~\ref{lem:tw-A-continuity} and \eqref{eq:tw-Mlambda}, $M_\lambda$ is continuous, so it attains its minimum at some $\Phi_\ast\in I$.  Choose a minimiser $\phi_\ast\in\mathcal{B}(\Phi_\ast)$ for $A(\Phi_\ast)$ and set
\[
\ell_\ast=\left(\frac{3A(\Phi_\ast)}{2\lambda}\right)^{1/4}.
\]
Define $\theta_\ast(s)=\phi_\ast(s/\ell_\ast)$ on $[0,\ell_\ast]$ and extend evenly to $[0,2\ell_\ast]$.  Then $(\ell_\ast,\theta_\ast)\in\mathcal{A}^{(2)}(I)$ and its energy is $M_\lambda(\Phi_\ast)$.
Conversely, let $(\ell,\theta)\in\mathcal A^{(2)}(I)$, set
\[
\Phi:=\theta(\ell)-\frac\pi2\in I,
\qquad
\phi(t):=\theta(\ell t),\quad 0\le t\le1.
\]
Then $\phi\in\mathcal B(\Phi)$, and therefore
\[
\begin{aligned}
J_\lambda^{(2)}(\ell,\theta)
&=\ell^{-3}\int_0^1(\phi'')^2\,dt+2\lambda\ell\\
&\ge \ell^{-3}A(\Phi)+2\lambda\ell\\
&\ge M_\lambda(\Phi)
\ge \min_{\Psi\in I}M_\lambda(\Psi).
\end{aligned}
\]
\end{proof}

\begin{proposition}[Interior window minimisers are balanced]\label{prop:tw-interior-balanced}
Let $(\ell,\theta)$ minimise $J^{(2)}_\lambda$ over $\mathcal{A}^{(2)}(I)$, and suppose
\[
\Phi(\theta):=\theta(\ell)-\frac\pi2\in\operatorname{int}I.
\]
Then there exists $a\in\R$ such that $\theta$ solves \eqref{eq:thetaequiv}
on the two open half-intervals, satisfies $\theta_{ss}(0)=\theta_{ss}(2\ell)=0$, and obeys the midpoint transversality condition
\begin{equation}\label{eq:tw-midpoint-transversality}
\theta_{sss}(\ell^-)=0.
\end{equation}
Consequently the even midpoint reflection is smooth across $s=\ell$.
\end{proposition}

\begin{proof}
Because $\Phi(\theta)$ lies in the interior of $I$, small admissible variations may change $\theta(\ell)$.  In half-domain variables this means that, for the corresponding minimising shape $\phi$, admissible variations $\eta$ satisfy
\[
\eta(0)=0,
\qquad
\eta'(1)=0,
\]
but no condition is imposed on $\eta(1)$.  Set
\[
\mathcal V:=\{\eta\in H^2(0,1):\eta(0)=0,\ \eta'(1)=0\}.
\]
The same constraint qualification used in Lemma~\ref{lem:tw-A-continuity} gives a multiplier $a\in\R$ such that
\[
2\int_0^1\phi''\eta''\,dt+a\int_0^1\cos\phi\,\eta\,dt=0
\]
for every $\eta\in\mathcal V$.  Integrating by parts twice yields the boundary term
\[
2[\phi''\eta'-\phi'''\eta]_0^1.
\]
The term at $t=1$ involving $\eta'(1)$ vanishes, while $\eta(1)$ is arbitrary, so
\[
\phi'''(1)=0.
\]
Likewise, $\eta'(0)$ is arbitrary and $\eta(0)=0$, so $\phi''(0)=0$.  Variations compactly supported in $(0,1)$
give the Euler--Lagrange equation.
Since $\theta_{sss}(\ell^-)=\ell^{-3}\phi'''(1)$, this is \eqref{eq:tw-midpoint-transversality}.  The remaining Euler-Lagrange equation and the endpoint natural condition are obtained by the same integration-by-parts argument with variations supported away from, or with free derivative at, the outer endpoint.

Evenness gives $\theta_s(\ell)=0$ and $\theta_{sss}(\ell^+)=-\theta_{sss}(\ell^-)$.  Hence \eqref{eq:tw-midpoint-transversality} makes the one-sided third derivatives agree and vanish.  The one-sided Cauchy data through order three match, so the ODE bootstraps the reflected angle to $C^\infty$ through the midpoint.
\end{proof}

\subsection{The geometric Euler-Lagrange equation}\label{subsec:tw-geometric-upgrade}

Let
\begin{equation}\label{eq:E-k-def}
\mathcal{E}[k]:=k_{ssss}+k^2k_{ss}-\frac12kk_s^2.
\end{equation}
For a smooth closed unit-speed curve, the normal first variation of $J_\lambda$ is
\begin{equation}\label{eq:J-first-var-geometric}
\delta J_\lambda[fN]
=-\int_\Gamma\big(\mathcal{E}[k]+\lambda k\big)f\,ds.
\end{equation}
Indeed, with the convention $T_s=kN$, one has $\delta ds=-kf\,ds$, $\delta k=f_{ss}+k^2f$, and $\delta\partial_s=kf\partial_s$; integrating by parts on the closed curve gives \eqref{eq:J-first-var-geometric}.  Tangential variations are reparametrisations.

\begin{proposition}[Hamiltonian upgrade]\label{prop:tw-hamiltonian-upgrade}
Let $\theta\in C^\infty([0,2\ell])$ be even about $\ell$, satisfy
\[
\theta(0)=\theta(2\ell)=\frac\pi2,
\qquad
\int_0^\ell\sin\theta\,ds=0,
\qquad
\theta_{ss}(0)=\theta_{ss}(2\ell)=0,
\]
and solve
\[
\theta_{ssss}=c\cos\theta.
\]
Set $k=\theta_s$, $p=k_s=\theta_{ss}$, and $q=k_{ss}=\theta_{sss}$.  Then
\[
\Lambda:=-kq+\frac12p^2+c\sin\theta
\]
is constant and $\mathcal{E}[k]=-\Lambda k$.  If, in addition, $(\ell,\theta)$ is stationary in the length variable for $J^{(2)}_\lambda$, then $\Lambda=\lambda$ and hence
\begin{equation}\label{eq:tw-full-EL}
\mathcal{E}[k]+\lambda k=0.
\end{equation}
\end{proposition}

\begin{proof}
Since $q_s=\theta_{ssss}=c\cos\theta$ and $k=\theta_s$,
\[
\Lambda_s=-(pq+kq_s)+pp_s+c\cos\theta\,\theta_s=0.
\]
Furthermore,
\[
k_{sss}=q_s=c\cos\theta,
\qquad
k_{ssss}=-ck\sin\theta,
\]
and therefore
\[
\mathcal{E}[k]
=-ck\sin\theta+k^2q-\frac12kp^2
=-k\left(-kq+\frac12p^2+c\sin\theta\right)
=-\Lambda k.
\]

It remains to identify $\Lambda$.  Integrating over the doubled interval and using evenness together with the closure
condition gives
\[
2\ell\Lambda
=
\int_0^{2\ell}\left(-kq+\frac12p^2+c\sin\theta\right)ds
=
\int_0^{2\ell}p^2\,ds+\frac12\int_0^{2\ell}p^2\,ds,
\]
because
\[
\int_0^{2\ell}-kq\,ds
=-[kp]_0^{2\ell}+\int_0^{2\ell}p^2\,ds
=
\int_0^{2\ell}p^2\,ds
\]
and $p(0)=p(2\ell)=0$.  To spell out length stationarity, keep the half-domain shape
$\phi(t):=\theta(\ell t)$ fixed and put $A:=\int_0^1(\phi'')^2\,dt$.  This rescaling preserves the endpoint and
closure constraints, while evenness gives
\[
J_\lambda^{(2)}(\ell,\phi)=\ell^{-3}A+2\lambda\ell,
\qquad
\int_0^{2\ell}p^2\,ds=2\ell^{-3}A.
\]
Stationarity in the length variable therefore means
\[
0=\frac{d}{d\ell}\bigl(\ell^{-3}A+2\lambda\ell\bigr)
=-3\ell^{-4}A+2\lambda.
\]
Multiplying by $\ell$ yields
\[
2\lambda\ell=3\ell^{-3}A=\frac32\int_0^{2\ell}p^2\,ds.
\]
Thus $2\ell\Lambda=2\lambda\ell$, so $\Lambda=\lambda$.
\end{proof}

\subsection{Proof of Theorem \ref{thm:main-infinite}}\label{subsec:tw-proof-main}

\begin{proof}[Proof of Theorem~\ref{thm:main-infinite}]
Fix $\lambda>0$.  By Theorem~\ref{thm:tw-good-windows}, choose $\delta>0$ and $N_0$ such that every window
\[
I_{n,j}=[2\pi n+\rho_j-\delta,2\pi n+\rho_j+\delta],
\qquad n\ge N_0,
\quad j\in\{1,2\},
\]
is good for $M_\lambda$.  Proposition~\ref{prop:tw-window-existence} gives a minimiser $(\ell_{n,j},\theta_{n,j})$ over $\mathcal{A}^{(2)}(I_{n,j})$.  Since the window is good, the selected turning
\[
\Phi_{n,j}^\dagger=\theta_{n,j}(\ell_{n,j})-\frac\pi2
\]
lies in the interior of $I_{n,j}$.  Proposition~\ref{prop:tw-interior-balanced} then gives
\[
(\theta_{n,j})_{sss}(\ell_{n,j}^-)=0,
\]
so the doubled arc is smooth through the midpoint.  Proposition~\ref{prop:doubled-xaxis-smooth} gives a smooth closed unit-speed immersion after $x$-axis reflection.

The same proposition and Lemma~\ref{lem:even-theta-point-symmetry} give point-reflection and $x$-axis symmetries.  The quarter-arc turning identities follow directly from the endpoint angles and from the reflection rule, with $\Phi_{n,j}^{\dagger}$ in place of the formerly prescribed fixed value.

Because $(\ell_{n,j},\theta_{n,j})$ minimises over the open length variable $\ell>0$, it is stationary under the
fixed-shape length variation used in Proposition~\ref{prop:tw-hamiltonian-upgrade}.  That proposition therefore gives
$\mathcal{E}[k]+\lambda k=0$ on the smooth doubled arc.  The $x$-axis reflection preserves this equation, so the closed
curve satisfies it everywhere.  The first variation formula \eqref{eq:J-first-var-geometric} then shows that
$\Gamma_{n,j}$ is a critical point of $J_\lambda$ among smooth closed immersions.

The lower bound \eqref{eq:main-lemniscate-energy-lower} is exactly Proposition~\ref{prop:closed-lb} applied to the fundamental quarter arc with turning $\Phi_{n,j}^{\dagger}$.  Since
\[
\Phi_{n,j}^{\dagger}\in[2\pi n+\rho_j-\delta,2\pi n+\rho_j+\delta],
\]
we have $\Phi_{n,j}^{\dagger}\to+\infty$ as $n\to\infty$, and hence the energies tend to infinity.

Finally, fix either $j=1$ or $j=2$.  Since $J_\lambda[\Gamma_{n,j}]\to\infty$, we may choose indices
$n_1<n_2<\cdots$ recursively so that the corresponding energies are strictly increasing.  Congruent curves have the
same value of $J_\lambda$, so $\{\Gamma_{n_r,j}\}_{r\ge1}$ is pairwise non-congruent.  This proves that the collection
contains infinitely many geometrically distinct closed critical points.
\end{proof}

\section{Proof of Theorem \ref{thm:circular-waves-infinite}}\label{sec:circular-waves}

This section illustrates, in a setting where global closure is {not} imposed, what becomes of the degenerate
$\lambda=0$ family of circular stationary solutions under a perturbative shooting construction.
We construct countably many one-parameter families of {circular waves}:
smooth stationary immersions $\Gamma:\R\to\R^2$ for the length-penalised ideal functional whose curvature jet and
tangent are periodic, but whose position drifts by a nonzero translation per period.  Consequently these waves are not
closed.  Since such an immersion has infinite total length, $J_\lambda[\Gamma]$ itself is infinite; here
\emph{stationary} means that the geometric Euler--Lagrange equation holds, equivalently that the first variation
vanishes for compactly supported variations.  Each family bifurcates from a {multiply-covered semicircle} at
$\lambda=0$.

\subsection{Stationary ODE and multiply-covered semicircle base solutions}\label{subsec:cw-ode-base}

We use the curvature-jet shooting system
\[
\frac{d}{ds}
\begin{pmatrix}k\\p\\q\\\theta\end{pmatrix}
=
\begin{pmatrix}p\\q\\c\cos\theta\\k\end{pmatrix},
\qquad
(k,p,q,\theta)(0)=\left(a,0,b,\frac\pi2\right),
\qquad
c=ab+\lambda.
\]
Thus $\theta$ satisfies $\theta_{ssss}=c\cos\theta$.
We also reconstruct position via
\begin{equation}\label{eq:xy-cwN}
x_s=\cos\theta,\qquad y_s=\sin\theta,\qquad (x(0),y(0))=(0,0),
\end{equation}
so that $\gamma=(x,y)$ is the reconstructed unit-speed curve.

\subsection{Endpoint data and the winding index $N$}

Fix $\ell>0$ and an integer $N\in\N_0$.
We seek solutions of \eqref{eq:shooting-cwN} on $[0,\ell]$ satisfying the endpoint conditions
\begin{equation}\label{eq:endpoint-data-cwN}
\theta(\ell)=\frac{\pi}{2}+(2N+1)\pi,
\qquad
p(\ell)=k_s(\ell)=0.
\end{equation}
The angle condition imposes lifted (unwrapped) turning $(2N+1)\pi$ on $[0,\ell]$ while keeping the terminal tangent
downward (mod $2\pi$).  The seam condition $p(\ell)=0$ ensures smooth reflection of the curvature profile at $s=\ell$.

Note that the ODE itself forces the third curvature derivative to vanish at such seam endpoints:
since $k_{sss}=q_s=c\cos\theta$, we have
\begin{equation}\label{eq:ksss-vanish-cwN}
k_{sss}(0)=c\cos\theta(0)=0,\qquad k_{sss}(\ell)=c\cos\theta(\ell)=0,
\end{equation}
because $\cos\theta(0)=0$ and $\cos\theta(\ell)=0$ by \eqref{eq:endpoint-data-cwN}.  Equivalently,
$\theta_{ssss}(0)=\theta_{ssss}(\ell)=0$.

\subsection{A square shooting map and an IFT family for each $N$}\label{subsec:cw-ift}

To obtain a one-parameter family (and to avoid the constant-curvature degeneracy), we prescribe the even jet value
\begin{equation}\label{eq:qell-eps-cwN}
q(\ell)=\varepsilon
\end{equation}
as a small branch parameter $\varepsilon$.

\begin{definition}[Multiply-covered semicircle shooting map]\label{def:Phi-cwN}
Fix $\ell>0$ and $N\in\N_0$.  For $(a,b,\lambda)\in\R^3$ let $(k,p,q,\theta)$ denote the solution of
\eqref{eq:shooting-cwN} on $[0,\ell]$.  Define
\begin{equation}\label{eq:Phi-cwN}
\Phi_N:\R^3\times\R\to\R^3,\qquad
\Phi_N(a,b,\lambda;\varepsilon):=
\begin{pmatrix}
p(\ell;a,b,\lambda)\\[2pt]
\theta(\ell;a,b,\lambda)-\big(\frac{\pi}{2}+(2N+1)\pi\big)\\[2pt]
q(\ell;a,b,\lambda)-\varepsilon
\end{pmatrix}.
\end{equation}
\end{definition}

\begin{lemma}[Base solution at $\lambda=0$]\label{lem:semicircle-base-N}
Let
\begin{equation}\label{eq:a-star-N}
a_{N,\ast}:=\frac{(2N+1)\pi}{\ell}.
\end{equation}
Then $\Phi_N(a_{N,\ast},0,0;0)=0$, with explicit solution
\[
k\equiv a_{N,\ast},\qquad p\equiv 0,\qquad q\equiv 0,\qquad \theta(s)=\frac{\pi}{2}+a_{N,\ast} s.
\]
The corresponding curve is an arc of the circle of radius $1/a_{N,\ast}$ that traverses $(2N+1)$ half-turns, i.e.\ a
semicircle plus $N$ full circles.
\end{lemma}

\begin{proof}
At $(a,b,\lambda)=(a_{N,\ast},0,0)$ we have $c=ab+\lambda=0$, so $q_s=0$ and with $q(0)=0$ we get $q\equiv 0$.
Then $p_s=q=0$ and $p(0)=0$ give $p\equiv 0$, hence $k_s=p=0$ and $k(0)=a_{N,\ast}$ give $k\equiv a_{N,\ast}$.
Finally $\theta_s=k$ gives $\theta(s)=\pi/2+a_{N,\ast} s$, so
\[
\theta(\ell)=\frac{\pi}{2}+a_{N,\ast}\ell=\frac{\pi}{2}+(2N+1)\pi,
\]
and thus $\Phi_N(a_{N,\ast},0,0;0)=0$.
\end{proof}

\begin{lemma}[Nondegenerate linearisation]\label{lem:jacobian-cwN}
At the base point $(a_{N,\ast},0,0;0)$ the Jacobian $D_{(a,b,\lambda)}\Phi_N$ is invertible. In fact,
\begin{equation}\label{eq:det-jacobian-cwN}
\det D_{(a,b,\lambda)}\Phi_N(a_{N,\ast},0,0;0)=\frac{\ell^3}{(2N+1)\pi}\neq 0.
\end{equation}
\end{lemma}

\begin{proof}
Fix $\ell>0$ and $N\in\N_0$, and write
\[
a_\ast:=a_{N,\ast}=\frac{(2N+1)\pi}{\ell}.
\]
For $(a,b,\lambda)$ near $(a_\ast,0,0)$ let $(k,p,q,\theta)(\cdot;a,b,\lambda)$ denote the solution of
\eqref{eq:shooting-cwN} on $[0,\ell]$, and recall that
\[
\Phi_N(a,b,\lambda;0)=
\begin{pmatrix}
p(\ell;a,b,\lambda)\\[2pt]
\theta(\ell;a,b,\lambda)-\big(\frac{\pi}{2}+(2N+1)\pi\big)\\[2pt]
q(\ell;a,b,\lambda)
\end{pmatrix}.
\]
Since the right-hand side of \eqref{eq:shooting-cwN} is smooth in $(k,p,q,\theta)$ and in the parameters $(a,b,\lambda)$
(through $c=ab+\lambda$) and the interval length is fixed, standard ODE theory gives that
$(a,b,\lambda)\mapsto (p(\ell),\theta(\ell),q(\ell))$ is $C^\infty$ near $(a_\ast,0,0)$.  Hence the Jacobian
$D_{(a,b,\lambda)}\Phi_N$ exists and can be computed by differentiating the ODE.

At the base point $(a_\ast,0,0)$ we have $c_\ast=ab+\lambda=0$, so the base solution is explicitly
\[
k_\ast\equiv a_\ast,\qquad p_\ast\equiv 0,\qquad q_\ast\equiv 0,\qquad
\theta_\ast(s)=\frac{\pi}{2}+a_\ast s,
\]
and therefore
\[
\cos\theta_\ast(s)=\cos\Big(\frac{\pi}{2}+a_\ast s\Big)=-\sin(a_\ast s).
\]

For $\mu\in\{a,b,\lambda\}$ set
\[
k_\mu:=\partial_\mu k\big|_\ast,\quad p_\mu:=\partial_\mu p\big|_\ast,\quad
q_\mu:=\partial_\mu q\big|_\ast,\quad \theta_\mu:=\partial_\mu\theta\big|_\ast.
\]
Differentiating \eqref{eq:shooting-cwN} at the base solution and using $c_\ast=0$ gives the decoupled variational system
\[
(k_\mu)_s=p_\mu,\qquad (p_\mu)_s=q_\mu,\qquad (q_\mu)_s=c_\mu\,\cos\theta_\ast,\qquad (\theta_\mu)_s=k_\mu,
\]
with initial data
\[
k_\mu(0)=\partial_\mu a,\qquad p_\mu(0)=0,\qquad q_\mu(0)=\partial_\mu b,\qquad \theta_\mu(0)=0.
\]
Here $c_\mu=\partial_\mu(ab+\lambda)|_\ast$ is
\[
c_a=b_\ast=0,\qquad c_b=a_\ast,\qquad c_\lambda=1.
\]

\medskip
\noindent\textbf{(i) The $a$-derivatives.}
Since $c_a=0$ and $q_a(0)=0$, we have $(q_a)_s=0$ and thus $q_a\equiv 0$, hence $p_a\equiv 0$ and $k_a\equiv 1$.
Therefore $\theta_a(s)=s$ and in particular
\begin{equation}\label{eq:partials-a}
\partial_a p(\ell)\big|_\ast=0,\qquad \partial_a q(\ell)\big|_\ast=0,\qquad
\partial_a\theta(\ell)\big|_\ast=\ell.
\end{equation}

\medskip
\noindent\textbf{(ii) The $b$-derivatives.}
Here $c_b=a_\ast$ and $q_b(0)=1$, so
\[
(q_b)_s=a_\ast\cos\theta_\ast=-a_\ast\sin(a_\ast s).
\]
Integrating gives
\[
q_b(s)=1+\int_0^s -a_\ast\sin(a_\ast\tau)\,d\tau
=1+\big(\cos(a_\ast s)-1\big)=\cos(a_\ast s).
\]
Hence
\[
p_b(s)=\int_0^s q_b(\tau)\,d\tau=\int_0^s \cos(a_\ast\tau)\,d\tau=\frac{1}{a_\ast}\sin(a_\ast s).
\]
Using $a_\ast\ell=(2N+1)\pi$ gives $\sin(a_\ast\ell)=0$ and $\cos(a_\ast\ell)=-1$, so
\begin{equation}\label{eq:partials-b}
\partial_b p(\ell)\big|_\ast=0,\qquad \partial_b q(\ell)\big|_\ast=\cos(a_\ast\ell)=-1.
\end{equation}

\medskip
\noindent\textbf{(iii) The $\lambda$-derivative needed for the determinant.}
Here $c_\lambda=1$ and $q_\lambda(0)=0$, so
\[
(q_\lambda)_s=\cos\theta_\ast=-\sin(a_\ast s),
\qquad
q_\lambda(s)=\int_0^s -\sin(a_\ast\tau)\,d\tau=\frac{\cos(a_\ast s)-1}{a_\ast}.
\]
Then
\[
p_\lambda(\ell)=\int_0^\ell q_\lambda(s)\,ds
=\int_0^\ell \frac{\cos(a_\ast s)-1}{a_\ast}\,ds
=\frac{1}{a_\ast^2}\sin(a_\ast\ell)-\frac{\ell}{a_\ast}
=-\frac{\ell}{a_\ast},
\]
since $\sin(a_\ast\ell)=0$. Thus
\begin{equation}\label{eq:partials-lambda}
\partial_\lambda p(\ell)\big|_\ast=-\frac{\ell}{a_\ast}=-\frac{\ell^2}{(2N+1)\pi}.
\end{equation}

\medskip
Collecting \eqref{eq:partials-a}, \eqref{eq:partials-b}, \eqref{eq:partials-lambda}, the Jacobian matrix
$D_{(a,b,\lambda)}\Phi_N(a_\ast,0,0;0)$ has the form
\[
D_{(a,b,\lambda)}\Phi_N\big|_\ast
=
\begin{pmatrix}
\partial_a p(\ell) & \partial_b p(\ell) & \partial_\lambda p(\ell)\\
\partial_a \theta(\ell) & \partial_b \theta(\ell) & \partial_\lambda \theta(\ell)\\
\partial_a q(\ell) & \partial_b q(\ell) & \partial_\lambda q(\ell)
\end{pmatrix}_{\!\!*}
=
\begin{pmatrix}
0 & 0 & -\ell/a_\ast\\
\ell & * & *\\
0 & -1 & *
\end{pmatrix},
\]
where the starred entries are not needed for the determinant.
Expanding along the first row,
\[
\det D_{(a,b,\lambda)}\Phi_N\big|_\ast
=
\Big(-\frac{\ell}{a_\ast}\Big)\det
\begin{pmatrix}
\ell & *\\
0 & -1
\end{pmatrix}
=
\Big(-\frac{\ell}{a_\ast}\Big)\cdot(-\ell)
=
\frac{\ell^2}{a_\ast}
=
\frac{\ell^3}{(2N+1)\pi}\neq 0.
\]
This proves \eqref{eq:det-jacobian-cwN}, and in particular the Jacobian is invertible at the base point.
\end{proof}

\begin{theorem}[IFT family of semicircle-like stationary arcs for each $N$]\label{thm:IFT-cwN}
Fix $\ell>0$ and $N\in\N_0$, and set $a_{N,\ast}=(2N+1)\pi/\ell$.
There exist $\varepsilon_{0,N}>0$ and unique $C^\infty$ functions
\[
\varepsilon\longmapsto\big(a_N(\varepsilon),b_N(\varepsilon),\lambda_N(\varepsilon)\big)
\]
defined for $|\varepsilon|<\varepsilon_{0,N}$, with
\[
(a_N(0),b_N(0),\lambda_N(0))=(a_{N,\ast},0,0),
\]
such that the solution of \eqref{eq:shooting-cwN} satisfies
\begin{equation}\label{eq:IFT-endpoint-cwN}
p(\ell)=0,\qquad
\theta(\ell)=\frac{\pi}{2}+(2N+1)\pi,\qquad
q(\ell)=\varepsilon.
\end{equation}
For every sufficiently small $\varepsilon\neq 0$, the corresponding curvature profile is non-constant.
\end{theorem}

\begin{proof}
Apply the implicit function theorem to $\Phi_N$ in \eqref{eq:Phi-cwN} at the base point
$(a_{N,\ast},0,0;0)$ using Lemma~\ref{lem:semicircle-base-N} and Lemma~\ref{lem:jacobian-cwN}.
If $\varepsilon\neq 0$ then $q(\ell)=k_{ss}(\ell)=\varepsilon\neq 0$; hence $q$ cannot vanish identically and $k$
cannot be constant.
\end{proof}

\subsection{Quadratic emergence of $\lambda$ along each branch}\label{subsec:cw-lambda}

The branch parameter $\varepsilon=q(\ell)$ measures deviation from the constant-curvature multiply-covered semicircle.
As in the $N=0$ case, $\lambda$ becomes strictly positive at second order; only the constant $a_{N,\ast}$ changes.

\begin{lemma}[Quadratic emergence of $\lambda$]\label{lem:lambda-quadratic-cwN}
Along the IFT branch of Theorem~\ref{thm:IFT-cwN},
\begin{equation}\label{eq:lambda-exp-cwN}
\lambda_N(\varepsilon)=\frac{5}{4a_{N,\ast}^{2}}\,\varepsilon^2+o(\varepsilon^2)
\qquad\text{as }\varepsilon\to 0.
\end{equation}
In particular, $\lambda_N(\varepsilon)>0$ for all sufficiently small $\varepsilon\neq 0$.
\end{lemma}

\begin{proof}
Write
\[
a_\ast:=a_{N,\ast}=\frac{(2N+1)\pi}{\ell},
\]
and let $\varepsilon\mapsto\big(a_N(\varepsilon),b_N(\varepsilon),\lambda_N(\varepsilon)\big)$ be the IFT branch from
Theorem~\ref{thm:IFT-cwN}.  For each $\varepsilon$, let $(k,p,q,\theta)(\cdot;\varepsilon)$ denote the corresponding
solution of \eqref{eq:shooting-cwN} on $[0,\ell]$, and set
\[
c(\varepsilon):=a_N(\varepsilon)b_N(\varepsilon)+\lambda_N(\varepsilon).
\]
By standard smooth dependence of ODE solutions on parameters, the three endpoint functions
$p(\ell)$, $\theta(\ell)$, and $q(\ell)$ depend $C^\infty$-smoothly on $(a,b,\lambda)$ near $(a_\ast,0,0)$.
Hence the IFT branch is $C^\infty$ in $\varepsilon$ near $0$, and we may compute the Taylor expansion of
$\lambda_N(\varepsilon)$ up to second order.

At $\varepsilon=0$ we are at the base solution from Lemma~\ref{lem:semicircle-base-N}:
\[
k_\ast\equiv a_\ast,\qquad p_\ast\equiv 0,\qquad q_\ast\equiv 0,\qquad \theta_\ast(s)=\frac{\pi}{2}+a_\ast s,
\qquad c(0)=0.
\]
Denote $\dot f:=\partial_\varepsilon f|_{\varepsilon=0}$ and $\ddot f:=\partial_\varepsilon^2 f|_{\varepsilon=0}$.

\medskip
\noindent\textbf{Step 1: first derivatives and vanishing of $\dot\lambda_N(0)$.}
Differentiate \eqref{eq:shooting-cwN} along the branch at $\varepsilon=0$.  Since $c(0)=0$, we obtain the linear system
\[
\dot k_s=\dot p,\qquad \dot p_s=\dot q,\qquad \dot q_s=\dot c\,\cos\theta_\ast,\qquad \dot\theta_s=\dot k,
\]
with initial data $\dot k(0)=\dot a$, $\dot p(0)=0$, $\dot q(0)=\dot b$, $\dot\theta(0)=0$, where
$\dot c=a_\ast\dot b+\dot\lambda$ (because $c=ab+\lambda$ and $b(0)=0$).
Differentiating the endpoint conditions \eqref{eq:IFT-endpoint-cwN} gives
\[
\dot p(\ell)=0,\qquad \dot\theta(\ell)=0,\qquad \dot q(\ell)=1.
\]
Now $\cos\theta_\ast(s)=\cos(\frac{\pi}{2}+a_\ast s)=-\sin(a_\ast s)$, hence integrating $\dot q_s=\dot c\cos\theta_\ast$
yields
\[
\dot q(s)=\dot b+\frac{\dot c}{a_\ast}\big(\cos(a_\ast s)-1\big).
\]
Using $\cos(a_\ast\ell)=\cos((2N{+}1)\pi)=-1$, the condition $\dot q(\ell)=1$ becomes
\[
1=\dot b-\frac{2\dot c}{a_\ast}.
\]
Next, integrating $\dot p_s=\dot q$ and using $\sin(a_\ast\ell)=0$ gives
\[
\dot p(\ell)=\ell\Big(\dot b-\frac{\dot c}{a_\ast}\Big)=0,
\]
hence $\dot b=\dot c/a_\ast$.  Substituting into the $\dot q(\ell)=1$ relation yields $\dot c=-a_\ast$ and therefore
$\dot b=-1$.

Finally, integrating $\dot k_s=\dot p$ and $\dot\theta_s=\dot k$ gives
\[
\dot\theta(s)=\dot a\,s+\frac{\sin(a_\ast s)}{a_\ast^3}-\frac{s}{a_\ast^2},
\]
so $\dot\theta(\ell)=0$ implies $\dot a=1/a_\ast^2$.  Since $\dot c=a_\ast\dot b+\dot\lambda$ and we have
$\dot c=-a_\ast$, $\dot b=-1$, it follows that $\dot\lambda=0$.

\medskip
\noindent\textbf{Step 2: second derivative of $\lambda$ at $0$.}
Differentiate the system a second time at $\varepsilon=0$.
Because $c(0)=0$, differentiating $q_s=c\cos\theta$ gives
\[
\ddot q_s=\ddot c\,\cos\theta_\ast -2\dot c\,\sin\theta_\ast\,\dot\theta.
\]
Here $\sin\theta_\ast(s)=\sin(\frac{\pi}{2}+a_\ast s)=\cos(a_\ast s)$, and from Step~1 we have $\dot c=-a_\ast$ and
$\dot\theta(s)=\sin(a_\ast s)/a_\ast^3$.  Therefore
\[
\ddot q_s=-\ddot c\,\sin(a_\ast s)+\frac{1}{a_\ast^2}\sin(2a_\ast s).
\]
The initial data are $\ddot q(0)=\ddot b$ and $\ddot p(0)=0$, and differentiating the endpoint conditions again gives
\[
\ddot q(\ell)=0,\qquad \ddot p(\ell)=0.
\]
Integrating the displayed $\ddot q_s$ equation and using
$\sin(a_\ast\ell)=0$, $\cos(a_\ast\ell)=-1$, $\sin(2a_\ast\ell)=0$, $\cos(2a_\ast\ell)=1$ yields
\[
\ddot q(\ell)=\ddot b-\frac{2\ddot c}{a_\ast}=0,
\qquad
\ddot p(\ell)=\ell\Big(\ddot b-\frac{\ddot c}{a_\ast}+\frac{1}{2a_\ast^3}\Big)=0.
\]
Solving these two equations gives
\[
\ddot c=-\frac{1}{2a_\ast^2},\qquad \ddot b=-\frac{1}{a_\ast^3}.
\]

Finally, since $c(\varepsilon)=a_N(\varepsilon)b_N(\varepsilon)+\lambda_N(\varepsilon)$ and $b_N(0)=0$, we have
\[
\ddot c
=\big(ab+\lambda\big)^{\!\ddot{}}\!(0)
= a_\ast\,\ddot b+\ddot\lambda +2\dot a\,\dot b.
\]
Using $\dot a=1/a_\ast^2$ and $\dot b=-1$ from Step~1, together with the values of $\ddot c$ and $\ddot b$ just found,
we obtain
\[
\ddot\lambda
=\ddot c-a_\ast\ddot b-2\dot a\,\dot b
=
-\frac{1}{2a_\ast^2}+\frac{1}{a_\ast^2}+\frac{2}{a_\ast^2}
=\frac{5}{2a_\ast^2}.
\]
Since $\lambda_N(0)=0$ and $\dot\lambda_N(0)=0$, Taylor's theorem yields
\[
\lambda_N(\varepsilon)=\frac12\,\ddot\lambda\,\varepsilon^2+o(\varepsilon^2)
=\frac{5}{4a_\ast^2}\,\varepsilon^2+o(\varepsilon^2),
\]
which is \eqref{eq:lambda-exp-cwN}.  The coefficient $\frac{5}{4a_\ast^2}$ is positive, hence
$\lambda_N(\varepsilon)>0$ for all sufficiently small $\varepsilon\neq 0$.
\end{proof}
\subsection{Repeated reflection and the circular wave extension}\label{subsec:cw-reflection}

Fix $N\in\N_0$ and $\varepsilon$ with $0<|\varepsilon|<\varepsilon_{0,N}$.
Let $(k,p,q,\theta)$ be the corresponding solution on $[0,\ell]$ from Theorem~\ref{thm:IFT-cwN}, and set
\[
c_N(\varepsilon):=a_N(\varepsilon)b_N(\varepsilon)+\lambda_N(\varepsilon).
\]

\paragraph{Seam data.}
By construction,
\[
\theta(0)=\frac{\pi}{2},\qquad \theta(\ell)=\frac{\pi}{2}+(2N+1)\pi,\qquad p(0)=0,\qquad p(\ell)=0.
\]
As noted in \eqref{eq:ksss-vanish-cwN}, the ODE gives $q_s(0)=q_s(\ell)=0$ automatically because $\cos\theta=0$ at both
endpoints.  This is the curvature-jet compatibility required to reflect smoothly.

\begin{lemma}[Reflective extension on one period]\label{lem:period-extension-cwN}
Define $\theta$ on $[0,2\ell]$ by
\begin{equation}\label{eq:theta-reflect-period-cwN}
\theta^{(2)}(s):=
\begin{cases}
\theta(s), & 0\le s\le \ell,\\[2pt]
\big(\pi+2(2N+1)\pi\big)-\theta(2\ell-s), & \ell\le s\le 2\ell.
\end{cases}
\end{equation}
Then $\theta^{(2)}\in C^\infty([0,2\ell])$ and satisfies the same ODE
\begin{equation}\label{eq:theta-ODE-cwN}
\big(\theta^{(2)}\big)_{ssss}=c_N(\varepsilon)\cos\big(\theta^{(2)}\big)
\qquad\text{on }(0,2\ell).
\end{equation}
Moreover,
\begin{equation}\label{eq:theta-2ell-cwN}
\theta^{(2)}(2\ell)=\theta^{(2)}(0)+2(2N+1)\pi.
\end{equation}
\end{lemma}

\begin{proof}
On each open half-interval $(0,\ell)$ and $(\ell,2\ell)$, $\theta^{(2)}$ is a smooth composition, hence smooth there.
For $s\in(\ell,2\ell)$ write $u=2\ell-s$.  Then
\[
\big(\theta^{(2)}\big)_{ssss}(s)=-\theta_{ssss}(u)
=-c_N(\varepsilon)\cos(\theta(u)).
\]
Since $\cos(\alpha+2m\pi)=\cos\alpha$ and $\cos(\pi-\alpha)=-\cos\alpha$, the choice of constant
\[
C_N:=\pi+2(2N+1)\pi
\]
ensures $\cos(C_N-\alpha)=-\cos\alpha$ (because $C_N\equiv \pi\ (\mathrm{mod}\ 2\pi)$), hence
\[
-c_N(\varepsilon)\cos(\theta(u))
=
c_N(\varepsilon)\cos(C_N-\theta(u))
=
c_N(\varepsilon)\cos(\theta^{(2)}(s)),
\]
which gives \eqref{eq:theta-ODE-cwN} on $(\ell,2\ell)$.  On $(0,\ell)$ it holds by assumption.

It remains to check smooth matching at $s=\ell$.
Let
\[
\theta(\ell)=\frac{\pi}{2}+(2N+1)\pi = \frac{3\pi}{2}+2N\pi,
\]
so $\theta(\ell)\in\pi/2+\pi\Z$, and $\theta_{ss}(\ell)=p(\ell)=0$.  Moreover,
$C_N=2\theta(\ell)$, so the second line of \eqref{eq:theta-reflect-period-cwN} is exactly the one-sided reflected
extension across $s=\ell$.  Applying Lemma~\ref{lem:reflection-principle} in the left-sided form gives
$\theta^{(2)}\in C^\infty([0,2\ell])$ and the ODE holds through the seam.

Finally,
\[
\theta^{(2)}(2\ell)=C_N-\theta(0)=\big(\pi+2(2N+1)\pi\big)-\frac{\pi}{2}
=\frac{\pi}{2}+2(2N+1)\pi,
\]
which is \eqref{eq:theta-2ell-cwN}.
\end{proof}

\begin{definition}[Circular wave extension]\label{def:circular-wave-N}
Let $\theta^{(2)}$ be as in Lemma~\ref{lem:period-extension-cwN}.
Extend $\theta^{(2)}$ to a smooth function $\Theta:\R\to\R$ by
\begin{equation}\label{eq:theta-global-cwN}
\Theta(s+2m\ell):=\theta^{(2)}(s)+2m(2N+1)\pi,
\qquad s\in[0,2\ell],\ m\in\Z.
\end{equation}
Define $\Gamma_{N,\varepsilon}:\R\to\R^2$ by integrating the tangent:
\begin{equation}\label{eq:Gamma-cwN}
\Gamma_{N,\varepsilon}(s):=\int_0^s (\cos\Theta(\sigma),\sin\Theta(\sigma))\,d\sigma.
\end{equation}
We call $\Gamma_{N,\varepsilon}$ the \emph{circular wave of index $N$} generated by the fundamental arc.
\end{definition}

\begin{proposition}[Periodicity up to translation and stationarity]\label{prop:cw-periodic-stationary-N}
For each $N\in\N_0$ and each sufficiently small $\varepsilon\neq 0$, the circular wave $\Gamma_{N,\varepsilon}$ is a
$C^\infty$ unit-speed immersion on $\R$ and satisfies the stationary equation for $J_{\lambda_N(\varepsilon)}$
(equivalently, its curvature jet satisfies \eqref{eq:shooting-cwN} with $c=c_N(\varepsilon)$).
Moreover, the tangent is $2\ell$-periodic and the curve is periodic up to a translation:
\begin{equation}\label{eq:Gamma-translation-cwN}
\Gamma_{N,\varepsilon}(s+2\ell)=\Gamma_{N,\varepsilon}(s)+P_{N,\varepsilon},
\qquad
P_{N,\varepsilon}:=\Gamma_{N,\varepsilon}(2\ell)-\Gamma_{N,\varepsilon}(0)
=\int_0^{2\ell}(\cos\Theta,\sin\Theta)\,ds.
\end{equation}
\end{proposition}

\begin{proof}
By construction, $\Theta\in C^\infty(\R)$ and $\Theta_s=k$ is $2\ell$-periodic, hence $(\cos\Theta,\sin\Theta)$ is
$2\ell$-periodic. Therefore $\Gamma_{N,\varepsilon}$ defined by \eqref{eq:Gamma-cwN} is $C^\infty$ and unit-speed.
Equation \eqref{eq:theta-ODE-cwN} holds on $[0,2\ell]$, and the extension rule
\eqref{eq:theta-global-cwN} shifts $\Theta$ by an integer multiple of $2\pi$, so $\cos\Theta$ is unchanged and the ODE
continues to hold on all of $\R$.  Thus $(k,p,q,\theta)$ satisfy \eqref{eq:shooting-cwN} globally, i.e.\
$\Gamma_{N,\varepsilon}$ is stationary for $J_{\lambda_N(\varepsilon)}$.

Finally, periodicity of the tangent implies \eqref{eq:Gamma-translation-cwN} by integration.
\end{proof}

\subsection{Non-closure near the circle}\label{subsec:cw-nonclosure}

As in the $N=0$ case, the wave closes (i.e.\ becomes a closed curve) if and only if $P_{N,\varepsilon}=0$.
For $\varepsilon=0$ the base solution is a (multiply covered) circle and $P_{N,0}=0$.
For $\varepsilon\neq 0$ small, the period vector is nonzero, so the wave is not closed.
The same first-order drift mechanism applies; only $a_{N,\ast}$ changes.

\begin{lemma}[First-order drift of the midpoint height]\label{lem:y-drift-cwN}
Let $(a_N(\varepsilon),b_N(\varepsilon),\lambda_N(\varepsilon))$ be the IFT branch of Theorem~\ref{thm:IFT-cwN}, and let
$(x(\cdot;\varepsilon),y(\cdot;\varepsilon))$ be the solution of \eqref{eq:xy-cwN} driven by the corresponding angle
$\theta(\cdot;\varepsilon)$ on $[0,\ell]$.
Then
\begin{equation}\label{eq:yell-linear-cwN}
y(\ell;\varepsilon)
=
-\frac{\ell}{2a_{N,\ast}^3}\,\varepsilon+O(\varepsilon^2)
=
-\frac{\ell^4}{2(2N+1)^3\pi^3}\,\varepsilon+O(\varepsilon^2)
\qquad\text{as }\varepsilon\to 0.
\end{equation}
In particular, $y(\ell;\varepsilon)\neq 0$ for all sufficiently small $\varepsilon\neq 0$.
\end{lemma}

\begin{proof}
At $\varepsilon=0$ we have $\theta_\ast(s)=\pi/2+a_{N,\ast}s$, hence
\[
y_\ast(\ell)=\int_0^\ell \sin\theta_\ast\,ds=\int_0^\ell \cos(a_{N,\ast}s)\,ds=\frac{\sin(a_{N,\ast}\ell)}{a_{N,\ast}}=0,
\]
since $a_{N,\ast}\ell=(2N+1)\pi$.
Differentiate $y(\ell;\varepsilon)=\int_0^\ell \sin(\theta(s;\varepsilon))\,ds$ along the IFT branch:
\[
\frac{d}{d\varepsilon}\Big|_{\varepsilon=0} y(\ell;\varepsilon)
=\int_0^\ell \cos(\theta_\ast(s))\,\dot\theta(s)\,ds,
\qquad \dot\theta:=\partial_\varepsilon\theta|_{\varepsilon=0}.
\]
From the first-variation computation in Step~1 of Lemma~\ref{lem:lambda-quadratic-cwN} one has
$\dot\theta(s)=\sin(a_{N,\ast}s)/a_{N,\ast}^3$ and $\cos(\theta_\ast(s))=-\sin(a_{N,\ast}s)$, hence
\[
\frac{d}{d\varepsilon}\Big|_{\varepsilon=0} y(\ell;\varepsilon)
=
-\frac{1}{a_{N,\ast}^3}\int_0^\ell \sin^2(a_{N,\ast}s)\,ds.
\]
Using $a_{N,\ast}\ell=(2N+1)\pi$ gives
\[
\int_0^\ell \sin^2(a_{N,\ast}s)\,ds
=\frac{1}{a_{N,\ast}}\int_0^{(2N+1)\pi}\sin^2u\,du
=\frac{1}{a_{N,\ast}}\cdot\frac{(2N+1)\pi}{2}
=\frac{\ell}{2},
\]
and therefore
\[
\frac{d}{d\varepsilon}\Big|_{\varepsilon=0} y(\ell;\varepsilon)
=
-\frac{\ell}{2a_{N,\ast}^3}.
\]
Taylor expansion yields \eqref{eq:yell-linear-cwN}.
\end{proof}

\begin{corollary}[Circular waves are not closed near the circle]\label{cor:not-closed-cwN}
For each $N\in\N_0$ and all sufficiently small $\varepsilon\neq 0$, the circular wave $\Gamma_{N,\varepsilon}$ is not
closed.
\end{corollary}

\begin{proof}
On $[0,2\ell]$, the reflection rule \eqref{eq:theta-reflect-period-cwN} implies $\sin\Theta$ is even about $\ell$, hence
\[
P_{N,\varepsilon}
=
\int_0^{2\ell}(\cos\Theta,\sin\Theta)\,ds
=
\Big(\int_0^{2\ell}\cos\Theta\,ds,\ 2\int_0^\ell \sin\Theta\,ds\Big)
=
\big(P_x,\ 2y(\ell;\varepsilon)\big).
\]
If $y(\ell;\varepsilon)\neq 0$ then $P_{N,\varepsilon}\neq 0$ and $\Gamma_{N,\varepsilon}$ cannot be closed.
By Lemma~\ref{lem:y-drift-cwN}, $y(\ell;\varepsilon)\neq 0$ for all sufficiently small $\varepsilon\neq 0$.
\end{proof}

\subsection{An intrinsic turning-rate invariant and geometric distinctness}\label{subsec:cw-distinctness}

The interval $[0,2\ell]$ is a convenient shooting cell, but it need not be a primitive tangent period.  We therefore
use an invariant that is independent of which multiple of a period is chosen.

\begin{definition}[Mean turning rate]\label{def:mean-turning-rate-cw}
Let $\Gamma:\R\to\R^2$ be a unit-speed immersion with nonconstant periodic tangent, and let $\Theta$ be a continuous
lifted tangent angle.  Define
\begin{equation}\label{eq:mean-turning-rate-cw}
\Omega(\Gamma):=\lim_{R\to\infty}\frac{|\Theta(R)-\Theta(0)|}{R}.
\end{equation}
\end{definition}

\begin{lemma}[Well-definedness, congruence invariance, and scaling]\label{lem:mean-turning-rate-invariant}
If $P>0$ is any tangent period and $\Theta(s+P)-\Theta(s)=2\pi d$ with $d\in\Z$, then
\[
\Omega(\Gamma)=\frac{2\pi|d|}{P}.
\]
Consequently $\Omega$ is independent of the lift, base point, and choice of tangent period, and is invariant under
Euclidean congruences and unit-speed reparametrisations $s\mapsto\pm s+s_0$.  For the unit-speed dilation
\[
\Gamma^\rho(s):=\rho\,\Gamma(s/\rho),\qquad \rho>0,
\]
one has $\Omega(\Gamma^\rho)=\rho^{-1}\Omega(\Gamma)$.
\end{lemma}

\begin{proof}
The difference $s\mapsto\Theta(s+P)-\Theta(s)$ is continuous and takes values in $2\pi\Z$, so it is the constant
$2\pi d$.  Writing $R=nP+r$ with $0\le r<P$ gives
\[
\Theta(R)-\Theta(0)=2\pi nd+\Theta(r)-\Theta(0).
\]
The last term is bounded, and division by $R$ proves the formula for $\Omega$.  This formula also proves independence
of the chosen period.  Changing the lift or base point changes only bounded terms.  Euclidean isometries and
unit-speed reversal change a lifted angle by a constant and possibly a sign, so the absolute asymptotic rate is
unchanged.  Finally, a lifted angle for $\Gamma^\rho$ is $\Theta(s/\rho)$ up to a constant, which gives the scaling law.
\end{proof}

\begin{proposition}[Distinctness for different $N$]\label{prop:cw-distinctness-N}
For each $N\in\N_0$ and sufficiently small $\varepsilon\neq0$,
\[
\Omega(\Gamma_{N,\varepsilon})=\frac{(2N+1)\pi}{\ell}.
\]
In particular, if $N\neq M$, then $\Gamma_{N,\varepsilon}$ and $\Gamma_{M,\varepsilon'}$ are not congruent for any
sufficiently small nonzero $\varepsilon,\varepsilon'$.
\end{proposition}

\begin{proof}
The tangent has period $2\ell$, and \eqref{eq:theta-global-cwN} gives
$\Theta(s+2\ell)-\Theta(s)=2(2N+1)\pi$.  Lemma~\ref{lem:mean-turning-rate-invariant} therefore yields the displayed
formula.  Distinct indices give distinct congruence invariants.
\end{proof}

\subsection{Scaling to prescribe $\lambda>0$}\label{subsec:cw-scaling}

The IFT branch yields $\lambda_N(\varepsilon)>0$ for $\varepsilon\neq 0$ small (Lemma~\ref{lem:lambda-quadratic-cwN}).
Using the dilation invariance of the stationary equation, we may normalise the parameter.

\begin{remark}[Scaling to fix $\lambda$]\label{rem:scale-cwN}
If $\Gamma$ is stationary for parameter $\lambda>0$ and we use the unit-speed dilation
$\widetilde\Gamma(s)=\rho\Gamma(s/\rho)$, then the stationary
parameter rescales as $\widetilde\lambda=\lambda/\rho^4$.
Thus, for any prescribed $\lambda_0>0$ and any sufficiently small $\varepsilon\neq 0$, choosing
$\rho=(\lambda_N(\varepsilon)/\lambda_0)^{1/4}$ produces a geometrically similar circular wave that is stationary for
$J_{\lambda_0}$.
As $\varepsilon\to 0$, one has $\lambda_N(\varepsilon)\to 0$ and the waves converge (after appropriate scaling) to the
corresponding $(2N{+}1)$-half-turn circular base solution.
\end{remark}

\begin{corollary}[Infinitely many circular waves at every prescribed parameter]\label{cor:cw-fixed-lambda-infinite}
For every $\lambda_\ast>0$, there are infinitely many pairwise non-congruent, nonclosed circular waves stationary for
$J_{\lambda_\ast}$.
\end{corollary}

\begin{proof}
For each $N$, choose $0<|\varepsilon_N|<\varepsilon_{0,N}$ and set
\[
\rho_N:=\left(\frac{\lambda_N(\varepsilon_N)}{\lambda_\ast}\right)^{1/4},
\qquad
\widehat\Gamma_N(s):=\rho_N\Gamma_{N,\varepsilon_N}(s/\rho_N).
\]
The scaled wave is stationary for $J_{\lambda_\ast}$ and remains nonclosed.  By
Lemma~\ref{lem:mean-turning-rate-invariant},
\[
\Omega(\widehat\Gamma_N)
=\frac{(2N+1)\pi}{\ell}
\left(\frac{\lambda_\ast}{\lambda_N(\varepsilon_N)}\right)^{1/4}.
\]
For each fixed $N$, this quantity tends to $+\infty$ as $\varepsilon_N\to0$, because
$\lambda_N(\varepsilon_N)\to0$.  We may therefore choose the parameters recursively so that the displayed mean
turning rates are strictly increasing.  The resulting waves are pairwise non-congruent.
\end{proof}

Let us now prove Theorem~\ref{thm:circular-waves-infinite}.

\begin{proof}
Items \textup{(i)} and the existence of the IFT branch follow from Theorem~\ref{thm:IFT-cwN} and
Lemma~\ref{lem:lambda-quadratic-cwN}.
The reflective extension and stationarity, as well as periodicity up to translation, follow from
Lemmas~\ref{lem:period-extension-cwN} and Proposition~\ref{prop:cw-periodic-stationary-N}.
Non-closure ($P_{N,\varepsilon}\neq 0$) follows from Corollary~\ref{cor:not-closed-cwN}.
Geometric distinctness across $N$ is Proposition~\ref{prop:cw-distinctness-N}, based on the intrinsic mean turning rate.
The fixed-$\lambda_\ast$ conclusion is Corollary~\ref{cor:cw-fixed-lambda-infinite}.
\end{proof}

\appendix

\section{Length-penalised ideal curves and super-lemniscates}
\label{app:lpif}

Here we give the detail alluded to earlier in Remark \ref{rem:lpif}.
Define
\[
\mathcal E[k]:= k_{ssss} + k^2k_{ss} - \frac12kk_s^2.
\]
The length-penalised ideal critical curves satisfy $\mathcal E[k] = -\lambda k$, whereas the curves identified in \cite[Theorem 1.1]{MW26} satisfy 
$k_{ss} + ck^3 = 0$, for some $c\in\R$.

\begin{lemma}\label{lem:overlap_with_E_equation}
Let \(I\subset\R\) be an interval, and \(k:I\to\R\) be a nontrivial solution of
\begin{equation}\label{eq:stationary_cubic_curvature_appendix}
k_{ss}+c\,k^3=0.
\end{equation}
Then
\begin{equation}\label{eq:E_formula_appendix}
\mathcal E[k]
=
-\Big(6c+\frac12\Big)A\,k+\frac{3c(8c-1)}{4}k^5,
\qquad
A:=k_s^2+\frac c2\,k^4.
\end{equation}
In particular, \(\mathcal E[k]=-\lambda k\) for some constant \(\lambda\in\R\) if and only if
either
\[
c=0
\qquad\text{or}\qquad
c=\frac18.
\]

More precisely:
\begin{enumerate}
\item if \(c=0\), then \(k(s)=as+b\) and
\[
\mathcal E[k]=-\frac{a^2}{2}\,k,
\]
so \(\lambda=\frac{a^2}{2}\ge0\);

\item if \(c=\frac18\), then
\[
\mathcal E[k]=-\frac54 A\,k,
\]
so \(\lambda=\frac54 A\). Moreover, for every nontrivial real solution one has \(A>0\), and hence
\[
\lambda>0.
\]
\end{enumerate}

For the explicit family
\[
k(s)=\frac{\alpha}{\sqrt c}\text{cn}(\alpha s+\beta;\tfrac12),
\]
one has
\[
A=\frac{\alpha^4}{2c},
\]
and therefore, when \(c=\frac18\),
\[
\lambda=5\alpha^4.
\]
\end{lemma}

\begin{proof}
Multiplying \eqref{eq:stationary_cubic_curvature_appendix} by \(2k_s\) and integrating gives the
first integral
\[
k_s^2+\frac c2\,k^4=A
\]
for some constant \(A\in\R\).

Next, differentiating \eqref{eq:stationary_cubic_curvature_appendix} twice yields
\[
k_{sss}=-3c\,k^2k_s,
\]
and hence
\[
k_{ssss}=-6c\,k\,k_s^2-3c\,k^2k_{ss}.
\]
Using \(k_{ss}=-c\,k^3\), we find
\[
k_{ssss}=-6c\,k\,k_s^2+3c^2k^5.
\]
Therefore
\begin{align*}
\mathcal E[k]
&=
\big(-6c\,k\,k_s^2+3c^2k^5\big)+k^2(-c\,k^3)-\frac12 k\,k_s^2
\\
&=
-\Big(6c+\frac12\Big)k\,k_s^2+(3c^2-c)k^5.
\end{align*}
Substituting the first integral \(k_s^2=A-\frac c2 k^4\) gives
\begin{align*}
\mathcal E[k]
&=
-\Big(6c+\frac12\Big)k\Big(A-\frac c2 k^4\Big)+(3c^2-c)k^5
\\
&=
-\Big(6c+\frac12\Big)A\,k+\frac{3c(8c-1)}{4}k^5,
\end{align*}
which proves \eqref{eq:E_formula_appendix}.

Now suppose \(\mathcal E[k]=-\lambda k\) for some constant \(\lambda\). If \(c=0\), then
\(k_{ss}=0\), so \(k(s)=as+b\), and \eqref{eq:E_formula_appendix} reduces to
\[
\mathcal E[k]=-\frac{a^2}{2}\,k.
\]
This gives the first case.

Assume next that \(c\neq0\). Then \(k\) cannot be constant unless \(k\equiv0\), which is excluded.
Also, the zeros of a nontrivial solution are isolated, so there is an open subinterval on which
\(k\neq0\) and \(k\) is nonconstant. On that interval, dividing
\eqref{eq:E_formula_appendix} by \(k\) gives
\[
-\lambda
=
-\Big(6c+\frac12\Big)A+\frac{3c(8c-1)}{4}k^4.
\]
Since the left-hand side is constant and \(k\) is nonconstant, the coefficient of \(k^4\) must
vanish. Hence
\[
\frac{3c(8c-1)}{4}=0.
\]
Because \(c\neq0\), this implies \(c=\frac18\).

When \(c=\frac18\), \eqref{eq:E_formula_appendix} becomes
\[
\mathcal E[k]=-\frac54 A\,k.
\]
If \(A=0\), then
\[
k_s^2+\frac1{16}k^4=0,
\]
which forces \(k\equiv0\), contradicting the nontriviality assumption. Thus \(A>0\), and so
\(\lambda=\frac54 A>0\).

Finally, for
\[
k(s)=\frac{\alpha}{\sqrt c}\text{cn}(\alpha s+\beta;\tfrac12),
\]
we may evaluate the first integral at a point where \(\text{cn}(\alpha s+\beta;\tfrac12)=1\) and
\(k_s=0\), obtaining
\[
A=\frac c2\Big(\frac{\alpha}{\sqrt c}\Big)^4=\frac{\alpha^4}{2c}.
\]
For \(c=\frac18\), this gives \(A=4\alpha^4\), and hence
\[
\lambda=\frac54 A=5\alpha^4.
\]
This completes the proof.
\end{proof}

\begin{figure}[t]
    \centering
    \includegraphics[width=0.32\textwidth]{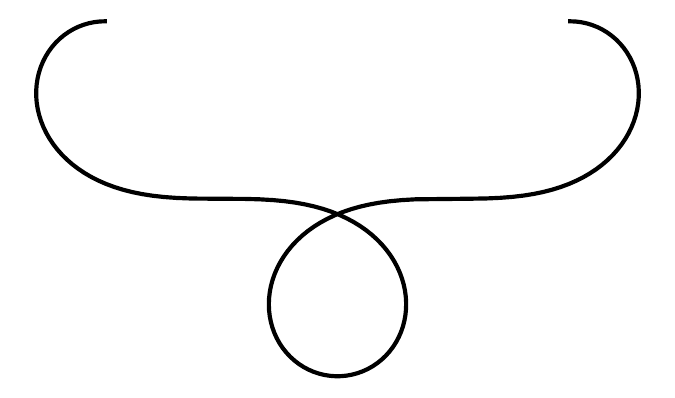}
    \caption{The open curve obtained by integrating the tangent field
    \(\gamma_s=(\cos\theta,\sin\theta)\) associated with the explicit solution of
    \(k_{ss}+\frac18 k^3=0\), namely
    \(
        k(s)=2\sqrt2\,\text{cn}\!\left(s;\tfrac12\right), 
        \theta(s)=4\arcsin\!\left(\frac{1}{\sqrt2}\text{sn}\!\left(s;\tfrac12\right)\right),
    \)
    over one \(4K(\tfrac12)\)-interval.}
    \label{fig:awesome_curve}
\end{figure}

None of the super-lemniscates identified in \cite[Theorem 1.1]{MW26} satisfy \(\mathcal E[k]=-\lambda k\) with \(\lambda>0\).
This is because, due to closure, they exist only for
\(c=0\) or \(c=c_j=\frac{2}{(4j-1)^2}\), \(j\ge1\).
By Lemma \ref{lem:overlap_with_E_equation}, if the same curvature also satisfies
\(\mathcal E[k]=-\lambda k\), then necessarily \(c=0\) or \(c=\frac18\).
Since
\(
\frac18\neq \frac{2}{(4j-1)^2}
\)
for every $j\in\N$, the case \(c=\frac18\) is impossible for closed stationary curves.

Therefore \(c=0\). 
However then \(\gamma\) is then an \(\omega\)-circle, so
\(k\) is constant. Hence \(
\mathcal E[k]\equiv0,
\)
implying \(-\lambda k=0\). Since \(k\not\equiv0\) for a closed immersed circle, we conclude \(\lambda=0\), which is again impossible.

\begin{remark}
Lemma \ref{lem:overlap_with_E_equation} shows that there is nevertheless a genuine overlap in the
\emph{open-curve} setting: every nontrivial local solution of
\[
k_{ss}+\frac18 k^3=0
\]
also satisfies
\[
\mathcal E[k]=-\lambda k
\]
for some \(\lambda>0\). In particular, the explicit Jacobi-elliptic family
\[
k(s)=2\sqrt2\,\alpha\,\text{cn}(\alpha s+\beta;\tfrac12)
\]
solves \(\mathcal E[k]=-5\alpha^4 k\).
We present a picture in Figure \ref{fig:awesome_curve}.
\end{remark}

\section{A principle of symmetric criticality}
\label{sec:symcrit}

For comparison, we record the following version of Palais' principle of symmetric criticality \cite{P79}.

\begin{lemma}[Finite-group averaging / symmetric criticality for finite groups]\label{lem:finite-group-averaging}
Let $J$ be a $C^{1}$ geometric functional on smooth closed immersions
$\Gamma:\R/L\Z\to\R^{2}$ which is invariant under
\emph{(i)} rigid motions of $\R^{2}$ and \emph{(ii)} reparametrisations of $\R/L\Z$
(in particular, under arclength-preserving translations and reversals of the parameter).
Let $G$ be a finite group of ``symmetries'' of a given immersion $\Gamma$, in the following sense:

\smallskip
\noindent For each $g\in G$ we are given
\begin{itemize}[leftmargin=2.0em,itemsep=2pt,topsep=2pt]
\item a rigid motion $\mathsf{A}_{g}:\R^{2}\to\R^{2}$, and
\item an arclength isometry $\sigma_{g}:\R/L\Z\to\R/L\Z$ (translation or reversal),
\end{itemize}
such that
\begin{equation}\label{eq:G-invariance-curve}
\Gamma \;=\; \mathsf{A}_{g}\circ \Gamma \circ \sigma_{g}
\qquad\text{for every }g\in G.
\end{equation}
(Equivalently, $\Gamma$ is a fixed point of the induced action $g\cdot\Gamma:=\mathsf{A}_{g}\circ\Gamma\circ\sigma_{g}$.)

Define the induced action of $G$ on smooth variation vector fields $V$ along $\Gamma$ by
\begin{equation}\label{eq:G-action-variation}
(g\cdot V)(s):=D\mathsf{A}_{g}\,V(\sigma_{g}(s)),
\qquad s\in\R/L\Z,
\end{equation}
where $D\mathsf{A}_{g}$ is the constant orthogonal matrix associated to the rigid motion $\mathsf{A}_{g}$
(translations do not affect $D\mathsf{A}_{g}$).

Assume that $\Gamma$ is stationary for $J$ with respect to all \emph{$G$-invariant} variations, i.e.
\begin{equation}\label{eq:stationary-on-fixed}
\delta J[\Gamma](W)=0
\qquad\text{for every smooth field }W\text{ along }\Gamma\text{ satisfying }g\cdot W=W\ \forall g\in G.
\end{equation}
Then $\Gamma$ is stationary for $J$ with respect to \emph{arbitrary} smooth variations:
\[
\delta J[\Gamma](V)=0\qquad\text{for every smooth field }V\text{ along }\Gamma.
\]
\end{lemma}

\begin{proof}
\textbf{Step 1: invariance of the first variation under the group action.}
Fix $g\in G$ and a smooth variation field $V$ along $\Gamma$.
Let $\Gamma_{\varepsilon}$ be any smooth variation with $\Gamma_{0}=\Gamma$ and
$\partial_{\varepsilon}\Gamma_{\varepsilon}|_{\varepsilon=0}=V$.
Consider the $g$-transformed family
\[
\widetilde\Gamma_{\varepsilon}
:=\mathsf{A}_{g}\circ \Gamma_{\varepsilon}\circ \sigma_{g}.
\]
By invariance of $J$ under rigid motions and reparametrisations,
\[
J[\widetilde\Gamma_{\varepsilon}] = J[\Gamma_{\varepsilon}]
\qquad\text{for all sufficiently small }\varepsilon.
\]
Differentiating both sides of the above and evaluating at $\varepsilon=0$ yields
\begin{equation*}
\delta J[\widetilde\Gamma_{0}]\!\left(\partial_{\varepsilon}\widetilde\Gamma_{\varepsilon}\big|_{\varepsilon=0}\right)
=
\delta J[\Gamma]\!\left(\partial_{\varepsilon}\Gamma_{\varepsilon}\big|_{\varepsilon=0}\right)
\end{equation*}
and then by definition of $\partial_\varepsilon\Gamma_\varepsilon|_{\varepsilon=0}$ we find
\begin{equation}
\label{eq:firstvar-invariance}
\delta J[\widetilde\Gamma_{0}]\!\left(\partial_{\varepsilon}\widetilde\Gamma_{\varepsilon}\big|_{\varepsilon=0}\right)
= \delta J[\Gamma](V).
\end{equation}
Using \eqref{eq:G-invariance-curve}, we have $\widetilde\Gamma_{0}=\Gamma$.
Moreover, differentiating $\widetilde\Gamma_{\varepsilon}=\mathsf{A}_{g}\circ\Gamma_{\varepsilon}\circ\sigma_{g}$
at $\varepsilon=0$ gives
\[
\partial_{\varepsilon}\widetilde\Gamma_{\varepsilon}\big|_{\varepsilon=0}(s)
=
D\mathsf{A}_{g}\,\partial_{\varepsilon}\Gamma_{\varepsilon}\big|_{\varepsilon=0}(\sigma_{g}(s))
=
D\mathsf{A}_{g}\,V(\sigma_{g}(s))
=
(g\cdot V)(s),
\]
i.e. $\partial_{\varepsilon}\widetilde\Gamma_{\varepsilon}|_{0}=g\cdot V$ in the sense of \eqref{eq:G-action-variation}.
Substituting into \eqref{eq:firstvar-invariance} yields the key identity
\begin{equation}\label{eq:firstvar-invariant-under-G}
\delta J[\Gamma](g\cdot V)=\delta J[\Gamma](V)
\qquad\text{for all }g\in G\text{ and all }V.
\end{equation}

\medskip
\textbf{Step 2: average an arbitrary variation to a symmetric one without changing the first variation.}
Given an arbitrary smooth field $V$ along $\Gamma$, define its $G$-average by
\begin{equation}\label{eq:G-average}
\overline V
:=\frac{1}{|G|}\sum_{g\in G} g\cdot V.
\end{equation}
By construction and the group property, $\overline V$ is $G$-invariant:
for any $h\in G$,
\[
h\cdot\overline V
=\frac{1}{|G|}\sum_{g\in G} h\cdot(g\cdot V)
=\frac{1}{|G|}\sum_{g\in G} (hg)\cdot V
=\frac{1}{|G|}\sum_{g\in G} g\cdot V
=\overline V.
\]
Since the first variation is linear in the variation field and \eqref{eq:firstvar-invariant-under-G} holds, we have
\[
\delta J[\Gamma](\overline V)
=
\frac{1}{|G|}\sum_{g\in G}\delta J[\Gamma](g\cdot V)
=
\frac{1}{|G|}\sum_{g\in G}\delta J[\Gamma](V)
=
\delta J[\Gamma](V).
\]

\medskip
\textbf{Step 3: conclude.}
Because $\overline V$ is $G$-invariant, assumption \eqref{eq:stationary-on-fixed} gives
$\delta J[\Gamma](\overline V)=0$, hence $\delta J[\Gamma](V)=0$ for arbitrary $V$.
\end{proof}

\begin{remark}[The symmetry group for $\Gamma_n$]\label{rem:symmetry-group-Gn}
Parameterise $\Gamma_n$ by arclength on $\R/(4\ell_n)\Z$ and set $B_n:=\Gamma_n(\ell_n)\in\{y=0\}$.
 The construction gives two involutive symmetries:
 \begin{itemize} \item point reflection about $B_n$ with parameter reversal,
 \[
 \mathsf{A}_{P}(x):=2B_n-x,\qquad \sigma_{P}(s):=2\ell_n-s \ \ (\mathrm{mod}\ 4\ell_n),
 \]
 \item reflection across the $x$-axis with parameter reversal,
 \[
 \mathsf{A}_{R}(x_1,x_2):=(x_1,-x_2),\qquad \sigma_{R}(s):=4\ell_n-s \ \ (\mathrm{mod}\ 4\ell_n).
 \]
 \end{itemize}
 These generate a finite group $G_n$ (of order at most $4$) satisfying \eqref{eq:G-invariance-curve}.
 \end{remark}

\section{Numerical generation of figures}\label{sec:numerical-visualisation}

This appendix records the numerical procedure used to generate the figures in the paper. 
The figures are purely illustrative and not used in any proof.

\subsection*{C.1. Selected lemniscate and butterfly galleries}

The numerical curves in Figures~\ref{fig:selected-first-pair},
\ref{fig:selected-lemniscate-type-candidates}, and~\ref{fig:selected-butterfly-type-candidates}, and the diagnostics in
Table~\ref{tab:selected-critical-point-data}, were generated by the same doubled-angle direct minimisation used in the
proof, together with the scalar selection criterion supplied by the midpoint balance defect.  The computations are
illustrative only; the existence proof uses the compactness, window-selection, and transversality arguments in the main
text.

\smallskip
\noindent\textbf{Fixed-turning minimisation.}
For a prescribed lifted quarter-arc turning value \(\Phi\), and fixed \(\lambda>0\), we first compute a numerical
surrogate of the fixed-turning minimiser
\[
\inf_{(\ell,\theta)\in\mathcal A^{(2)}(\Phi)}J^{(2)}_\lambda(\ell,\theta),
\]
where \(\mathcal A^{(2)}(\Phi)\) and \(J^{(2)}_\lambda\) are defined in
\eqref{eq:Adouble} and~\eqref{eq:Jdouble}.  Passing to the fixed half-domain
\(t=s/\ell\in[0,1]\), we write
\[
        \phi(t)=\theta(\ell t),
        \qquad
        \phi(0)=\frac\pi2,
        \qquad
        \phi(1)=\frac\pi2+\Phi,
        \qquad
        \phi'(1)=0,
\]
and impose the fixed-domain closure constraint
\[
        \int_0^1\sin\phi(t)\,dt=0.
\]
The length variable is optimised explicitly from the fixed-domain bending value: if
\[
        A(\Phi,c):=\int_0^1(\phi''(t))^2\,dt,
\]
then the doubled energy is
\[
        J^{(2)}_\lambda(\ell,\phi)=\ell^{-3}A(\Phi,c)+2\lambda\ell,
\]
so the optimal numerical length is
\[
        \ell=\left(\frac{3A(\Phi,c)}{2\lambda}\right)^{1/4}.
\]
Thus the finite-dimensional optimisation is carried out over the Galerkin coefficients, subject only to the scalar
closure constraint; the length is updated from the virial/scale condition.

\smallskip
\noindent\textbf{Galerkin ansatz.}
We use the equality-profile cubic
\[
        H(t)=\frac32t-\frac12t^3,
\]
and write
\[
        \phi(t)=\frac\pi2+\Phi H(t)+\sum_{j=0}^{K-1}c_j\psi_j(t),
        \qquad 0\le t\le1.
\]
In the computations reported here,
\[
        \psi_j(t)=t(1-t)^2T_j(2t-1),
\]
where \(T_j\) denotes the Chebyshev polynomial of degree \(j\).  Hence
\[
        \psi_j(0)=\psi_j(1)=\psi_j'(1)=0,
\]
so the endpoint angle conditions and the midpoint seam condition \(\phi'(1)=0\) are imposed exactly for every
coefficient vector \(c\).  The basis does \emph{not} impose the third-derivative condition at the midpoint.  Indeed
\(\psi_j'''(1)\) is generally nonzero, and therefore the midpoint balance defect below is a genuine numerical diagnostic,
not a built-in artefact of the ansatz.

The half-angle is extended to the doubled interval by even reflection,
\[
        \phi(1+t)=\phi(1-t),\qquad 0\le t\le1,
\]
which gives the point-symmetric doubled arc after reconstruction.  All integrals are evaluated by high-order
Gauss--Legendre quadrature.  The displayed run used \(\lambda=1\), \(K=48\), and \(700\) quadrature nodes.

\smallskip
\noindent\textbf{The midpoint balance defect and selected turning values.}
For a fixed-\(\Phi\) minimiser the midpoint angle is constrained, so the doubled solution may carry a Dirac multiplier
at the midpoint.  Numerically we measure the associated balance defect by
\[
        Q(\Phi):=\theta_{sss}(\ell^-)=\ell^{-3}\phi'''(1^-).
\]
The balanced values are the roots of \(Q\).  Equivalently, along a smooth branch of fixed-turning minimisers, the
envelope identity gives
\[
        M_\lambda'(\Phi)=-2Q(\Phi),
\]
so the roots of \(Q\) are precisely the stationary points of the fixed-turning value function.  This is the
numerical version of the transversality mechanism used in the proof: when \(Q=0\), the midpoint Dirac mass vanishes
and the even midpoint reflection is smooth to third order, hence smooth by the ODE.

The numerical search proceeds in increasing \(\Phi\).  We first solve the fixed-turning constrained minimisation on an
ordered grid of \(\Phi/\pi\)-values, warm-starting each solve from nearby solutions.  We then record all adjacent grid
intervals on which \(Q\) changes sign.  Each sign-change interval is refined by a bracketed root search; at every trial
value of \(\Phi\), the fixed-turning minimisation problem is solved again, again using warm starts from the bracket.
A candidate is accepted when the refined value satisfies small closure residual and small \(|Q|\).  In the run reported
in the figures, the scan step was \(0.04\) in \(\Phi/\pi\), and the acceptance thresholds were
\[
        |Q|\le 5\times10^{-5},
        \qquad
        \left|\int_0^1\sin\phi(t)\,dt\right|\le 5\times10^{-7}.
\]

\smallskip
\noindent\textbf{Reconstruction, closing, and diagnostics.}
For each accepted selected value we reconstruct the doubled arc by integrating the unit tangent,
\[
        \gamma(s)=\int_0^s(\cos\theta(\sigma),\sin\theta(\sigma))\,d\sigma,
\]
and then close it by the same \(x\)-axis reflection with parameter reversal used analytically in
\eqref{eq:doubled-xaxis-close}.  The total length and energy reported in
Table~\ref{tab:selected-critical-point-data} are computed from the closed curve,
\[
        L[\Gamma]=4\ell,
        \qquad
        J_1[\Gamma]=\frac12\int_\Gamma k_s^2\,ds+L[\Gamma].
\]
We also monitor the fixed-domain closure residual, the midpoint balance defect \(Q\), the endpoint natural residual
\(\theta_{ss}(0)\), the midpoint seam residual \(\theta_s(\ell)\), and the virial residual.  These diagnostics are used
only to decide which numerical candidates are visually and quantitatively reliable enough to display.

\smallskip
\noindent\textbf{Curation and the two observed families.}
The root search produces two visually stable numerical subfamilies.  We refer to them as the
\emph{lemniscate-type} and \emph{butterfly-type} branches.  In the ordered list of accepted candidates, the first
candidate is the least-energy retained lemniscate-type curve and the second is the least-energy retained
butterfly-type curve.  The two types then alternate through candidate \(39\).  Candidate \(40\), although it occurs
where the alternating pattern would suggest a butterfly-type curve, is visually lemniscate-type.  Candidates \(41\),
\(42\), and \(43\) are classified as butterfly-type, lemniscate-type, and butterfly-type respectively.

Thus the displayed groups are
\[
        \mathcal L=\{1,3,5,\ldots,39,40,42\}
\]
for the lemniscate-type gallery, and
\[
        \mathcal B=\{2,4,6,\ldots,38,41,43\}
\]
for the butterfly-type gallery.

\subsection*{C.2. Circular waves: numerical shooting and tiling}

For the circular-wave plots in Figure~\ref{fig:circular-waves} we numerically solve the shooting boundary-value problem
for the curvature-jet system \eqref{eq:shooting-cwN} with endpoint conditions
\eqref{eq:IFT-endpoint-cwN} for a small branch parameter $\varepsilon=q(\ell)$, as in
Theorem~\ref{thm:IFT-cwN}.  The resulting half-arc is extended to the shooting cell $[0,2\ell]$ by reflection
(Lemma~\ref{lem:period-extension-cwN}) and then tiled using the translation-per-period relation
\eqref{eq:Gamma-translation-cwN}.  Finally, we rescale by dilation to normalise the stationary parameter to $\lambda=1$
as in Remark~\ref{rem:scale-cwN}, and we report shooting-cell diagnostics (energy and length) as in
Table~\ref{tab:circular-waves}.

\bibliographystyle{alpha}
\bibliography{refs_okabe_revised}

\begin{thebibliography}{AMWW20}

\bibitem[AMWW20]{AMWW20}
Ben Andrews, James McCoy, Glen Wheeler, and Valentina-Mira Wheeler.
\newblock Closed ideal planar curves.
\newblock {\em Geometry \& Topology}, 24(2):1019--1049, September 2020.

\bibitem[AW26]{AW26}
Ben Andrews and Glen Wheeler.
\newblock Jellyfish exist.
\newblock arXiv:2601.21227, January 2026.

\bibitem[HT10]{HararyTal2010EulerSpirals}
Gur Harary and Ayellet Tal.
\newblock {3D} {Euler} spirals for {3D} curve completion.
\newblock In {\em Proceedings of the Twenty-Sixth Annual Symposium on
  Computational Geometry}, SoCG '10, pages 393--402, New York, NY, USA, June
  2010. Association for Computing Machinery.

\bibitem[Lev08]{Levien2008EulerSpiralHistory}
Raph Levien.
\newblock The {Euler} spiral: a mathematical history.
\newblock Technical Report UCB/EECS-2008-111, EECS Department, University of
  California, Berkeley, September 2008.

\bibitem[Lev09]{Levien2009FromSpiralToSpline}
Raphael~Linus Levien.
\newblock {\em From Spiral to Spline: Optimal Techniques in Interactive Curve
  Design}.
\newblock PhD thesis, EECS Department, University of California, Berkeley,
  December 2009.

\bibitem[Mor92]{Moreton1992MVC}
Henry~Packard Moreton.
\newblock {\em Minimum Curvature Variation Curves, Networks, and Surfaces for
  Fair Free-Form Shape Design}.
\newblock PhD thesis, University of California, Berkeley, Berkeley, California,
  1992.
\newblock Also issued as Technical Report UCB/CSD-93-732, March 1993.

\bibitem[MS92]{MoretonSequin1992FairSurface}
Henry~P. Moreton and Carlo~H. S{\'e}quin.
\newblock Functional optimization for fair surface design.
\newblock In {\em Proceedings of the 19th Annual Conference on Computer
  Graphics and Interactive Techniques}, SIGGRAPH '92, pages 167--176, New York,
  NY, USA, July 1992. Association for Computing Machinery.

\bibitem[MW26a]{McW26}
James McCoy and Glen Wheeler.
\newblock On the generalised ideal flow of closed planar curves.
\newblock arXiv:2605.09379, May 2026.

\bibitem[MW26b]{MW26}
Tatsuya Miura and Glen Wheeler.
\newblock Scale-critical curve diffusion flows.
\newblock arXiv:2604.01716, April 2026.

\bibitem[MWW20]{McCoyWheelerWu2020SixthOrderBC}
James McCoy, Glen Wheeler, and Yuhan Wu.
\newblock A sixth order flow of plane curves with boundary conditions.
\newblock {\em Tohoku Mathematical Journal, Second Series}, 72(3):379--393,
  September 2020.

\bibitem[MWW22a]{McCoyWheelerWu2022HighOrderNeumann}
James McCoy, Glen Wheeler, and Yuhan Wu.
\newblock High order curvature flows of plane curves with generalised {Neumann}
  boundary conditions.
\newblock {\em Advances in Calculus of Variations}, 15(3):497--513, 2022.
\newblock First published online 5 February 2021.

\bibitem[MWW22b]{MWW}
James~A. McCoy, Glen~E. Wheeler, and Yuhan Wu.
\newblock A length-constrained ideal curve flow.
\newblock {\em The Quarterly Journal of Mathematics}, 73(2):685--699, June
  2022.
\newblock First published online 15 November 2021.

\bibitem[Ohl85]{O85}
S.~C. Ohlin.
\newblock {2-D and 3-D Curve Interpolation by Consistent Splines}.
\newblock Internal report, IBM Nederland N.V., CAD/CAM Systems Support Group,
  Amsterdam, The Netherlands, March 1985.

\bibitem[Ohl87]{O87}
S.~C. Ohlin.
\newblock Splines for engineers.
\newblock In Guy Mar{\'e}chal, editor, {\em Eurographics '87: Proceedings of
  the European Computer Graphics Conference and Exhibition}, pages 555--565,
  Amsterdam, The Netherlands, August 1987. Eurographics Association,
  North-Holland.

\bibitem[OY26]{OY26}
Shinya Okabe and Hikaru Yamaguchi.
\newblock The ideal flow for planar closed curves with local length constraint.
\newblock {\em Advances in Nonlinear Analysis}, 15(1):20250156, June 2026.

\bibitem[Pal79]{P79}
Richard~S. Palais.
\newblock The principle of symmetric criticality.
\newblock {\em Communications in Mathematical Physics}, 69(1):19--30, October
  1979.

\bibitem[Wu21]{Wu2021ShortTimeExistence}
Yuhan Wu.
\newblock Short time existence for higher order curvature flows with and
  without boundary conditions.
\newblock In David~R. Wood, Jan de~Gier, Cheryl~E. Praeger, and Terence Tao,
  editors, {\em 2019--20 {MATRIX} Annals}, volume~4 of {\em {MATRIX} Book
  Series}, pages 773--783. Springer, Cham, 2021.

\end{thebibliography}

\end{document}